\documentclass[11pt,reqno]{amsart}
\usepackage[T1]{fontenc}
\usepackage{amsmath,amssymb,amsthm,mathrsfs}
\usepackage{longtable,array,needspace}
\usepackage[margin=1.1in]{geometry}
\usepackage[hidelinks]{hyperref}
\usepackage{tikz}

\newtheorem{theorem}{Theorem}[section]
\newtheorem{proposition}[theorem]{Proposition}
\newtheorem{lemma}[theorem]{Lemma}
\newtheorem{corollary}[theorem]{Corollary}
\theoremstyle{definition}
\newtheorem{definition}[theorem]{Definition}
\newtheorem{example}[theorem]{Example}

\newtheorem{hypothesis}[theorem]{Hypothesis}
\theoremstyle{remark}
\newtheorem{remark}[theorem]{Remark}

\newcommand{\CC}{\mathbb C}
\newcommand{\ZZ}{\mathbb Z}
\newcommand{\FT}{\mathcal F}
\newcommand{\mic}{\mathrm{mic}}
\newcommand{\Art}{\mathrm{Art}}

\DeclareMathOperator{\Span}{span}
\DeclareMathOperator{\Ad}{Ad}
\DeclareMathOperator{\IC}{IC}
\DeclareMathOperator{\Irr}{Irr}
\DeclareMathOperator{\Lie}{Lie}
\DeclareMathOperator{\Perv}{Perv}
\DeclareMathOperator{\rank}{rank}
\newcommand{\SL}{\mathrm{SL}}
\newcommand{\Gd}{G^\vee}
\newcommand{\gd}{\mathfrak g^\vee}
\newcommand{\hd}{h^\vee}
\numberwithin{equation}{section}

\title[Microlocal Arthur packets]{Microlocal Arthur packets as a quasi-canonical basis of stable distributions}

\author{Dan Ciubotaru}
 \address[D. Ciubotaru]{Mathematical Institute, University of Oxford, Oxford OX2 6GG, UK}
        \email{dan.ciubotaru@maths.ox.ac.uk}
        
        \author
{Emile Okada}
        \address[E. Okada]{Department of Mathematics, Nanyang Technological University, Singapore 639798}
        \email{emile.okada@ntu.edu.sg}

\date{}

\begin{document}
\begin{abstract}
We show that microlocal Arthur packets can be reconstructed from an otherwise underdetermined system of linear equations arising from Kashiwara’s local index theorem. The key is a geometric rigidity property of Arthur orbits which, together with Fourier symmetry, uniquely determines the microlocal multiplicities defining these packets.
This realizes the associated Arthur distribution as the unique
normalized element in the intersection of two filtered spaces of
standard stable distributions related by Fourier transform.
Under standard enhanced local Langlands hypotheses and Aubert
compatibility, this gives an alternative representation-theoretic
definition of microlocal Arthur packets for general connected reductive
$p$-adic groups that does not use twisted endoscopy and ensures
stability through the stability of tempered packet distributions.
In the unipotent setting, we use this characterization to prove agreement
between endoscopic and microlocal Arthur packets for split $\mathrm{SO}_{2n+1}$
when the residue characteristic is at least $6n+4$.
We further use this characterization to compute all unipotent
Arthur packets with unramified real infinitesimal character for inner-to-split
adjoint exceptional $p$-adic groups.
In this exceptional setting, we prove the microlocal weak Arthur
packet conjecture and give counterexamples
to the microlocal analogues of the least- and greatest-parameter conjectures.
\end{abstract}
\maketitle

\section{Introduction}
\label{sec:introduction}

Arthur's conjectures organize the shape of the discrete automorphic
spectrum of a reductive group in terms of families of representations, {\it Arthur packets}, 
satisfying certain prescribed properties: stability, endoscopic transfer, unitarity \cite{Arthur1989,Arthur2013}. There are two main approaches to the construction of Arthur packets: via endoscopy, in the work of Arthur, and via microlocal geometry on the space of Langlands parameters, as pioneered by Adams--Barbasch--Vogan \cite{ABV} for real reductive groups, and then generalized for groups over local fields by Vogan \cite{Vog}. When comparing these two definitions one has to relate the stability of 
character distributions to the microlocal properties of equivariant perverse sheaves on varieties of Langlands parameters.

This leads one to the natural question of whether the Arthur packets admit a unique characterization. This is the first main result of our paper. We show that the distribution attached to a microlocal Arthur
packet is determined up to scalar multiples by two support conditions related
by Fourier transform inside the span of stable standard-packet distributions. This characterization is inspired and resembles Lusztig's construction of canonical bases. The geometric fact that underlies this interpretation is a rigidity theorem for Arthur orbits in the context of the prehomogeneous vector spaces that appear in the Langlands correspondence for $p$-adic groups.   We then combine this rigidity with Kashiwara's index theorem to obtain a linear algebra argument for constructing the Arthur microlocal
multiplicities just from intersection cohomology information. To exemplify the applications, we then show that in the in Lusztig's unipotent
category, this unique characterization is sufficient to show that the microlocal
and endoscopic Arthur packets for split odd orthogonal groups are the same.
Moreover,
\begin{itemize}
    \item we classify the Arthur parameters with real infinitesimal parameter, \item describe the weak Arthur packets, in the sense of \cite{CMBOweak},  using Lusztig's geometry of special pieces, 
    \item determine explicitly the Arthur packets (with real infinitesimal character) for exceptional $p$-adic groups.  
\end{itemize}
As an interesting corollary of the latter classification, we notice that certain extremality conjectures for Arthur packets can fail for exceptional groups.

\medskip

Let us explain our main geometric setting.
Let $\Gd$ be a connected complex reductive group (the Langlads dual group), with Lie algebra
$\gd$. Fix an $\mathfrak{sl}_2$-triple
$(e^\vee,\hd,f^\vee)$. Set 
$$
 M=Z_{\Gd}(\hd),\qquad
 V=\gd(2)=\{x^\vee\in\gd:[\hd,x^\vee]=2x^\vee\}.
$$
The group $M$ acts with finitely many orbits on $V$. An orbit
$C=Me_1^\vee$ is called an \emph{Arthur orbit} if there are two commuting $\mathfrak{sl}_2$-triples
$(e_i^\vee,h_i^\vee,f_i^\vee)$, $i=1,2$, with
$h_1^\vee+h_2^\vee=\hd$. These triples determine an Arthur pair
$(Me_1^\vee,Me_2^\vee)$. Conormal or Piasetsky duality defines an involution $p$ on $M\backslash V$ via a precise identification of $V$ and $V^*=\mathfrak g^\vee(-2)$. 
On Arthur orbits, this involution exchanges the two factors of an Arthur pair 
\cite[Lemma~3.6.1]{CFMMX}. We order the orbits by closure:
$C\leq D$ means $C\subseteq\overline D$. In general, $p$ is not order-reservising.
Our first main theorem precisely concerns the effect of $p$ on the closure order for an Arthur orbit.

\begin{theorem}[Rigidity of Arthur orbits]
Let $C$ be an Arthur orbit. For every $M$-orbit $D\subset V$,
$$
 C\leq D,\qquad p(C)\leq p(D)
 \quad\Longrightarrow\quad C=D.
$$
\end{theorem}
The proof uses the lengths of neutral elements of homogeneous Lie triples attached to $C$ and $D$.

\medskip

 Write
$a=(C_a,\rho_a)$ for an {\it enhanced orbit}, where $\rho_a$ is an
irreducible representation of the stabilizer component group,
and let
$$
 P_a=\IC(\overline{C_a},\mathcal L_{\rho_a})
$$
be the corresponding simple equivariant perverse sheaf.
For an orbit $C$, denote by $m_C(P_a)\in \mathbb Z_{\ge 0}$ the multiplicity of its
conormal bundle in the characteristic cycle of $P_a$. Under a local Langlands correspondence, the microlocal packet
with support $C$ is
$$
 \Pi_C^\mic=\{\pi_a:m_C(P_a)>0\}.
$$
Kashiwara's local index formula \cite{Kas73,KS} relates these
multiplicities to IC stalk Euler characteristics and local Euler
obstructions. Fourier symmetry imposes additional linear
relations. We show that the resulting reconstruction system
determines every row indexed by an Arthur orbit uniquely.
Consequently, the IC stalk data, the orbit closure order, and the
actions of conormal duality and Fourier transform determine all
Arthur microlocal multiplicities.

The above geometric rigidity has an intrinsic formulation in terms of
packet distributions. For the representation theoretic
interpretation, assume the properties of the Langlands 
correspondence in Hypothesis~\ref{hyp:enhanced-llc}. Let $S_D$
be the Langlands standard packet sum attached to an orbit $D$. Let $T_C$ be the normalized microlocal
virtual character of
\eqref{eq:normalized-microlocal-distribution}. These distributions
are considered across pure inner forms. Set
$$
 U_C=\Span_{\CC}\{S_D:D\geq C\},\qquad
 U_{>C}=\Span_{\CC}\{S_D:D>C\},
$$
and let $\FT$ denote the induced Fourier involution on the
character space of the $p$-adic group and its pure inner forms.

\begin{theorem}[Characterization of Arthur distributions]
For every Arthur orbit $C$,
$$
 U_C\cap\FT^{-1}U_{p(C)}=\CC T_C.
$$
The distribution $T_C$ is the unique element of this intersection
satisfying
$$
 T_C\equiv S_C\pmod{U_{>C}}.
$$
\end{theorem}

The compatibility between the geometric Fourier transform and the Aubert--Zelevinsky involution gives a representation theoretic
interpretation of the two support conditions. More precisely, write $\widehat\pi$
for the irreducible Aubert dual of $\pi$, and $C(\pi)$ for the
orbit of its Langlands parameter. Suppose $\psi: \SL_2\times \SL_2\to\Gd$ is an unramified Arthur parameter with real infinitesimal parameter (equivalently, $\psi$ is trivial on the Weil group).  If $C=C_\psi$ and $\psi^t$
is obtained by exchanging the two $\SL_2$ factors of $\psi$, the
conditions on a constituent $\pi$ of the Arthur packet of $\psi$ become
$$
 C_\psi\leq C(\pi),\qquad
 C_{\psi^t}\leq C(\widehat\pi).
$$
Within the character space spanned by the Langlands standard representations, these bounds together with the 
normalization characterize the Arthur distribution.

\medskip

We verify the hypotheses required for the comparison between endoscopic and microlocal packets for the representations in Lusztig's unipotent category 
when $G$ is the split form of $\mathrm{SO}_{2n+1}(F)$, where the local
field $F$ has characteristic zero, and its residual characteristic
is at least $6n+4$. For every unramified Arthur parameter $\psi$,
we obtain
$$
 \Pi_\psi^\mic(G(F))=\Pi_\psi^\Art(G(F)).
$$
The proof uses Lusztig's unipotent correspondence
\cite{Lus95,Lus02}, Waldspurger's normalization and endoscopic
identities \cite{WaldI,WaldII}, and the Fourier/Aubert-Zelevinsky comparison
of Antor--Okada \cite{AO}. The closure bounds of
Hazeltine--Liu--Lo--Zhang \cite{HLLZ} and packetwise Aubert
duality \cite{Xu2017} show that the endoscopic Arthur distribution
satisfies the uniqueness conditions in our theorem.

\medskip

In order to classify the Arthur pairs we use a centralizer
criterion. Given $e_1^\vee\in V$, choose a homogeneous triple
$(e_1^\vee,h_1^\vee,f_1^\vee)$ and let $\mathfrak s_1$ be the $\mathfrak{sl}_2$-algebra spanned by it. The criterion is that  $Me_1^\vee$ is an Arthur orbit precisely when
$$
 \hd-h_1^\vee
 \quad\text{is a Dynkin element in}\quad
 Z_{\gd}(\mathfrak s_1).
$$
This is independent of the fixed homogeneous triple.
For classical groups, the criterion leads to an explicit condition on the
multiplicity spaces in the defining representation, equivalently
, to a classification by $\SL_2\times\SL_2$-modules. For 
exceptional groups, we combine the criterion with the methods of Vinberg and de Graaf
\cite{Vinberg1987,dGtheta}. We obtain a complete classification
for all $157$ Dynkin gradings in exceptional simple Lie algebras, together with the explicit 
Piasetsky involution. 

A further consequence of our classification concerns the special
pieces. For a grading given by a neutral element, let
$\mathcal O^\vee=\Gd e^\vee$ be the saturation of the open orbit
in $V$. Let $d$ denote Lusztig--Spaltenstein duality. The
\emph{relative special piece} of $\mathcal O^\vee$ consists of
the nilpotent orbits $\mathcal U^\vee\leq\mathcal O^\vee$ for
which $d(\mathcal U^\vee)=d(\mathcal O^\vee)$. Its graded orbits
form the set
$$
 \mathscr C_{\mathrm{rel}}
 =\{C\in M\backslash V:
               d(\Gd C)=d(\mathcal O^\vee)\}.
$$

\begin{theorem}[Graded orbits in relative special pieces]
Every orbit in $\mathscr C_{\mathrm{rel}}$ is Arthur, and the
saturation map
$$
 C\longmapsto\Gd C
$$
is injective on $\mathscr C_{\mathrm{rel}}$. It is also surjective
onto the ambient orbits in the relative special piece in
classical types, and for exceptional types when $\mathcal O^\vee$
is even.
\end{theorem}

In types $B$ and $C$, the lifting and uniqueness
statements follow from Gurevich--Okada \cite{GO}. In the argument, We use the special piece combinatorics used from \cite{JLSYspecial}. The results for exceptional groups 
follow from the complete graded orbit classification.

For unipotent representations in the sense of Lusztig, the wavefront set formula of
\cite{CMBOWFII} connects this special piece geometry to weak Arthur packets.
At the neutral infinitesimal character attached to
$\mathcal O^\vee$, let $\Pi_{\mathcal O^\vee}^{\mathrm{weak}}$
consist of the representations whose geometric wavefront set is
$\{d(\mathcal O^\vee)\}$. Under the correspondence and compatibility
hypotheses of Section~\ref{sec:weak}, for groups that are inner to split
(and sufficiently large residual
characteristic), we prove
$$
 \Pi_{\mathcal O^\vee}^{\mathrm{weak}}
   =\bigcup_{C\in\mathscr C_{\mathrm{rel}}}\Pi_{p(C)}^\mic.
$$
Every packet in this union is  a microlocal Arthur
packet. When $\mathcal O^\vee$ is even, we obtain an explicit
parameterization of the individual representations by ambient
nilpotent orbits in its special piece and irreducible
representations of their stabilizer component groups.
These results relate the weak packet decompositions of
\cite{LiuLoWeak,GO} to the geometry of special pieces
\cite{AcharSage,JLSYspecial}.

\medskip

The linear reconstruction theorem can be used to explicitly determine the microlocal Arthur packets for the exceptional groups.
Using the IC polynomials computed in \cite{CiubotaruOkadaKL},
we give explicit orbit, local-system, and packet tables.
These computations also reveal counterexamples to certain conjectured 
extremality properties of Arthur packets. The microlocal analogue of
\cite[Conjecture~1.5]{HLLZ} would require the Arthur parameters
whose packets contain a fixed representation to have a least
and a greatest element in the closure order. For the
$F_4(a_3)$ grading, we exhibit a representation for which this
set has two minimal elements. For the $E_8(a_7)$ grading,
we exhibit one for which it has two maximal elements.

\

The paper is organized as follows.
Sections~\ref{sec:setup} and~\ref{sec:arthur-infinitesimal}
establish the parameter geometry and its reduction to graded
Lie algebras. Section~\ref{sec:conormal-bounds} proves rigidity,
and Section~\ref{sec:reconstruction} develops reconstruction and
the characterization of packet distributions.
Section~\ref{sec:comparison} treats the unipotent correspondence
and the odd orthogonal comparison.
Section~\ref{sec:classification} classifies Arthur pairs.
Section~\ref{sec:weak} studies special pieces and weak packets,
and Section~\ref{sec:examples} contains the explicit computations
and extremality counterexamples.

\section{Parameters, graded geometry, and microlocal packets}
\label{sec:setup}

\subsection{Preliminaries}
\label{subsec:groups}
Let $F$ be a nonarchimedean local field of characteristic zero, with
residue field of cardinality $q$, Weil group $W_F$, and inertia subgroup
$I_F$. Write $|\cdot|:W_F\longrightarrow q^{\ZZ}$ for the norm character.
Choose $w_q\in W_F$ with $|w_q|=q$.

Let $G$ be a quasi-split connected reductive group over $F$, and let
$\Gd$ be its complex Langlands dual group, with Lie algebra $\gd$.
Write ${}^LG=\Gd\rtimes W_F$ for its Weil-form $L$-group,
with the Weil action preserving a chosen pinning.
We retain the isogeny type of $\Gd$:
although a central isogeny preserves the adjoint orbit geometry, it can
change the equivariant local systems and their representation-theoretic
interpretation. We denote by $G(F)$ the group of $F$-points and by
$G^\delta(F)$ that of a pure inner twist $\delta\in H^1(F,G)$;
see \cite[Definition~2.6 and Proposition~2.7(c)]{Vog}.

The superscript $\vee$ marks the complex dual side: for example,
$e^\vee,h^\vee,f^\vee\in\gd$ and
$\mathcal O^\vee\subset\gd$.
Once the ambient graded space $V$ is fixed, we write $C\subset V$ for
its orbits, omitting the superscript on $C$ and $V$. For a vector space
$V$, we will use $V^*$ to mean the linear dual.

Write
\begin{equation}
 E=\begin{pmatrix}0&1\\0&0\end{pmatrix},\qquad
 H_0=\begin{pmatrix}1&0\\0&-1\end{pmatrix},\qquad
 F=\begin{pmatrix}0&0\\1&0\end{pmatrix}
\end{equation}
for the standard $\mathfrak{sl}_2$-triple.

In the $\SL_2$ formulation of the Langlands parameters, write
\begin{equation}
 \phi:W_F\times\SL_2(\CC)\longrightarrow{}^LG
\end{equation}
for a Langlands homomorphism: continuous on $W_F$, algebraic on
$\SL_2$, having semisimple Weil image, and with the standard
admissibility conditions. Its composition with ${}^LG\to W_F$ is
the projection. Parameters are considered up to $\Gd$-conjugacy;
we use actual homomorphisms when fixing an infinitesimal parameter.
Vogan uses the equivalent Weil--Deligne formulation in
\cite[Definition~4.2]{Vog}. Relevance to a particular inner form is
imposed when passing to representations of that form. Put
\begin{equation}
 a_w=\begin{pmatrix}|w|^{1/2}&0\\0&|w|^{-1/2}\end{pmatrix}.
\end{equation}
By definition, the infinitesimal parameter associated to $\phi$ is
\begin{equation}\label{eq:infinitesimal-parameter}
 \lambda_\phi(w)=\phi(w,a_w).
\end{equation}
Following \cite[(4.4)(a)]{Vog},
an infinitesimal parameter is a continuous homomorphism
$\lambda:W_F\longrightarrow{}^LG$ which is a section of
${}^LG\to W_F$ and has semisimple image. Its $\Gd$-conjugacy class
is the infinitesimal character in the sense of
\cite[Definition~1.10]{Vog}.

Fix an infinitesimal parameter $\lambda$ for the remainder of this section.

\subsection{The parameter variety at an infinitesimal parameter}
\label{subsec:parameter-variety}

Parameter varieties have their origins in Lusztig's study of square-integrable
representations \cite{Lusztig1983}. We use Vogan's
formulation at an infinitesimal parameter \cite[(4.4)(c)--(f)]{Vog}.
Set
\begin{equation}
 H_\lambda=Z_{\Gd}(\lambda(W_F)),\qquad
 K_\lambda=Z_{\Gd}(\lambda(I_F)),\qquad
 \mathfrak k_\lambda=\Lie(K_\lambda).
\end{equation}
Put
\begin{equation}\label{eq:parameter-variety}
 V_\lambda^\pm
 =\{x^\vee\in\gd:\Ad(\lambda(w))x^\vee=|w|^{\pm1}x^\vee
                         \text{ for all }w\in W_F\}.
\end{equation}
Write $V_\lambda=V_\lambda^+$. When $\lambda$ is clear from context,
we drop the subscript and write $V^\pm=V_\lambda^\pm$ and
$V=V_\lambda$; we also write $H=H_\lambda$.
Thus
\begin{equation}
 V_\lambda^\pm
 =\{x^\vee\in\mathfrak k_\lambda:
                  \Ad(\lambda(w_q))x^\vee=q^{\pm1}x^\vee\}.
\end{equation}

We recall some standard facts about these objects. The identity components of $H_\lambda$ and $K_\lambda$ are reductive.
The group $H_\lambda$ acts on $V_\lambda$ by conjugation, with finitely
many orbits and a unique open orbit. 
The assignment
\begin{equation}
 \phi\longmapsto x_\phi
 :=d(\phi|_{\SL_2})(E)
\end{equation}
induces a bijection between the $H_\lambda$-conjugacy classes of
parameters satisfying $\lambda_\phi=\lambda$ and the $H_\lambda$-orbits
in $V_\lambda$. We write $C_\phi=H_\lambda x_\phi$.

\subsection{Arthur pairs}
\label{subsec:arthur-pairs}

An Arthur parameter \cite[\S6]{Arthur1989} is an admissible homomorphism
\begin{equation}
 \psi:W_F\times\SL_2(\CC)_D\times\SL_2(\CC)_A
       \longrightarrow{}^LG
\end{equation}
whose Weil part is bounded and whose restrictions to the two
$\SL_2$ factors are algebraic. 
Define
\begin{equation}\label{eq:arthur-langlands-transposition}
 \phi_\psi(w,x)=\psi(w,x,a_w),\qquad
 \psi^t(w,x,y)=\psi(w,y,x).
\end{equation}

For an Arthur parameter $\psi$, the infinitesimal parameters satisfy
\begin{equation}
 \lambda_\psi(w):=\lambda_{\phi_\psi}(w)
   =\psi(w,a_w,a_w)=\lambda_{\psi^t}(w).
\end{equation}
For $\lambda_\psi=\lambda$, the two factors define commuting triples
$(e_i^\vee,h_i^\vee,f_i^\vee)$, $i=1,2$, with
\begin{equation}
 e_i^\vee\in V_\lambda,\qquad
 h_i^\vee\in\Lie(H_\lambda),\qquad
 f_i^\vee\in V_\lambda^-.
\end{equation}
In particular,
\begin{equation}\label{eq:arthur-base}
 C_\psi:=C_{\phi_\psi}=H e_1^\vee,\qquad
 C_{\psi^t}=H e_2^\vee.
\end{equation}

\begin{definition}\label{def:arthur-pair}
An ordered pair $(C_1,C_2)$ of $H$-orbits in $V$ is an
\emph{Arthur pair at $\lambda$} if
$(C_1,C_2)=(C_\psi,C_{\psi^t})$ for an Arthur parameter
$\psi$ with $\lambda_\psi=\lambda$.
An orbit occurring as the first member is an \emph{Arthur orbit}.
\end{definition}

\subsection{Orbits, conormal bundles, and Piasetski duality}
\label{subsec:conormals}

For $H$-orbits $C,D\subset V$, set
\begin{equation}\label{eq:closure-order}
 C\leq D\quad\Longleftrightarrow\quad C\subseteq\overline D.
\end{equation}
Write $\mathcal O_C^\vee=\Gd\cdot C$ for the ambient nilpotent orbit.

Let $B$ be the Killing form on $[\gd,\gd]$. It is invariant
under $\Gd$ and the pinned Weil action
\cite[Proposition~1.119]{Knapp}, and is positive definite on
the real coroot span.
It induces a perfect
$H$-invariant pairing between $V_\lambda$ and $V_\lambda^-$.
Using $B$, we identify $V^-$ with $V^*$ and $T^*V$ with $V\times V^-$ by
\begin{equation}
 V^-\xrightarrow{\sim}V^*,\quad y^\vee\longmapsto B(-,y^\vee),
 \qquad (x,B(-,y))\in T^*V\longmapsto(x,y)\in V\times V^-.
\end{equation}

As in \cite[\S6.3]{CFMMX}, for $x^\vee\in C$, the conormal fibre is
\begin{equation}\label{eq:conormal-fibre}
 (T_C^*V)_{x^\vee}
 =\{y^\vee\in V^-:[x^\vee,y^\vee]=0\}.
\end{equation}

We have $T_{x^\vee}C=[\Lie(H),x^\vee]$ and
$[x^\vee,y^\vee]\in\Lie(H)$; invariance and nondegeneracy of the
form give the assertion. More explicitly, for $a\in\Lie(H)$,
\begin{equation}
 B([a,x^\vee],y^\vee)=B(a,[x^\vee,y^\vee]).
\end{equation}

Put $\Lambda_C=\overline{T_C^*V}$ and
\begin{equation}
 \Lambda=\{(x^\vee,y^\vee)\in V\times V^-:
                  [x^\vee,y^\vee]=0\}
         =\bigcup_{C\in H\backslash V}\Lambda_C.
\end{equation}
If $H$ is disconnected, $[\Lambda_C]$ denotes the sum of the
fundamental cycles of its irreducible components.

\begin{definition}\label{def:piasetski}
Following \cite{Piasetski1975}, the \emph{Piasetski dual} of $C$
is the unique dense $H$-orbit $C^*\subset V^-$ in the second
projection of $\Lambda_C$.
\end{definition}

By \cite[Lemma~6.4.1]{CFMMX}, the cotangent flip $(x^\vee,y^\vee)\mapsto(y^\vee,-x^\vee)$
identifies $\Lambda_C$ with $\Lambda_{C^*}$. Thus duality gives a
bijection between the orbit sets in $V$ and $V^-$, also when $H$
is disconnected.

\subsection{Enhancements and microlocal packet data}
\label{subsec:microlocal-data}
Set
\begin{equation}
 A_H(x^\vee)=\pi_0 Z_H(x^\vee).
\end{equation}
The irreducible $H$-equivariant local systems on $C$ correspond to
$\rho\in\Irr A_H(x^\vee)$. Write $\mathcal L_\rho$ for the
local system and use the perverse normalization
\begin{equation}\label{eq:ic-normalization}
 P_{C,\rho}=\IC(\overline C,\mathcal L_\rho),\qquad
 P_{C,\rho}|_C=\mathcal L_\rho[\dim C].
\end{equation}
These are the simple objects of the category $\Perv_H(V)$ of $H$-equivariant perverse sheaves on $V$ \cite[Theorem~8.6]{Vog}.

For $P\in\Perv_H(V)$ its characteristic cycle has the form
\begin{equation}\label{eq:characteristic-cycle}
 \operatorname{CC}(P)=\sum_{C\in H\backslash V}
                  m_C(P)[\Lambda_C],\qquad m_C(P)\in\ZZ_{\geq0}.
\end{equation}
In particular, 
\begin{equation}\label{eq:microlocal-diagonal}
 m_C(P_{C,\rho})=\rank\mathcal L_\rho=\dim\rho,
 \qquad m_C(P_{D,\sigma})=0\quad\text{unless }C\leq D.
\end{equation}
The coefficient $m_C(P)$  is the dimension of a generic normal Morse group, or
equivalently the generic rank of normalized microlocal vanishing cycles
along the conormal component. These are notoriously difficult to compute, but we will only need their formal properties.  
At the sheaf level define
\begin{equation}\label{eq:sheaf-packet}
 \mathscr P_C^{\mic}
 =\{P\in\Irr\Perv_H(V):m_C(P)>0\}.
\end{equation}
Write
\begin{equation}
 \mathcal I_\lambda
 =\{(C,\rho):C\in H\backslash V,\
        \rho\in\Irr A_H(x^\vee),\ x^\vee\in C\},
\end{equation}
where a representative $x^\vee$ is chosen in each orbit.
For $a=(C_a,\rho_a)\in\mathcal I_\lambda$, write
$P_a=P_{C_a,\rho_a}$ and $d_a=\dim C_a$.

An enhanced Langlands parameter is a pair $(\phi,\rho)$ with
\begin{equation}
 \rho\in\Irr A_{H_{\lambda_\phi}}(x_\phi),\qquad
 x_\phi=d(\phi|_{\SL_2})(E),
\end{equation}
considered up to $\Gd$-conjugacy. This agrees with the usual enhancement
by the component group of the parameter centralizer, via the canonical
isomorphism of \cite[Lemma~4.6.1]{CFMMX}.
Write $\Phi^{\mathrm{enh}}(G)$ for the enhanced parameter classes.
Those with infinitesimal character $[\lambda]$ correspond bijectively
to $\mathcal I_\lambda$ via $[\phi,\rho]\mapsto(C_\phi,\rho)$,
using representatives with $\lambda_\phi=\lambda$.

Fix a Whittaker datum for $G$. Given an enhanced local Langlands correspondence
\begin{equation}
 \bigsqcup_{\delta\in H^1(F,G)}\Irr(G^\delta(F))
 \longleftrightarrow\Phi^{\mathrm{enh}}(G),
 \qquad \pi\longleftrightarrow[\phi_\pi,\rho_\pi],
\end{equation}
write $\pi(\phi,\rho)$ for its inverse and
$\Pi_\phi(G^\delta(F))=\{\pi\in\Irr(G^\delta(F)):[\phi_\pi]=[\phi]\}$.
For $a=(C_a,\rho_a)\in\mathcal I_\lambda$, denote the corresponding
representation by $\pi_a$, its pure inner twist by $\delta_a$, and
its Langlands standard module by $I_a$. Put
\begin{equation}\label{eq:signed-restriction-multiplicities}
 s_{ba}=\sum_k(-1)^k
 \bigl[\mathcal H^k(P_a|_{C_b}):\mathcal L_{\rho_b}\bigr],
\end{equation}
where the bracket denotes multiplicity of an $H$-equivariant local system.
Let $\kappa_\delta$ be the Kottwitz character of
$\pi_0(Z(\Gd)^{W_F})$ attached to $\delta$; see
\cite[Proposition~4.16 and (4.17)]{Vog}.

Throughout the paper, assume a bijection of the above form for $G$
and each standard Levi subgroup, normalized by the chosen Whittaker
datum and its restrictions. 
For $G$ and each standard Levi subgroup, require the following
properties.
\begin{hypothesis}[Enhanced local Langlands correspondence]
\label{hyp:enhanced-llc}
$\hphantom{ }$
\begin{enumerate}
\item \emph{Inner forms.} For $\pi\in\Irr(G^\delta(F))$,
\begin{equation}
 \rho_\pi|_{\pi_0(Z(\Gd)^{W_F})}
 =\kappa_\delta\,\mathrm{id}_{\rho_\pi},
\end{equation}
where restriction uses the natural map to
$A_{H_{\lambda_{\phi_\pi}}}(x_{\phi_\pi})$.

\item \emph{Kazhdan--Lusztig multiplicities.}
For $a,b\in\mathcal I_\lambda$ with $\delta_a=\delta_b$,
\begin{equation}\label{eq:llc-kl-hypothesis}
 [I_b:\pi_a]=(-1)^{d_a}s_{ba}.
\end{equation}

\item \emph{Temperedness and stability.}
The representation $\pi$ is tempered if and only if
$\phi_\pi|_{W_F}$ is bounded. For every tempered parameter $\phi$ and pure inner form $G^\delta$,
\begin{equation}\label{eq:tempered-packet-stability}
 S\Theta_\phi^\delta=\sum_{\pi\in\Pi_\phi(G^\delta(F))}
                  \dim(\rho_\pi)\Theta_\pi
\end{equation}
is stable on $G^\delta(F)$.

\item \emph{Langlands classification.}
Let $\phi=\iota_L\circ(\phi_L\otimes z)$ have Langlands data
$(L,\phi_L,z)$: $\phi_L$ is tempered and $z$ is a positive real
unramified central Weil parameter, strictly positive for the standard
parabolic $P$ with Levi $L$.
Here $\iota_L:{}^LL\hookrightarrow{}^LG$ is the standard embedding.
For each pure inner form $G^\delta$ to which $\phi$ is relevant,
let $P_\delta$ have Levi $L_\delta$, the corresponding pure inner form
of $L$. Require
\begin{equation}\label{eq:llc-langlands-classification}
 \pi_{G^\delta}(\phi,\rho)
 =J_{P_\delta}^{G^\delta}
       \bigl(\pi_{L_\delta}(\phi_L,\rho_L)\otimes\nu_\delta\bigr),
\end{equation}
where $\nu_\delta$ is the character of $L_\delta(F)$ corresponding to
$z$ under the unramified character LLC, and $J$ denotes the unique
irreducible quotient of normalized
parabolic induction. The canonical identification of
parameter component groups identifies the enhancements on
$G^\delta$ with those on $L_\delta$.
\end{enumerate}
\end{hypothesis}
Under Hypothesis~\ref{hyp:enhanced-llc}, the infinitesimal character
of $\pi$ is $[\lambda_{\phi_\pi}]$. Write
$\Irr_\lambda(G^\delta(F))$ for the representations with character
$[\lambda]$. Restricting the correspondence gives
\begin{equation}
 \bigsqcup_{\delta\in H^1(F,G)}\Irr_\lambda(G^\delta(F))
 \longleftrightarrow\mathcal I_\lambda.
\end{equation}

The sheaf packet \eqref{eq:sheaf-packet} corresponds to the
\emph{microlocal packet}
\begin{equation}
 \Pi_C^{\mic}=\{\pi_a:a\in \mathcal I_\lambda, \ m_C(P_a)>0\},\qquad
 \Pi_C^{\mic}(G^\delta(F))
 =\Pi_C^{\mic}\cap\Irr_\lambda(G^\delta(F)).
\end{equation}
For an Arthur parameter $\psi$ at $\lambda$, put
\begin{equation}\label{eq:microlocal-packet}
 \Pi_\psi^{\mic}=\Pi_{C_\psi}^{\mic};
\end{equation}
see \cite[\S8, discussion of (8.16)]{Vog}.

\begin{definition}[Normalized microlocal virtual character]
\label{def:normalized-microlocal-character}
For an $H$-orbit $C\subset V$, define
\begin{equation}\label{eq:normalized-microlocal-distribution}
 T_C=\sum_{a\in\mathcal I_\lambda}
 e(\delta_a)(-1)^{d_a-\dim C}m_C(P_a)\Theta_{\pi_a},
\end{equation}
where $e(\delta_a)\in\{1,-1\}$ is the Kottwitz sign.
We use the normalization of \cite[(8.16)(b)]{Vog} and regard $T_C$
as an element of the direct sum of character spaces over pure inner forms.
\end{definition}

\subsection{Fourier transform}
\label{subsec:general-fourier}

We normalize Fourier transform so that it is exact for the
perverse $t$-structure. Fourier transform is then an equivalence
\begin{equation}
 \mathfrak F_{V_\lambda}:\Perv_{H_\lambda}(V_\lambda)
       \longrightarrow\Perv_{H_\lambda}(V_\lambda^-).
\end{equation}

Fourier transform permutes the
simple objects and satisfies
\begin{equation}
 m_{C^*}(\mathfrak F_{V_\lambda}P)=m_C(P).
\end{equation}

\section{Infinitesimal parameters of Arthur type}
\label{sec:arthur-infinitesimal}

\subsection{The bounded Weil part}
\label{subsec:arthur-type-specialization}

An infinitesimal parameter $\lambda$ is of \emph{Arthur type} if
$\lambda=\lambda_{\psi_0}$ for an Arthur parameter $\psi_0$.
Choose such a $\psi_0$ and define
\begin{equation}
 \tau(w)=\psi_0(w,1,1),\qquad
 \varphi_\Delta(g)=\psi_0(1,g,g),\qquad
 (e^\vee,\hd,f^\vee)=d\varphi_\Delta(E,H_0,F).
\end{equation}
For $t>0$, write
\begin{equation}
 t^{\hd/2}=\varphi_\Delta
       \begin{pmatrix}t^{1/2}&0\\0&t^{-1/2}\end{pmatrix}.
\end{equation}

Throughout this section, $\lambda$ is of Arthur type.

We use the elliptic--hyperbolic decomposition of
\cite[\S\S5.2--5.3 and proof of Proposition~6.6.1]{CFMMX}.
The homomorphism $\tau$ is bounded and centralizes $\varphi_\Delta$.
Moreover,
\begin{equation}\label{eq:diagonal-infinitesimal}
 \lambda(w)=\tau(w)|w|^{\hd/2}.
\end{equation}
The pair $(\tau,\hd)$ depends only on the fixed representative $\lambda$.
In particular, every Arthur parameter $\psi$ with $\lambda_\psi=\lambda$
has $\psi(w,1,1)=\tau(w)$.

The product decomposition follows from the definition of $\lambda_\psi$.
In the elliptic--hyperbolic decomposition of $\lambda(w_q)$,
\begin{equation}
 \lambda(w_q)_{\mathrm{ell}}=\tau(w_q),\qquad
 \lambda(w_q)_{\mathrm{hyp}}=q^{\hd/2},
\end{equation}
and then
$\tau(w)=\lambda(w)|w|^{-\hd/2}$ for every $w$.

Set $\phi_\Delta(w,g)=\tau(w)\varphi_\Delta(g)$.
The parameter $\phi_\Delta$ is tempered, with
$\lambda_{\phi_\Delta}=\lambda$ and $\phi_\Delta|_{W_F}=\tau$.

\subsection{Inertia, Frobenius, and the bounded centralizer}
\label{subsec:bounded-centralizer}
Following \cite[\S\S5.2--5.3]{CFMMX}, set
\begin{equation}
 J=Z_{\Gd}(\tau(W_F)),\qquad \mathfrak j=\Lie(J).
\end{equation}
Since $\lambda|_{I_F}=\tau|_{I_F}$ and $W_F$ is generated by
$I_F$ and $w_q$,
\begin{equation}
 J=K_\lambda^{\Ad(\tau(w_q))}.
\end{equation}
The group $J$ may be disconnected, but $J^\circ$ is reductive:
the Zariski closure of $\Ad(\tau(W_F))$ is linearly reductive,
since $\tau$ is bounded, so \cite[A.8.12]{CGP} applies.
The image of $\varphi_\Delta$ is connected and centralizes $\tau(W_F)$,
hence lies in $J^\circ$. Thus $(e^\vee,\hd,f^\vee)$ is an
$\mathfrak{sl}_2$-triple in $\mathfrak j$.

Define
\begin{equation}
 \mathfrak j(r)=\{x^\vee\in\mathfrak j:[\hd,x^\vee]=r x^\vee\}.
\end{equation}
By \cite[Lemma~5.3.1 and Remark~5.3.2]{CFMMX},
\begin{equation}
 \begin{gathered}
 H=H_\lambda=Z_J(\hd),\qquad H^\circ=Z_{J^\circ}(\hd),\\
 \Lie(H)=\mathfrak j(0),\qquad V_\lambda^\pm=\mathfrak j(\pm2).
 \end{gathered}
\end{equation}
The orbit $H^\circ e^\vee=He^\vee$ is open in $V$,
since $[\mathfrak j(0),e^\vee]=\mathfrak j(2)$.
An $\mathfrak{sl}_2$-triple in $\mathfrak j$ is \emph{homogeneous}
if its entries lie in $\mathfrak j(2),\mathfrak j(0),\mathfrak j(-2)$,
respectively; we also allow the zero triple.

Set
\begin{equation}
 n_\Delta=\varphi_\Delta
 \begin{pmatrix}0&1\\-1&0\end{pmatrix}\in J^\circ,
 \qquad \alpha=\operatorname{Int}(n_\Delta)|_H.
\end{equation}
The identity $\Ad(n_\Delta)\hd=-\hd$ shows that $n_\Delta$
normalizes $H$ and exchanges $V^+$ and $V^-$. Moreover,
$n_\Delta^2=\varphi_\Delta(-I)\in Z(H)$ acts trivially on
$V^+\oplus V^-$. Hence $\alpha^2=1$ and the maps
\begin{equation}
 \iota_\Delta^\pm=\Ad(n_\Delta)|_{V^\pm}:
 V^\pm\xrightarrow{\sim}V^\mp
\end{equation}
are inverse and satisfy
\begin{equation}
 \iota_\Delta^\pm(\Ad(m)v)
 =\Ad(\alpha(m))\iota_\Delta^\pm(v)
 \qquad(m\in H,\ v\in V^\pm).
\end{equation}

\subsection{Saturation of graded orbits}
\label{sec:rigidity}
The grading is on $\mathfrak j$, while
$\mathcal O_C^\vee=\Gd\cdot C$ still denotes saturation in the
original dual Lie algebra $\gd$.

\begin{lemma}\label{lem:strict-saturation}
If $C<D$ are $H$-orbits in $V=\mathfrak j(2)$, then
$\mathcal O_C^\vee<\mathcal O_D^\vee$.
\end{lemma}
\begin{proof}
Since $J^\circ$ is reductive, complete reducibility gives a
$J$-stable complement $\gd=\mathfrak j\oplus U$
\cite[Corollary~22.43]{MilneAG}. For $x^\vee\in V$, we have
$[U,x^\vee]\subseteq U$ and
$[\mathfrak j(r),x^\vee]\subseteq\mathfrak j(r+2)$. Hence
\begin{equation}
 [\gd,x^\vee]\cap V
 =[\mathfrak j,x^\vee]\cap\mathfrak j(2)
 =[\mathfrak j(0),x^\vee]
 =[\Lie(H),x^\vee].
\end{equation}
By \cite[\S2.3, Lemma]{Vinberg1976}, each $H^\circ$-orbit is closed
in its ambient orbit intersected with $V$. The same holds for an
$H$-orbit, a finite union of $H^\circ$-orbits. Thus, if
$\mathcal O_C^\vee=\mathcal O_D^\vee=\mathcal O$, then
\begin{equation}
 C\subseteq\overline D\cap\mathcal O\cap V=D,
\end{equation}
contrary to $C\ne D$. Since saturation preserves closure inclusion,
$\mathcal O_C^\vee<\mathcal O_D^\vee$.
\end{proof}

\subsection{Reduction to unipotent geometry}
\label{subsec:unipotent-reduction}
Removing the bounded Weil part $\tau$ reduces the geometry at $\lambda$
to that of an unramified hyperbolic parameter for $J\times W_F$. Set
\begin{equation}
 {}^LJ_{\mathrm{geom}}=J\times W_F,\qquad
 \lambda_J(w)=(|w|^{\hd/2},w),
\end{equation}
and define
\begin{equation}
 r_\tau:{}^LJ_{\mathrm{geom}}\longrightarrow{}^LG,
 \qquad (j,w)\longmapsto j\tau(w).
\end{equation}
Since $\tau(W_F)$ centralizes $J$, this is an injective homomorphism
with $r_\tau\circ\lambda_J=\lambda$. We do not assume that
${}^LJ_{\mathrm{geom}}$ is the $L$-group of a reductive group over $F$.

The following is the reduction of
\cite[Remark~5.3.2 and proof of Proposition~6.6.1]{CFMMX}.
\begin{proposition}[Unipotent reduction]
\label{prop:unipotent-reduction}
The parameter $\lambda_J$ has centralizer $H$ in $J$ and weight spaces
$V^\pm=\mathfrak j(\pm2)$. Its equivariant geometry is therefore the
same as that of $\lambda$: the acting group, stabilizers, enhanced orbit
labels, conormal closures, intersection cohomology complexes,
microlocal multiplicities, and Fourier transform are unchanged.

Composition with $r_\tau$ gives an $H$-equivariant bijection
\begin{equation}
 \{\psi_J:\lambda_{\psi_J}=\lambda_J,\
                 \psi_J(w,1,1)=(1,w)\}
 \xrightarrow{\sim}
 \{\psi:\lambda_\psi=\lambda\}
\end{equation}
between Arthur parameters for $J\times W_F$ and ${}^LG$,
preserving transposition and Arthur pairs. Explicitly, it sends
\begin{equation}
 \psi_J(w,g_1,g_2)=(\Psi(g_1,g_2),w)
 \quad\longmapsto\quad
 \psi(w,g_1,g_2)=\tau(w)\Psi(g_1,g_2).
\end{equation}

The reduction also identifies enhanced Langlands parameter classes
at $\lambda_J$ and $\lambda$ via their common labels
$(C,\rho)\in\mathcal I_\lambda$.
\end{proposition}

Explicitly, for $x^\vee\in C$, choose an algebraic homomorphism
$\gamma:\SL_2\to J^\circ$ such that $d\gamma(E)=x^\vee$ and
$d\gamma(E,H_0,F)$ is homogeneous; such a choice exists by
\cite[Theorem~1]{Vinberg1979}. Corresponding parameters are
\begin{equation}
 \phi_{J,x}(w,g)=
 (|w|^{\hd/2}\gamma(a_w)^{-1}\gamma(g),w),\qquad
 \phi_x=r_\tau\circ\phi_{J,x},
\end{equation}
with the same enhancement group $\pi_0Z_H(x^\vee)$.

\subsection{The diagonal fibre and commuting triples}
\label{subsec:diagonal-fibre}
Let $\Delta(g)=(g,g)$. The diagonal fibre is
\begin{equation}\label{eq:diagonal-fibre}
 \mathscr A_{\tau,\varphi_\Delta}
 =\{\Psi\in\operatorname{Hom}_{\mathrm{alg}}
 (\SL_2\times\SL_2,J^\circ):
 \Psi\circ\Delta\sim_J\varphi_\Delta\}/J.
\end{equation}
By Proposition~\ref{prop:unipotent-reduction}, it parametrizes
$H$-conjugacy classes of Arthur parameters at $\lambda$.
Equivalently, an Arthur pair is given by two commuting homogeneous
triples $(e_i^\vee,h_i^\vee,f_i^\vee)$ satisfying
\begin{equation}
 h_1^\vee+h_2^\vee=\hd,\qquad C_i=He_i^\vee.
\end{equation}
Indeed, the corresponding algebraic homomorphisms
$\gamma_i:\SL_2\to J^\circ$ define
\begin{equation}
 \psi(w,g_1,g_2)=\tau(w)\gamma_1(g_1)\gamma_2(g_2),
\end{equation}
and the equality of neutral elements gives $\lambda_\psi=\lambda$.
Conversely, the two $\SL_2$-factors of $\psi$ give these triples.

Representatives may be chosen so that, in addition,
\begin{equation}\label{eq:triple-sums}
 e_1^\vee+e_2^\vee=e^\vee,\qquad
 f_1^\vee+f_2^\vee=f^\vee.
\end{equation}
Thus we may fix
$\Psi\circ\Delta=\varphi_\Delta$, leaving conjugation by
$Z_J(\varphi_\Delta)$:
\begin{equation}
 \mathscr A_{\tau,\varphi_\Delta}
 \simeq
 \{\Psi\in\operatorname{Hom}_{\mathrm{alg}}(\SL_2\times\SL_2,J^\circ):
 \Psi\circ\Delta=\varphi_\Delta\}/Z_J(\varphi_\Delta).
\end{equation}

\subsection{Graded duality and Arthur partners}
\label{subsec:weyl-arthur}
Recall that $H=Z_J(\hd)$ and $V^\pm=\mathfrak j(\pm2)$.
The map $\iota_\Delta^-=\Ad(n_\Delta)|_{V^-}$ identifies $V^-$ with $V$.
Composing Piasetski duality with $\iota_\Delta^-$ defines
\begin{equation}\label{eq:partner-map}
 p:H\backslash V\longrightarrow H\backslash V,\qquad
 p(C)=\iota_\Delta^-(C^*).
\end{equation}
Since $\Ad(n_\Delta)$ preserves $B$, the cotangent lift is
\begin{equation}
 V^-\times V^+\longrightarrow V^+\times V^-,\qquad
 (y,x)\longmapsto(\Ad(n_\Delta)y,\Ad(n_\Delta)x).
\end{equation}
It carries conormal closures to conormal closures. Since both the
cotangent flip and Weyl transport square to the identity on orbit sets
and commute there, $p^2=1$.

\begin{lemma}[Arthur partners]\label{lem:cunningham}
For an Arthur parameter $\psi$ at $\lambda$ with triples
$(e_i^\vee,h_i^\vee,f_i^\vee)$,
\begin{equation}
 C_\psi^*=Hf_2^\vee,\qquad p(C_\psi)=C_{\psi^t}.
\end{equation}
\end{lemma}
\begin{proof}
The first equality follows from
\cite[Propositions~6.6.1 and 6.5.1, Lemma~6.4.2]{CFMMX}.
Choose representatives satisfying \eqref{eq:triple-sums}. Then
$\Ad(n_\Delta)f_2^\vee=-e_2^\vee$, so
$p(C_\psi)=H(-e_2^\vee)=He_2^\vee=C_{\psi^t}$;
negation preserves $H$-orbits by the grading torus action.
\end{proof}

\subsection{Fourier transform and full equivariance}
Transport along $\iota_\Delta^-$ gives the equivalence
\begin{equation}
 \FT=(\iota_\Delta^-,\alpha)_*\circ\mathfrak F_V:
 \Perv_H(V)\longrightarrow\Perv_H(V).
\end{equation}
Here $(\iota_\Delta^-,\alpha)_*$ denotes direct image with the
$H$-action transported by $\alpha$: $m\in H$ acts on the target
through the action of $\alpha^{-1}(m)$ on the source.

The identification $\iota_\Delta^-$ makes $\FT$ the Fourier transform
for the pairing
\begin{equation}
 Q(x,z)=B(x,\Ad(n_\Delta)z)\qquad(x,z\in V).
\end{equation}
This pairing is nondegenerate and symmetric, since $\Ad(n_\Delta)$
preserves $B$ and squares to the identity on $V^+\oplus V^-$.
Fourier inversion therefore gives $\FT^2\simeq(-1_V)^*$;
the equivariant structures agree because $\alpha^2=1$.
The element $c=\varphi_\Delta(\operatorname{diag}(i,-i))\in Z(H)$
acts on $V=\mathfrak j(2)$ by $i^2=-1$. The $H$-equivariant
structure of each sheaf identifies its pullback by this action with
itself. Hence
\begin{equation}
 \FT^2\simeq(-1_V)^*\simeq\mathrm{id}.
\end{equation}

On characteristic cycles, $\mathfrak F_V$ sends $[\Lambda_C]$ to
$[\Lambda_{C^*}]$, and transport by $\iota_\Delta^-$ sends this to
$[\Lambda_{p(C)}]$. Thus
\begin{equation}\label{eq:fourier-multiplicities}
 m_{p(C)}(\FT P)=m_C(P).
\end{equation}

Finally, Fourier transform preserves the action of
$Z(\Gd)^{W_F}$, which acts trivially on $V^\pm$. Transport by
$\iota_\Delta^-$ replaces the action of $z$ by that of
$\alpha^{-1}(z)$. Since $\alpha$ is conjugation by $n_\Delta\in\Gd$,
we have $\alpha(z)=z$ for $z\in Z(\Gd)^{W_F}$.
Thus $\FT$ preserves enhancement central characters and, by the
inner-form condition in Hypothesis~\ref{hyp:enhanced-llc}, the
corresponding pure inner forms.

\section{Conormal bounds and rigidity}
\label{sec:conormal-bounds}

We continue with an Arthur-type infinitesimal parameter $\lambda$,
$H=Z_J(\hd)$, and $V^\pm=\mathfrak j(\pm2)$.
All norms and nilpotent saturations in this section are taken in the
original dual Lie algebra $\gd$.

\subsection{Neutral-element norms}
For a semisimple element $z^\vee$ conjugate into a real coroot span
of $[\gd,\gd]$, write $\|z^\vee\|^2=B(z^\vee,z^\vee)$.
For a nilpotent element $x^\vee\in\gd$, put
\begin{equation}\label{eq:neutral-norm}
 \nu(x^\vee)=\|h_x^\vee\|^2,
 \qquad (x^\vee,h_x^\vee,f_x^\vee)\text{ an }\mathfrak{sl}_2
 \text{-triple},\qquad \nu(0)=0.
\end{equation}
Conjugacy of triples makes $\nu$ well-defined and constant on
$\Gd$-orbits. We write $\nu(C)$ for its value on an $H$-orbit
in either $V^+$ or $V^-$. Since $\iota_\Delta^-$ is ambient conjugation,
\begin{equation}\label{eq:partner-norm}
 \nu(p(C))=\nu(C^*).
\end{equation}

We use two consequences of \cite[Proposition~3.2 and Corollary~3.3]{GS}.
First, for ambient nilpotent orbits,
\begin{equation}\label{eq:strict-neutral-norm}
 \mathcal O_1^\vee<\mathcal O_2^\vee
 \quad\Longrightarrow\quad
 \nu(\mathcal O_1^\vee)<\nu(\mathcal O_2^\vee).
\end{equation}
Second, if $z^\vee$ lies in a real coroot span and
$\gd(z^\vee;r)$ denotes its $r$-eigenspace, then
\begin{equation}\label{eq:GS-bound}
 x^\vee\in\bigoplus_{r\geq1}\gd(z^\vee;r)
 \quad\Longrightarrow\quad
 \nu(x^\vee)\leq4\|z^\vee\|^2.
\end{equation}

By Lemma~\ref{lem:strict-saturation} and
\eqref{eq:strict-neutral-norm},
\begin{equation}\label{eq:graded-strict-norm}
 C<D\quad\Longrightarrow\quad\nu(C)<\nu(D).
\end{equation}

\subsection{The conormal norm bound}
\begin{theorem}\label{thm:conormal-norm}
For every $(x^\vee,y^\vee)\in\Lambda$,
\begin{equation}\label{eq:conormal-norm-bound}
 \nu(x^\vee)+\nu(y^\vee)\leq\|\hd\|^2.
\end{equation}
Consequently, for every $H$-orbit $C\subset V$,
\begin{equation}\label{eq:orbit-conormal-bound}
 \nu(C)+\nu(p(C))\leq\|\hd\|^2.
\end{equation}
\end{theorem}
\begin{proof}
Choose a homogeneous triple $(x^\vee,h_x^\vee,f_x^\vee)$ in
$\mathfrak j$, using \cite[Theorem~1]{Vinberg1979}; take the zero
triple when $x^\vee=0$. Set $z^\vee=\hd-h_x^\vee$.
Homogeneity shows that $z^\vee$ centralizes the triple, so
\begin{equation}
 B(h_x^\vee,z^\vee)
 =B([x^\vee,f_x^\vee],z^\vee)
 =B(x^\vee,[f_x^\vee,z^\vee])=0.
\end{equation}
The commuting semisimple elements $\hd$ and $h_x^\vee$ lie in
$[\gd,\gd]$ and have integral adjoint eigenvalues. After simultaneous
conjugation into a Cartan subalgebra, they and $z^\vee$ therefore
lie in its real coroot span.

Decompose $y^\vee=\sum_m y_m^\vee$ into $h_x^\vee$-weight vectors.
Since $[x^\vee,y^\vee]=0$, each nonzero $y_m^\vee$ is a
highest-weight vector for the triple, and hence $m\geq0$.
As $[\hd,y^\vee]=-2y^\vee$, we obtain
\begin{equation}
 [z^\vee,y_m^\vee]=(-2-m)y_m^\vee,
 \qquad
 y^\vee\in\bigoplus_{r\geq1}\gd(-z^\vee/2;r).
\end{equation}
Applying \eqref{eq:GS-bound} yields
\begin{equation}
 \nu(y^\vee)\leq\|z^\vee\|^2
 =\|\hd\|^2-\|h_x^\vee\|^2
 =\|\hd\|^2-\nu(x^\vee).
\end{equation}
The same conclusion is immediate if $y^\vee=0$.
Finally, choose a conormal pair in $C\times C^*$ and use
\eqref{eq:partner-norm} to obtain \eqref{eq:orbit-conormal-bound}.
\end{proof}

\subsection{Rigidity of Arthur orbits}
\begin{theorem}\label{thm:arthur-rigidity}
If $C\subset V$ is an Arthur orbit, then
\begin{equation}\label{eq:arthur-norm-equality}
 \nu(C)+\nu(p(C))=\|\hd\|^2.
\end{equation}
For every $H$-orbit $D\subset V$,
\begin{equation}\label{eq:arthur-rigidity}
 C\leq D,\qquad p(C)\leq p(D)
 \quad\Longrightarrow\quad C=D.
\end{equation}
\end{theorem}
\begin{proof}
Choose commuting triples defining $C=He_1^\vee$, with
$h_1^\vee+h_2^\vee=\hd$. Lemma~\ref{lem:cunningham} gives
$p(C)=He_2^\vee$. Since the triples commute,
\begin{equation}
 B(h_1^\vee,h_2^\vee)
 =B(e_1^\vee,[f_1^\vee,h_2^\vee])=0,
 \qquad
 \nu(C)+\nu(p(C))
 =\|h_1^\vee\|^2+\|h_2^\vee\|^2=\|\hd\|^2.
\end{equation}
If $C<D$ and $p(C)\leq p(D)$, strict monotonicity
\eqref{eq:graded-strict-norm} and Theorem~\ref{thm:conormal-norm} give
\begin{equation}
 \|\hd\|^2
 =\nu(C)+\nu(p(C))
 <\nu(D)+\nu(p(D))
 \leq\|\hd\|^2,
\end{equation}
a contradiction.
\end{proof}

\section{Reconstruction and canonical packet distributions}
\label{sec:reconstruction}

\subsection{Stalk data and the local index formula}\label{subsec:local-index}
Suppose the $H$-orbits on $V$ are $C_1,\dots,C_n$, listed in increasing order of dimension $d_i=\dim C_i$. Choose a representative $x_i^\vee\in C_i$ for each $i$. We will use the labels $a\in \mathcal I_\lambda$ for the enhanced parameters with the corresponding simple $H$-equivariant perverse sheaves $P_a=\IC(\overline C_a,\mathcal L_{\rho_a})$. Let $r=|\mathcal I_\lambda|$. Write $d_a=\dim C_a$. 

Consider the stalk matrix $X$ and the nonnegative characteristic-cycle matrix $\mathbf M$. These are rectangular matrices of size $n\times r$, defined by
\begin{equation}\label{eq:stalk-cycle-matrices}
 X_{ia}=\chi((P_a)_{x_i^\vee})
       =\sum_{k\in\ZZ}(-1)^k\dim\mathcal H^k(P_a)_{x_i^\vee},
 \qquad \mathbf M_{ia}=m_{C_i}(P_a).
\end{equation}
Since the perverse sheaves are equivariant, the Euler characteristics $X_{ia}$ are constant along $C_i$, so independent of the choice of $x_i^\vee$. With our normalization:
\begin{equation}\label{eq:stalk-diagonal}
 X_{ia}=(-1)^{d_i}\dim\rho_a,\qquad
 \mathbf M_{ia}=\dim\rho_a\quad\text{if }C_a=C_i,
\end{equation}
and $X_{ia}=\mathbf M_{ia}=0$ unless $C_i\leq C_a$. 

Using the signed restriction multiplicities $s_{ba}$ from \eqref{eq:signed-restriction-multiplicities},
\begin{equation}\label{eq:enhanced-to-stalk}
 X_{ia}=\sum_{b,~C_b=C_i}(\dim\rho_b)s_{ba}.
\end{equation}
Compare \cite[Definitions~23.4 and~23.6]{ABV}.
When the IC stalks satisfy the parity convention
$\mathcal H^k(P_a)|_{C_i}=0$ unless $k\equiv-d_a\pmod2$,
we may instead use the analogue of the Kazhdan--Lusztig polynomials
\begin{equation}
 Q_{ba}(t)=\sum_{\ell\geq0}
 \dim\operatorname{Hom}(\mathcal L_{\rho_b},
             \mathcal H^{2\ell-d_a}(P_a)|_{C_i})\,t^\ell.
\end{equation}
Concretely,
\begin{equation}\label{eq:kl-to-stalk}
 X_{ia}=(-1)^{d_a}\sum_{b, ~C_b=C_i}(\dim\rho_b)Q_{ba}(1).
\end{equation}

If $Y\subset V$ is an irreducible complex variety, let
$\operatorname{Eu}_Y$ denote MacPherson's local Euler obstruction
function (extended by zero outside $Y$). Normalize
$\operatorname{Eu}_Y=1$ on the smooth locus of $Y$; see
\cite{Mac74}. For a possibly reducible orbit closure, set
\begin{equation}
 \operatorname{Eu}_{\overline C}
 =\sum_{Y\in\operatorname{Irr}(\overline C)}\operatorname{Eu}_Y.
\end{equation}
This function is $H$-invariant and equals $1$ on $C$, consistently
with our convention for $[\Lambda_C]$. 

Define the $n\times n$ matrices
\begin{equation}\label{eq:euler-sign-matrices}
 E_{ij}=\operatorname{Eu}_{\overline{C_j}}(x_i^\vee),
 \qquad D=\operatorname{diag}((-1)^{d_1},\ldots,(-1)^{d_n}).
\end{equation}
Since $E_{ii}=1$ and $E_{ij}=0$ unless $C_i\leq C_j$, the matrix
$E$ is upper unitriangular with integer entries. The columns of $E$ form a basis
of the integer-valued $H$-invariant functions on $V$.

The following is Kashiwara's local index theorem \cite{Kas73}; see
\cite[Theorem~1.31 and \S24]{ABV} for the formulation
in terms of local and microlocal multiplicities. 

\begin{theorem}[Local index formula]\label{thm:local-index}
For every $P\in\Perv_H(V)$ and every $x_i^\vee\in C_i$,
\begin{equation}\label{eq:local-index}
 \chi(P_{x_i^\vee})
 =\sum_{j=1}^n(-1)^{d_j}
     \operatorname{Eu}_{\overline{C_j}}(x_i^\vee)\,m_{C_j}(P).
\end{equation}
Equivalently, for the matrices above,
\begin{equation}\label{eq:local-index-matrix}
 X=ED\mathbf M,\qquad
 \mathbf M=KX,\qquad K:=DE^{-1}.
\end{equation}
In particular $K$ has integer entries and satisfies
\begin{equation}\label{eq:index-triangularity}
 K_{ii}=(-1)^{d_i},\qquad
 K_{ij}=0\quad\text{unless }C_i\leq C_j.
\end{equation}
\end{theorem}

We note that the
constructible function corresponding to $[\Lambda_{C_j}]$ is
$(-1)^{d_j}\operatorname{Eu}_{\overline{C_j}}$, so \eqref{eq:local-index} follows by applying this
to \eqref{eq:characteristic-cycle}. In fact, the entries $K_{ij}$ are the coefficients $c(C_i,C_j)$ of
\cite[Theorem~1.31]{ABV}. Their description by normal Morse data
is given in \cite[(24.17)(b)]{ABV}. As a sanity check for the sign conventions: 
let $Y\subset V$ be smooth, closed, and $H$-stable, of dimension
$d$, and take the extension of $\CC_Y[d]$ to $V$. The stalk
function is $(-1)^d1_Y$, its characteristic cycle is $[T_Y^*V]$,
and $\operatorname{Eu}_Y=1_Y$, as required by
\eqref{eq:local-index}.

\begin{lemma}\label{lem:stalk-full-row-rank}
The matrices $X$ and $\mathbf M$ have full row rank $n$. 
\end{lemma}
\begin{proof}
This is well known. For each $C_j$, select the column corresponding to
$\IC(\overline{C_j},\CC)$, i.e., the trivial local system.
The resulting submatrix of $X$ is upper triangular with diagonal
$(-1)^{d_j}$ by \eqref{eq:stalk-diagonal}. The corresponding submatrix of $\mathbf M$
is also upper unitriangular, by \eqref{eq:microlocal-diagonal}.
\end{proof}
\subsection{Fourier compatibility and determination of Arthur rows}

Write $p(i)$ for the index of $p(C_i)$, and let $f$ be the permutation
of $\mathcal I_\lambda$ determined by $\FT P_a\simeq P_{f(a)}$.
Given the conventions in Section~\ref{subsec:weyl-arthur}, these are both involutive permutations: $p$ acts on the $n$ orbits,
while $f$ acts on the $r$ enhanced parameters. Define the corresponding  permutation matrices by
\begin{equation}
 \mathsf P_{ij}=\begin{cases}1&j=p(i),\\0&\text{otherwise},\end{cases}
 \qquad
 \mathsf F_{ab}=\begin{cases}1&b=f(a),\\0&\text{otherwise}.\end{cases}
\end{equation}
We have $(\mathsf P A)_{ia}=A_{p(i),a}$ and
$(A\mathsf F)_{ia}=A_{i,f(a)}$ for any $n\times r$ matrix $A$.
The relation between the Fourier transform and the characteristic cycle  \eqref{eq:fourier-multiplicities}
can be expressed in linear algebra language as 
\begin{equation}\label{eq:fourier-matrix}
 \mathsf P\mathbf M=\mathbf M\mathsf F\qquad\text{ or }
 \qquad m_{C_i}(P_a)=m_{C_{p(i)}}(P_{f(a)}).
\end{equation}
We analyze the rigidity of the linear system $X=ED\mathbf M$ from Theorem \ref{thm:local-index}. Suppose $L$ is an unknown $n\times n$ matrix satisfying the linear system
\begin{equation}\label{eq:reconstruction-system}
 \begin{split}
 L_{ij}&=0\quad\text{unless }C_i\leq C_j,\qquad
 L_{ii}=(-1)^{d_i},\\
 \mathsf P L X&=L X\mathsf F.
 \end{split}
\end{equation}
The geometric matrix $K=DE^{-1}$ is a solution, so this system is
consistent. We do not assume any knowledge of $E$ beyond its upper unitriangularity. 

\begin{proposition}
\label{prop:reconstruction-ambiguity}
Every solution of \eqref{eq:reconstruction-system} has the form
\begin{equation}\label{eq:reconstruction-ambiguity}
 L=(1+N)K,
 \end{equation}
where
\begin{equation}\label{eq:ambiguity-support}
 N_{ij}=0\quad\text{unless }C_i<C_j,
 \qquad N_{ij}=N_{p(i),p(j)}.
\end{equation}
Conversely, every such $N$ gives a solution. In particular, a
nonzero $N_{ij}$ requires both
$C_i<C_j$ and $p(C_i)<p(C_j)$.
\end{proposition}

\begin{proof}
    For a solution $L$, set $N=LK^{-1}-1$. From the conditions on $L$ and $K$, $N_{ij}=0$ unless $C_i<C_j$. Since $LX=(1+N)\mathbf M$, we get
    \begin{equation}
    \mathsf P(1+N)\mathbf M=\mathsf PLX=LX\mathsf F=(1+N) \mathbf M \mathsf F=(1+N)\mathsf P\mathbf M,
    \end{equation}
    whence
    \begin{equation}
    (\mathsf P N-N\mathsf P) \mathbf M=0.
    \end{equation}
    Since $\mathbf M$ has full row rank, it follows that $\mathsf PN=N\mathsf P$, which is equivalent
to the second identity in \eqref{eq:ambiguity-support}. The converse is proved in the same way. 
\end{proof}

\begin{theorem}[Uniqueness of Arthur rows]
\label{thm:arthur-row-reconstruction}
For an Arthur orbit
$C_i$, every solution $L$ of \eqref{eq:reconstruction-system}
satisfies
\begin{equation}
 L_{i\bullet}=K_{i\bullet},\qquad
 (LX)_{i\bullet}=\mathbf M_{i\bullet}.
\end{equation}
Consequently all microlocal
multiplicities $m_{C_i}(P_a)$, and hence the packet
$\Pi_{C_i}^{\mic}$, at every Arthur orbit are uniquely determined.
\end{theorem}
\begin{proof}
If $N_{ij}\ne0$ in Proposition~\ref{prop:reconstruction-ambiguity},
then $C_i<C_j$ and $p(C_i)<p(C_j)$, contrary to
\eqref{eq:arthur-rigidity}. Hence the $i$th row of
$N$ vanishes. Now use $L=(1+N)K$ and multiply by $X$.
\end{proof}

In other words, knowing the matrix $X$, the closure order, and the
permutations $p$ and $f$ uniquely gives us a linear reconstruction method
to determine the microlocal Arthur packets. 

\begin{remark}
    The ambiguity in the linear system \eqref{eq:reconstruction-system}, away from the Arthur rows, is 
explicit. Set
\begin{equation}
 \mathcal J=\{(i,j):C_i<C_j,\ p(C_i)<p(C_j)\}.
\end{equation}
Then the vector space of solution matrices $N$ in
\eqref{eq:ambiguity-support} has a basis indexed by each
orbit of $(i,j)\mapsto(p(i),p(j))$ on $\mathcal J$.
\end{remark}

\subsection{A canonical characterization of the Arthur microlocal packet}

The geometric rigidity of the Arthur orbits provides a characterization of the
normalized distributions $T_C$ from
\eqref{eq:normalized-microlocal-distribution}. The underlying space is the direct
sum of the character spaces over pure inner forms at $\lambda$.

Recall that $\mathcal I_\lambda$ indexes the enhanced parameters
$a=(C_a,\rho_a)$ at the fixed infinitesimal parameter $\lambda$.
The corresponding irreducible representation $\pi_a$ belongs to the pure inner
form $G^{\delta_a}(F)$. Write $\Theta_a=\Theta_{\pi_a}$ for its
distribution character. Define
\begin{equation}
 \mathscr R_\lambda
 =\bigoplus_{\delta\in H^1(F,G)}
   \ \bigoplus_{\substack{a\in\mathcal I_\lambda\\\delta_a=\delta}}
       \CC\,\Theta_a.
\end{equation}

Let $\CC^{\mathcal I_\lambda}$ denote the space of vectors
$v=(v_a)_{a\in\mathcal I_\lambda}$, with one coordinate for each
enhanced parameter $a$. Define the linear isomorphism
\begin{equation}\label{eq:character-pairing-identification}
 \mathcal J_\lambda:\CC^{\mathcal I_\lambda}\longrightarrow
       \mathscr R_\lambda,\qquad
 v\longmapsto\sum_{a\in\mathcal I_\lambda}
       e(\delta_a)(-1)^{d_a}v_a\Theta_a.
\end{equation}
In other words, this map replaces the $a$-th coordinate by the coefficient of
$\Theta_a$ times the sign $e(\delta_a)(-1)^{d_a}$.
The map is clearly invertible since the distributions
$\Theta_a$ form a basis of $\mathscr R_\lambda$.
The signs come from the normalization of
$T_C$. More precisely, if $\mathbf m_C=(m_C(P_a))_{a\in\mathcal I_\lambda}$
is the vector of microlocal multiplicities at $C$, then
\begin{equation}
 T_C=(-1)^{\dim C}\mathcal J_\lambda(\mathbf m_C).
\end{equation}
When $C=C_i$, $\mathbf m_{C_i}=\mathbf M_{i\bullet}$.

\medskip

Define the geometric standard packet sums by
\begin{equation}\label{eq:standard-packet-sums}
 S_{C_i}=\mathcal J_\lambda(X_{i\bullet}),\qquad
 \mathscr S_\lambda^{\mathrm{std}}
      =\Span_{\CC}\{S_{C_i}:1\leq i\leq n\},\qquad
 \mathscr S_{\lambda,\ZZ}^{\mathrm{std}}
      =\bigoplus_{i=1}^n\ZZ S_{C_i}.
\end{equation}
\begin{lemma}\label{lem:standard-packet-inner-forms}
Under Hypothesis~\ref{hyp:enhanced-llc},
\begin{equation}\label{eq:standard-packet-enhanced}
 S_C=\sum_{b:C_b=C}e(\delta_b)\dim(\rho_b)\Theta_{I_b}.
\end{equation}
The restriction of $S_C$ to every pure inner form $G^\delta(F)$ is stable.
\end{lemma}
\begin{proof}
Equivariant restriction preserves the scalar action of
$Z(\Gd)^{W_F}$. Thus $s_{ba}=0$ unless the Kottwitz characters
agree, hence unless $\delta_a=\delta_b$. By
\eqref{eq:enhanced-to-stalk} and the Kazhdan--Lusztig formula
\eqref{eq:llc-kl-hypothesis},
\begin{equation}
\begin{aligned}
 S_C
 &=\sum_{b:C_b=C}e(\delta_b)\dim(\rho_b)
       \sum_{a:\delta_a=\delta_b}(-1)^{d_a}s_{ba}\Theta_a\\
 &=\sum_{b:C_b=C}e(\delta_b)\dim(\rho_b)\Theta_{I_b}.
\end{aligned}
\end{equation}
For $C=C_\phi$, write
$\phi=\iota_L\circ(\phi_L\otimes z)$ for its Langlands data.
If $\phi$ is irrelevant to $G^\delta$, the restriction is zero.
Otherwise, the Langlands classification compatibility in
Hypothesis~\ref{hyp:enhanced-llc} gives
\begin{equation}
 S_C|_{G^\delta(F)}
 =e(\delta)i_{P_\delta}^{G^\delta}
       (\nu_\delta S\Theta_{\phi_L}^{\delta_L}),
\end{equation}
where $\delta_L$ denotes the pure inner twist $L_\delta$ of $L$
and $\nu_\delta$ is the character corresponding to $z$.
Tempered stability on $L_\delta(F)$, together with preservation of
stability under unramified twisting and normalized induction, proves
the claim.
\end{proof}

\begin{remark}
The superscript $\mathrm{std}$ denotes standard packet sums.
Thus $\mathscr S_\lambda^{\mathrm{std}}$ is their span across pure
inner forms, and its restriction to each $G^\delta(F)$ consists of stable virtual
characters under Hypothesis~\ref{hyp:enhanced-llc}.
In \cite[Conjecture~8.14]{Vog}, attributed there to Langlands and
Shelstad, these standard packet sums are conjectured to form a basis
of the lattice of strongly stable virtual representations across pure
inner forms at $[\lambda]$. Thus $\mathscr S_\lambda^{\mathrm{std}}$
is expected to be the complex span of that lattice. The lemma proves
stability on each pure inner form under our hypotheses; the full
basis assertion is not needed below.
\end{remark}

Note for comparison the real group microlocal character framework in
\cite[Theorem~1.29 and Corollary~1.32]{ABV}, and the local field
generalization in \cite[\S8]{Vog}.

\medskip

Since $T_{C_i}=(-1)^{d_i}\mathcal J_\lambda(\mathbf M_{i\bullet})$,
the identity $\mathbf M=DE^{-1}X$ gives
\begin{equation}\label{eq:triangular-microlocal-distributions}
 T_{C_i}=\sum_j(E^{-1})_{ij}S_{C_j}
       =S_{C_i}+\sum_{C_j>C_i}(E^{-1})_{ij}S_{C_j}.
\end{equation}
The inverse identity is $S_{C_i}=\sum_j E_{ij}T_{C_j}$.
Therefore $\{T_{C_i}\}$ is another integral basis of
$\mathscr S_{\lambda,\ZZ}^{\mathrm{std}}$, with leading standard
coefficient one. This fixes the normalization in
\eqref{eq:normalized-microlocal-distribution}.
\begin{corollary}[Stability of microlocal distributions]
\label{cor:microlocal-stability}
Under Hypothesis~\ref{hyp:enhanced-llc}, the restriction of $T_C$
to every pure inner form $G^\delta(F)$ is stable for every $H$-orbit $C\subset V$.
In particular, the microlocal Arthur distribution $T_{C_\psi}|_{G^\delta(F)}$
is stable for every Arthur parameter $\psi$ at $\lambda$.
\end{corollary}
\begin{proof}
By \eqref{eq:triangular-microlocal-distributions}, $T_C$ is a linear
combination of the $S_B$, whose restrictions to each $G^\delta(F)$ are stable by
Lemma~\ref{lem:standard-packet-inner-forms}.
\end{proof}

\medskip

We will also need the equivalent pairing formulation for the compatibility with the Fourier transform. Put
\begin{equation}
 \mathscr K_\lambda
   =K_0(\Perv_H(V))\otimes_{\ZZ}\CC
   =\bigoplus_{b\in\mathcal I_\lambda}\CC[P_b].
\end{equation}
Here $K_0$ denotes the Grothendieck group of the abelian category. Define a complex perfect bilinear pairing
$\mathscr R_\lambda\times\mathscr K_\lambda\to\CC$ by
\begin{equation}
 \langle\Theta_a,[P_b]\rangle
      =e(\delta_a)(-1)^{d_a}\delta_{ab},
\end{equation}
where $\delta_{ab}$ is the Kronecker symbol. By construction,
\begin{equation}
 \langle\mathcal J_\lambda(v),[P_b]\rangle=v_b.
\end{equation}

Transport the Fourier transform action to the character space by
\begin{equation}\label{eq:character-fourier-transport}
 \FT\mathcal J_\lambda(v)=\mathcal J_\lambda(v\mathsf F).
\end{equation}
This is the involution characterized by
$\langle\FT T,[P]\rangle=\langle T,[\FT P]\rangle$.
Equations~\eqref{eq:fourier-matrix}
and \eqref{eq:character-fourier-transport} give
\begin{equation}\label{eq:fourier-microlocal-distributions}
 \FT T_C=(-1)^{\dim C-\dim p(C)}T_{p(C)}.
\end{equation}
In particular, $\FT$ preserves the standard packet space.

For an orbit $C$, define
\begin{equation}\label{eq:standard-support-filtrations}
 U_C=\Span_{\CC}\{S_D:C\leq D\},\qquad
 U_{C,\ZZ}=\bigoplus_{D\geq C}\ZZ S_D,
 \qquad U_{>C}=\Span_{\CC}\{S_D:C<D\}.
\end{equation}
The triangularity in 
\eqref{eq:triangular-microlocal-distributions} implies
\begin{equation}
 U_C=\Span_{\CC}\{T_D:C\leq D\},\qquad
 U_{C,\ZZ}=\bigoplus_{D\geq C}\ZZ T_D.
\end{equation}

\begin{proposition}\label{prop:joint-filtration}
For any two orbits $C,B$,
\begin{align}
 U_C\cap\FT^{-1}U_B
   &=\bigoplus_{\substack{D\geq C\\p(D)\geq B}}\CC T_D,
       \label{eq:joint-filtration}\\
 U_{C,\ZZ}\cap\FT^{-1}U_{B,\ZZ}
   &=\bigoplus_{\substack{D\geq C\\p(D)\geq B}}\ZZ T_D.
       \label{eq:joint-filtration-integral}
\end{align}
\end{proposition}
\begin{proof}
Write a vector $\chi$ in the basis $(T_D)$, $\chi=\sum_D\chi_D T_D$.
If $\chi\in U_C$, then $\chi_D=0$ unless $D\geq C$. By
\eqref{eq:fourier-microlocal-distributions}, if $\chi\in\FT^{-1}U_B$,
$\chi_D=0$ unless $p(D)\geq B$. The claim follows.
\end{proof}

\begin{theorem}[Canonical characterization of an Arthur distribution]
\label{thm:canonical-line}
Let $C$ be an Arthur orbit. Then
\begin{equation}\label{eq:canonical-line}
 U_C\cap\FT^{-1}U_{p(C)}=\CC T_C,
 \qquad
 U_{C,\ZZ}\cap\FT^{-1}U_{p(C),\ZZ}=\ZZ T_C.
\end{equation}
The element $T_C$ is the unique vector in the complex intersection
satisfying
\begin{equation}\label{eq:canonical-normalization}
 T_C\equiv S_C\pmod{U_{>C}}.
\end{equation}
\end{theorem}
\begin{proof}
Set $B=p(C)$ in Proposition~\ref{prop:joint-filtration}.
The geometric rigidity of the Arthur orbit $C$,
\eqref{eq:arthur-rigidity}, leaves only $D=C$ in the direct sums.
The leading coefficient in
\eqref{eq:triangular-microlocal-distributions} is one, which proves
the normalization and uniqueness.
\end{proof}

\subsection{A representation-theoretic characterization of the Arthur distributions}\label{subsec:langlands-aubert-support}
For a virtual character $T=\sum_a t_a\Theta_a\in\mathscr R_\lambda$,
define its support by
\begin{equation}
 \operatorname{Supp}_{\mathrm{irr}}(T)=\{\pi_a:t_a\ne0\}.
\end{equation}
Write $C(\pi_a)=C_a$ for the orbit supporting the enhanced Langlands parameter of $\pi_a$.
Given an orbit $C$, let
\begin{equation}
 \mathscr R_{\geq C}
   =\Span_{\CC}\{\Theta_a:C\leq C_a\}.
\end{equation}
Thus $T\in\mathscr R_{\geq C}$ means that the Langlands orbit for every irreducible
character occurring in $T$ contains $C$ in its closure. 

\begin{lemma}
\label{lem:standard-support-detection}
For every orbit $C$,
\begin{equation}\label{eq:standard-support-detection}
 \mathscr S_\lambda^{\mathrm{std}}\cap\mathscr R_{\geq C}=U_C.
\end{equation}
\end{lemma}
\begin{proof}
Since $X_{D,a}=0$ unless $D\leq C_a$, we have
$U_C\subseteq\mathscr S_\lambda^{\mathrm{std}}\cap\mathscr R_{\geq C}$.
Conversely, let $T=\sum_Du_DS_D$ belong to this intersection.
If some $u_D\ne0$ with $D\not\geq C$, choose such a $D$ minimal
in the closure order. For $a=(D,1)$, only $u_DX_{D,a}$ contributes
to the coefficient of $\Theta_a$, up to the nonzero normalization sign:
smaller orbits with nonzero coefficients cannot lie above $C$.
Since $X_{D,a}=(-1)^{\dim D}$, this coefficient is nonzero,
contradicting $T\in\mathscr R_{\geq C}$.
\end{proof}

For an Arthur orbit $C$, define
\begin{equation}\label{eq:fourier-admissible-support}
 \mathcal B_C^{\FT}
 =\{\pi_a:C\leq C_a,\ p(C)\leq C_{f(a)}\}.
\end{equation}
The Fourier operator of
\eqref{eq:character-fourier-transport} satisfies
\begin{equation}
 \FT\Theta_a
 =(-1)^{d_a+d_{f(a)}}\Theta_{f(a)}.
\end{equation}
Thus $\FT$ transports irreducible support by $f$. The canonical-line
theorem now has the following support formulation.

\begin{corollary}\label{cor:fourier-character-support}
For an Arthur orbit $C$,
\begin{equation}\label{eq:fourier-character-support}
 \mathscr S_\lambda^{\mathrm{std}}\cap
 \Span_{\CC}\{\Theta_a:\pi_a\in\mathcal B_C^{\FT}\}
 =\CC T_C.
\end{equation}
\end{corollary}
\begin{proof}
By Lemma~\ref{lem:standard-support-detection}, the intersection is
$U_C\cap\FT^{-1}U_{p(C)}$. Apply Theorem~\ref{thm:canonical-line}.
\end{proof}

Now let $\operatorname{AZ}$ be Aubert--Zelevinsky duality on the
Grothendieck group of smooth representations, and write
$\widehat\pi=|\operatorname{AZ}(\pi)|$ for the dual irreducible
representation, whenever $\pi$ is irreducible. 
The following additional hypothesis allows us to express the second
orbit bound in terms of Aubert duality.

\begin{hypothesis}[Aubert compatibility]\label{hyp:aubert-compatibility}
For every $a\in\mathcal I_\lambda$, the representation
$\widehat\pi_a$ has infinitesimal character $[\lambda]$.
For every Arthur orbit $A\subset V$,
\begin{equation}\label{eq:aubert-orbit-compatibility}
 A\leq C(\widehat\pi_a)\quad\Longleftrightarrow\quad A\leq C_{f(a)}.
\end{equation}
\end{hypothesis}

This is an additional compatibility hypothesis beyond the standard
LLC properties in Hypothesis~\ref{hyp:enhanced-llc}.
In Section~\ref{subsec:aubert-comparison}, we deduce it
from Antor--Okada's theorem for unipotent representations of inner
forms of simple adjoint unramified groups at unramified Arthur-type
infinitesimal parameters.
We expect this compatibility to hold for all infinitesimal parameters
of Arthur type.

Define the set of representations satisfying the simultaneous bounds by
\begin{equation}\label{eq:aubert-admissible-support}
 \mathcal B_\psi
 =\{\pi_a:C_\psi\leq C(\pi_a),\quad
                  C_{\psi^t}\leq C(\widehat\pi_a)\},
\end{equation}
where $\psi$ is an Arthur parameter as above. 
Lemma~\ref{lem:cunningham} gives
$C_{\psi^t}=p(C_\psi)$. Under
Hypothesis~\ref{hyp:aubert-compatibility}, therefore,
$\mathcal B_\psi=\mathcal B_{C_\psi}^{\FT}$.

Let $\operatorname{res}_0$ denote the projection onto the summand for the quasi-split $p$-adic form, and set
\begin{equation}
 \mathcal I_\lambda^0=\{a\in\mathcal I_\lambda:\delta_a=0\},
 \qquad
 \mathscr R_\lambda^0=\bigoplus_{a\in\mathcal I_\lambda^0}
                         \CC\Theta_a.
\end{equation}
Introduce the notation
\begin{equation}
 S_C^0=\operatorname{res}_0S_C,\qquad
 T_C^0=\operatorname{res}_0T_C,\qquad
 \mathscr S_\lambda^0=\Span_{\CC}\{S_C^0:C\in H\backslash V\}.
\end{equation}
Write $U_C^0=\operatorname{res}_0U_C$ and
$\mathscr R_{\geq C}^0=\mathscr R_{\geq C}\cap
\mathscr R_\lambda^0$. For a distribution $D=\sum_a c_a\Theta_a$, let
\begin{equation}\langle\Theta_a\mid D\rangle=c_a,\end{equation}
denote the coefficient of $\Theta_a$.
If $C=C_\psi$ is an Arthur orbit, set
$\mathcal B_\psi^0=\mathcal B_\psi\cap
\Irr_\lambda(G(F))$.

Let $f$ be the involution defined by $\FT P_a\simeq P_{f(a)}$.
It preserves $\mathcal I_\lambda^0$. The Fourier operator
restricts to an involution $\FT^0$ on
$\mathscr R_\lambda^0$, preserves $\mathscr S_\lambda^0$, and satisfies
$\operatorname{res}_0\FT=\FT^0\operatorname{res}_0$.

\begin{theorem}[Characterization of microlocal Arthur distributions]
\label{thm:aubert-support-characterization}
\label{subsec:analytic-inputs}
\label{thm:endoscopic-comparison-criterion}
\label{cor:analytic-uniqueness}
Let $C=C_\psi$ be an Arthur orbit. Assume
Hypothesis~\ref{hyp:aubert-compatibility}.
\begin{enumerate}
\item The restriction of $T_C$ to every pure inner form $G^\delta(F)$
is stable, and across pure inner forms we have
\begin{equation}\label{eq:aubert-support-characterization}
 \mathscr S_\lambda^{\mathrm{std}}\cap
 \Span_{\CC}\{\Theta_a:\pi_a\in\mathcal B_\psi\}
 =\CC T_C.
\end{equation}
\item On the quasi-split form, $T_C^0$ is the unique
$D\in\mathscr R_\lambda^0$ satisfying
\begin{align}
 D&\in\mathscr S_\lambda^0,
       \label{eq:comparison-standard-span}\\
 \langle \Theta_a\mid D\rangle \ne0&\ \Longrightarrow\quad
       C\leq C(\pi_a),\qquad p(C)\leq C(\widehat\pi_a),
       \label{eq:comparison-two-bounds}\\
 \langle\Theta_{\pi_{C,\mathbf1}}\mid D\rangle&=1.
       \label{eq:comparison-top-coefficient}
\end{align}
\end{enumerate}
For (2), it suffices to assume Hypothesis~\ref{hyp:aubert-compatibility}
on the quasi-split summand.
\end{theorem}
\begin{proof}
For (1), stability follows from Corollary~\ref{cor:microlocal-stability}.
For the equality, use $\mathcal B_\psi=\mathcal B_{C}^{\FT}$ and apply
Corollary~\ref{cor:fourier-character-support}.

For (2), the trivial enhancements belong to the quasi-split form, and their
coefficients in the $S_B^0$ form a unitriangular matrix by
\eqref{eq:stalk-diagonal}. Thus restriction identifies
$\mathscr S_\lambda^{\mathrm{std}}$ with $\mathscr S_\lambda^0$.
The argument of Lemma~\ref{lem:standard-support-detection} gives
\begin{equation}\label{eq:split-standard-support}
 \mathscr S_\lambda^0\cap\mathscr R_{\geq B}^0=U_B^0.
\end{equation}
Since restriction commutes with Fourier transform,
Hypothesis~\ref{hyp:aubert-compatibility} and
Theorem~\ref{thm:canonical-line} identify the distributions satisfying
\eqref{eq:comparison-standard-span}--\eqref{eq:comparison-two-bounds} with
\begin{equation}\label{eq:split-canonical-line}
 U_C^0\cap(\FT^0)^{-1}U_{p(C)}^0=\CC T_C^0.
\end{equation}
For $D=\sum_{B\geq C}u_BS_B^0$, the diagonal stalk formula gives
\begin{equation}\label{eq:split-leading-coefficient}
 \langle\Theta_{\pi_{C,\rho}}\mid D\rangle=u_C\dim\rho
 \qquad((C,\rho)\in\mathcal I_\lambda^0).
\end{equation}
As $T_C^0$ has leading standard coefficient one,
\eqref{eq:comparison-top-coefficient} selects $D=T_C^0$.
\end{proof}

For a standard expansion $D=\sum_Bu_BS_B^0$, the first orbit
bound forces $u_B=0$ unless $B\geq C$, by
\eqref{eq:split-standard-support}. Formula
\eqref{eq:split-leading-coefficient} then identifies
\eqref{eq:comparison-top-coefficient} with $u_C=1$.
Hence a leading standard coefficient of one gives the
normalization required by the theorem.

\medskip

\section{Unipotent packets of unramified groups}
\label{sec:comparison}

We now verify the hypotheses of our characterization for unipotent
representations. 

\subsection{The unipotent correspondence}
\label{subsec:unipotent-llc}
Let $G$ be quasi-split, simple and adjoint, and suppose that it splits
over an unramified extension of $F$. We use Lusztig's correspondence
for unipotent representations of $G$ and its inner forms
\cite{Lus95,Lus02}, with the conventions of \cite{AO}.
All these inner forms are pure because $G$ is adjoint.

\begin{proposition}\label{prop:unipotent-llc-hypotheses}
\begin{enumerate}
\item Lusztig's unipotent correspondence satisfies the inner-form,
Kazhdan--Lusztig multiplicity, temperedness, and Langlands-classification
properties of Hypothesis~\ref{hyp:enhanced-llc}, for $G$ and its
standard Levi subgroups, on all pure inner forms.
\item If $G=\mathrm{SO}_{2n+1}$ is split, $\operatorname{char}F=0$, and the
residual characteristic satisfies $p\geq6n+4$, Lusztig's parametrization
on $G(F)$ agrees with Arthur's, and tempered packet stability holds
on both pure inner forms. Thus Hypothesis~\ref{hyp:enhanced-llc} holds in full for
unipotent representations in this case.
\end{enumerate}
\end{proposition}
\begin{proof}
For (1), the inner-form parametrization is part of Lusztig's
correspondence \cite{Lus95,Lus02}. The multiplicity formula follows
from his Hecke-algebra realization and
\cite[\S\S10.4--10.8]{LusGHAII}, while \cite{LusGHAIII} supplies the
induction and temperedness results. Their formulations for reductive
groups, including the Levi subgroups needed here, are given in
\cite[Theorem~1]{SolLLC} and \cite[Theorem~5.4(b)]{Sol}.

For (2), we use Waldspurger's tempered normalization of Lusztig's
correspondence \cite[Introduction]{WaldI}. This is consistent with
Antor--Okada's normalization in \cite[\S4.9]{AO}.
The stability and twisted-transfer
identities recalled in \cite[\S2.1]{WaldI}, together with the
endoscopic identities proved in \cite[\S1.4]{WaldII}, characterize
the enhanced tempered parametrization. They therefore identify it
with Arthur's on the split form and give tempered packet stability
on both pure inner forms. The overall minus sign in the transfer
to the nonsplit form does not affect stability.

The Levi subgroups have general linear and smaller odd orthogonal
factors. For the general linear factors, agreement with the classical LLC
is the unipotent case of \cite[Proposition~5.2]{AMSHecke}. Waldspurger gives the corresponding agreement for the odd
orthogonal factors. Compatibility with unramified twists and Langlands
classification then extends the agreement on the split form to all
unipotent representations. Tempered stability on the Levi subgroups
follows from the same factorwise description, completing (2).
\end{proof}

Tempered packet stability for Lusztig's correspondence remains open
in general, and we retain it as an assumption outside the cases where
it is known.

\subsection{Aubert compatibility in the unipotent setting}
\label{subsec:aubert-comparison}
Antor--Okada identify Aubert duality with Fourier transform followed
by Chevalley transport and duality of local systems in the setting of
\cite[Theorem~6.20(1)]{AO}.

Antor--Okada use the Chevalley involution to identify the $2$- and
$-2$-eigenspaces. In our general setting, however,
a Chevalley involution of $J^\circ$ need not come with an extension to
$J$. Weyl transport is defined on the full disconnected group and
preserves its equivariant structures. To compare these choices when
Chevalley transport is available, we introduce $\beta$ below.

\begin{lemma}\label{lem:chevalley-weyl-comparison}
Choose a maximal torus $T\subset J^\circ$ containing the grading torus.
Suppose a Chevalley involution of $J^\circ$ acting by inversion on $T$
extends to an algebraic automorphism $\theta$ of $J$. Its differential
induces a map $V^-\to V$. The automorphism
\begin{equation}
 \beta=\theta\circ(\iota_\Delta^-)^{-1}:V\longrightarrow V
\end{equation}
permutes $H$-orbits, preserves their closure order, and fixes every
Arthur orbit. Consequently, for an Arthur orbit $A$ and any orbit $D$,
\begin{equation}\label{eq:beta-arthur-bound}
 A\leq\beta(D)\quad\Longleftrightarrow\quad A\leq D.
\end{equation}
\end{lemma}
\begin{proof}
Since $\theta(\hd)=-\hd$, both $\theta$ and Weyl transport normalize
$H$ and reverse the grading. Thus $\beta$ induces a closure-order
automorphism of $H\backslash V$.

Let $A=He_2^\vee$ arise from commuting triples with
$h_1^\vee+h_2^\vee=\hd$. By \eqref{eq:triple-sums},
$(\iota_\Delta^-)^{-1}(A)=Hf_2^\vee$.
Conjugating by $H^\circ$, we may assume $h_2^\vee\in\Lie(T)$;
then $h_1^\vee\in\Lie(T)$ as well. In the reductive group
$L=Z_{J^\circ}(h_1^\vee)$ the triples
\begin{equation}
 (e_2^\vee,h_2^\vee,f_2^\vee),\qquad
 (-\theta(f_2^\vee),h_2^\vee,-\theta(e_2^\vee))
\end{equation}
have the same neutral element. Their positive elements have open
orbits in the $2$-eigenspace under $Z_L(h_2^\vee)$, by
$\mathfrak{sl}_2$ representation theory, so these orbits coincide.
Since $Z_L(h_2^\vee)\subset H$ and negation preserves graded orbits,
$\theta(Hf_2^\vee)=He_2^\vee$. Hence $\beta(A)=A$, proving
\eqref{eq:beta-arthur-bound}.
\end{proof}

\begin{corollary}\label{cor:ao-aubert-compatibility}
Hypothesis~\ref{hyp:aubert-compatibility} holds for Lusztig's unipotent
correspondence for inner forms of simple adjoint unramified groups,
at every unramified infinitesimal parameter of Arthur type.
\end{corollary}
\begin{proof}
Write $s=\lambda(w_q)=s_cs_h$ for its compact--hyperbolic decomposition.
Since $\lambda$ is unramified, $J=Z_{\Gd}(s_c)$. This group is connected,
as in \cite[proof of Theorem~6.20]{AO}, so it admits a Chevalley
involution $\theta$ reversing the grading torus.
Let $D_a^-\subset V^-$ be the supporting orbit of $\mathfrak F_VP_a$.
By \cite[Theorem~6.20(1), (6.5.1)]{AO}, Aubert duality preserves $s$ and
\begin{equation}
 C(\widehat\pi_a)=\theta(D_a^-)=\beta(C_{f(a)}).
\end{equation}
The duality of local systems in that formula does not change the
supporting orbit. Preservation of $s$ gives preservation of the
infinitesimal character, and Lemma~\ref{lem:chevalley-weyl-comparison}
gives the required inequalities for every Arthur orbit.
\end{proof}

\subsection{Comparison for odd orthogonal groups}
\label{subsec:orthogonal-comparison}
For the rest of this section, let $F$ have characteristic zero and
residual characteristic $p\geq6n+4$, and let
$G=\mathrm{SO}_{2n+1}$ be split.
By Proposition~\ref{prop:unipotent-llc-hypotheses} and
Corollary~\ref{cor:ao-aubert-compatibility}, our hypotheses hold in this
setting, and the unipotent LLC agrees with Arthur's.

For an unramified Arthur parameter $\psi$, write
$\Pi_\psi^{\Art}(G(F))$ for Arthur's endoscopic packet and
\begin{equation}\label{eq:arthur-endoscopic-distribution}
 D_\psi^{\Art}
 =\sum_{\pi\in\Pi_\psi^{\Art}(G(F))}
       \langle s_\psi,\pi\rangle_\psi\Theta_\pi,
 \qquad s_\psi=\psi(1,1,-I).
\end{equation}
Here $\langle\ ,\pi\rangle_\psi$ is Arthur's packet character, with
Whittaker normalization \cite[Theorem~2.2.1]{Arthur2013}.

\begin{theorem}[Equality of endoscopic and microlocal Arthur packets]
\label{thm:unipotent-packet-comparison}
For every unramified Arthur parameter $\psi$ for $G$,
\begin{equation}\label{eq:unipotent-packet-comparison}
 D_\psi^{\Art}=T_{C_\psi}^0,
 \qquad
 \Pi_\psi^{\Art}(G(F))=\Pi_\psi^{\mic}(G(F)).
\end{equation}
\end{theorem}
\begin{proof}
Put $C=C_\psi$ and $\lambda=\lambda_\psi$.
Arthur's standard-character expansion
\cite[(2.2.9), (2.2.12)]{Arthur2013}, together with Waldspurger's
identification of the tempered correspondence, gives
\begin{equation}
 D_\psi^{\Art}\in\mathscr S_\lambda^0
\end{equation}
with coefficient one at the standard sum $S_C^0$.
Indeed, Arthur's stable standard characters are the normalized
inductions of the tempered packet sums, hence are precisely our
$S_B^0$. The coefficient at $\phi_\psi$ is one because Arthur's
standard and irreducible character expansions are inverse
unitriangular matrices; see \cite[(2.2.7)--(2.2.9)]{Arthur2013}.

The closure-order theorem \cite[Theorem~1.3]{HLLZ} gives
$C_\psi\leq C(\pi)$ for every $\pi\in\Pi_\psi^{\Art}(G(F))$.
Packetwise Aubert duality \cite[Appendix~A]{Xu2017} gives
$\widehat\pi\in\Pi_{\psi^t}^{\Art}(G(F))$.
Applying the closure-order theorem again yields
\begin{equation}
 C_\psi\leq C(\pi),\qquad
 C_{\psi^t}\leq C(\widehat\pi).
\end{equation}
These are the two bounds in
Theorem~\ref{thm:endoscopic-comparison-criterion}.
The leading standard coefficient supplies its normalization, by
\eqref{eq:split-leading-coefficient}. Thus
$D_\psi^{\Art}=T_C^0$.
Finally, the coefficients $\langle s_\psi,\pi\rangle_\psi$ are
$\pm1$, so the irreducible character support of $D_\psi^{\Art}$ is
exactly Arthur's packet. The support of $T_C^0$ is the microlocal
packet, proving the second equality.
\end{proof}

\begin{remark}[An alternative argument using twisted endoscopy]
The closure bounds used in
Theorem~\ref{thm:unipotent-packet-comparison} can also be obtained
from twisted endoscopy and the geometry of symmetric quivers.
Let
$$
 \iota:\mathrm{Sp}_{2n}(\mathbb C)
       \hookrightarrow\mathrm{GL}_{2n}(\mathbb C).
$$
Arthur's twisted character identities identify the coefficients in
$$
 D_\psi^{\Art}=\sum_\phi n_{\psi,\phi}S_{C_\phi}^0
$$
with those in the expansion of the twisted character of
$L_{\mathrm{GL}}(\iota\phi_\psi)$ in the twisted characters of
$I_{\mathrm{GL}}(\iota\phi)$; see
\cite[(2.2.12)]{Arthur2013} and
\cite[proof of Proposition~9.1.1]{LiMetaplectic}.
Here $\phi$ ranges over parameters with infinitesimal parameter
$\lambda_\psi$, and $I_{\mathrm{GL}}(\eta)$ denotes the standard
module with Langlands quotient $L_{\mathrm{GL}}(\eta)$.

A nonzero twisted composition coefficient forces a nonzero ordinary
composition multiplicity. The usual Kazhdan--Lusztig support property
therefore makes the twisted standard-to-irreducible character matrix
triangular for the orbit-closure order, with nonzero diagonal.
Inverting this matrix over all self-dual general-linear parameters
with infinitesimal parameter $\iota\lambda_\psi$ preserves that
support. Consequently,
$$
 n_{\psi,\phi}\ne0
 \quad\Longrightarrow\quad
 C_{\iota\phi_\psi}\leq C_{\iota\phi}.
$$

In the unramified setting, decomposition into Frobenius eigenspaces
reduces the comparison under $\iota$ to ordinary and symmetric
type-$A$ quiver varieties. By
\cite[Theorems~2.5 and~7.1]{BCI}, the resulting map on orbits is
injective and reflects the closure order. Thus
$n_{\psi,\phi}\ne0$ implies $C_\psi\leq C_\phi$.
Every member $\pi$ of $\Pi_\psi^{\Art}(G(F))$ has nonzero coefficient
in $D_\psi^{\Art}$, so some standard summand contributing to $\pi$
has $n_{\psi,\phi}\ne0$. The standard-module multiplicity formula
\cite[Theorem~5.4(b)]{Sol} then gives
$$
 C_\psi\leq C_\phi\leq C(\pi).
$$
Applying the same argument to $\psi^t$ and using packetwise Aubert
duality \cite[Appendix~A]{Xu2017} supplies the second closure bound.
Theorem~\ref{thm:endoscopic-comparison-criterion}, with the leading
coefficient normalization established above, completes the comparison.
\end{remark}

\section{Classification of Arthur pairs}
\label{sec:classification}

\subsection{The centralizer criterion}
Retain $J=Z_{\Gd}(\tau(W_F))$, $H=Z_J(\hd)$, and
$V=\mathfrak j(2)$ from Section~\ref{subsec:bounded-centralizer}.
A triple $(e_1^\vee,h_1^\vee,f_1^\vee)$ is \emph{homogeneous} if
its entries belong to $\mathfrak j(2)$, $\mathfrak j(0)$, and
$\mathfrak j(-2)$, respectively.

\begin{lemma}[{Graded Jacobson--Morozov Theorem, \cite{Vinberg1987}}]\label{lem:homogeneous-triple}
Every $e_1^\vee\in V$ belongs to a homogeneous triple. Any two
homogeneous triples with this nilpositive element are conjugate under
$Z_H(e_1^\vee)$.
\end{lemma}

Choose such a triple and denote the $\mathfrak{sl}_2$ algebra it generates by $\mathfrak s_1$. Set
\begin{equation}\label{eq:residual-centralizer}
 K_1=Z_J(\mathfrak s_1),\qquad
 \mathfrak c_1=\Lie K_1=Z_{\mathfrak j}(\mathfrak s_1),\qquad
 z^\vee=\hd-h_1^\vee.
\end{equation}
The group $K_1$ is a possibly disconnected, reductive group. Notice that $z^\vee\in\mathfrak c_1$ and it is semisimple since $\hd$ and $h_1^\vee$ commute.

\begin{proposition}[Centralizer criterion]\label{thm:centralizer-criterion}
The orbit $C=H e_1^\vee$ is Arthur if and only if $z^\vee$ is a
Dynkin element in $\mathfrak c_1$. This condition is independent of
the chosen representative and homogeneous triple. When it holds,
choose a triple $(e_2^\vee,z^\vee,f_2^\vee)$ in $\mathfrak c_1$.
Then $(C,H e_2^\vee)$ is an Arthur pair.
\end{proposition}
\begin{proof}
If $C$ is Arthur, conjugating by $H$, we may assume that the nilpositive element of a Lie triple for $C$ is $e_1^\vee$. Lemma~\ref{lem:homogeneous-triple} then
allows us to conjugate this Lie triple to a chosen one while fixing
$e_1^\vee$ and $\hd$. Then the second triple lies in $\mathfrak c_1$ and
has neutral element $\hd-h_1^\vee$.

Conversely, a triple in $\mathfrak c_1$ with neutral element $z^\vee$
commutes with the first triple. Since $\hd=h_1^\vee+z^\vee$,
its nilpositive and nilnegative elements have degrees $2$ and $-2$
for $\hd$. The sum is therefore a triple with neutral element $\hd$,
and the two triples define an Arthur pair.  Finally, changing the
homogeneous triple conjugates both $\mathfrak c_1$ and $z^\vee$ by
$Z_H(e_1^\vee)$, which proves independence of the choices.
\end{proof}

Notice that the centralizer test is in the full reductive algebra $\mathfrak c_1$. In
particular, the projection of $z^\vee$ to the center of $\mathfrak c_1$ must be zero.

The following is the orbit formulation of the classical injectivity of
$[\psi]\mapsto[\phi_\psi]$; see \cite[\S4, p.~25]{Arthur1989}
and, for the $p$-adic formulation, \cite[Lemma~3.6.1]{CFMMX}.

\begin{corollary}\label{cor:arthur-orbit-parameter}
The map
\begin{equation}
 \mathscr A_{\tau,\varphi_\Delta}
 \longrightarrow\{\text{Arthur orbits at }\lambda\},\qquad
 [\Psi]\longmapsto C_{\tau\Psi},
\end{equation}
is a bijection.
\end{corollary}

\subsection{Classical groups}\label{subsec:classical-arthur-pairs}
For the remainder of this section, assume $G$ is split and the
bounded Weil part is trivial. Then $J=\Gd$, $\mathfrak j=\gd$,
$H=Z_{\Gd}(\hd)$, and $V=\gd(2)$.

We give an explicit test for an orbit to be Arthur, then enumerate
all Arthur pairs with a prescribed diagonal using multiplicity arrays.
For special orthogonal groups, we also describe the splitting of
orthogonal conjugacy classes and the partner operation on these classes.
Throughout, $W$ is the defining module and
$L_d=\operatorname{Sym}^{d-1}\CC^2$, of dimension $d$.

\subsubsection{Testing an Arthur orbit}
For the first Lie triple, let $U_d=\operatorname{Hom}_{\SL_2}(L_d,W)$. Decompose 
\begin{equation}\label{eq:first-triple-decomposition}
 W=\bigoplus_{d\geq1}L_d\otimes U_d.
\end{equation}
Thus $\dim U_d$ is the multiplicity of a Jordan block of size $d$
for $e_1^\vee$. In the orthogonal or symplectic cases, write
$\epsilon=1$ or $-1$ for the sign of the form on $W$.
The invariant form on $L_d$ has sign $(-1)^{d-1}$.
Schur's lemma makes the summands in
\eqref{eq:first-triple-decomposition} orthogonal and induces a
nondegenerate form on $U_d$ of sign
\begin{equation}\label{eq:multiplicity-form}
 \epsilon_d=\epsilon(-1)^{d-1}.
\end{equation}
Every endomorphism commuting with the first triple acts on its
$d$-isotypic summand as $1\otimes X_d$.
The trace or form condition gives
\begin{equation}\label{eq:classical-centralizers}
\mathfrak c_1=
\begin{cases}
 \displaystyle\bigoplus_d\mathfrak{gl}(U_d),
 &\gd=\mathfrak{gl}(W),\\[6pt]
 \displaystyle\left\{(X_d)\in\bigoplus_d\mathfrak{gl}(U_d):
                  \sum_d d\,\operatorname{tr}(X_d)=0\right\},
 &\gd=\mathfrak{sl}(W),\\[6pt]
 \displaystyle\bigoplus_{d\text{ odd}}\mathfrak{so}(U_d)
 \oplus\bigoplus_{d\text{ even}}\mathfrak{sp}(U_d),
 &\gd=\mathfrak{so}(W),\\[6pt]
 \displaystyle\bigoplus_{d\text{ odd}}\mathfrak{sp}(U_d)
 \oplus\bigoplus_{d\text{ even}}\mathfrak{so}(U_d),
 &\gd=\mathfrak{sp}(W).
\end{cases}
\end{equation}
A symplectic
multiplicity space always has even dimension. These are the usual
centralizer descriptions; see \cite[Ch.~6]{CM}.

The element $z^\vee=\hd-h_1^\vee$ commutes with the first triple,
so its restriction to $L_d\otimes U_d$ has the form
$$
 z^\vee|_{L_d\otimes U_d}=1\otimes z_d.
$$
Each $z_d$ is semisimple, in the orthogonal and symplectic cases
it acts as a skew-adjoint operator with respect to the form on the multiplicity space.
Its spectrum gives the grading information needed in addition
to the Jordan partition of $e_1^\vee$.

\begin{theorem}[Classical test]\label{thm:classical-test}
The orbit $H e_1^\vee$ is Arthur if and only if, for every $d$,
the eigenvalues (with multiplicity) of  $z_d$ form a union of the strings
\begin{equation}\label{eq:dynkin-strings}
 b-1,b-3,\ldots,1-b\qquad(b\geq1),
\end{equation}
with multiplicities $n_{d,b}\in\ZZ_{\geq0}$ satisfying:
\begin{enumerate}
\item for a general linear factor, there is no additional restriction;
\item for an orthogonal factor, $n_{d,b}$ is even whenever $b$ is even;
\item for a symplectic factor, $n_{d,b}$ is even whenever $b$ is odd.
\end{enumerate}
For $\gd=\mathfrak{sl}(W)$, the general linear string test applies
to each $U_d$ separately. In particular, it requires
$\operatorname{tr}(z_d)=0$ for every $d$, beyond the weighted
trace condition in \eqref{eq:classical-centralizers}.

Equivalently, let $\mu_d(j)$ denote the multiplicity of $j$ in
$z_d$. Its eigenvalues are integers, since $\hd$ and $h_1^\vee$
commute and have integral weights on $W$. The test is
$\mu_d(j)=\mu_d(-j)$, together with nonnegativity and the above
parity conditions for
\begin{equation}\label{eq:string-multiplicities}
 n_{d,b}=\mu_d(b-1)-\mu_d(b+1)\qquad(b\geq1).
\end{equation}
These multiplicities are uniquely determined.
\end{theorem}
\begin{proof}
By Proposition~\ref{thm:centralizer-criterion}, it suffices to test whether
$z^\vee$ is a Dynkin element in $\mathfrak c_1$.
The classical classification of $\mathfrak{sl}_2$-modules with
invariant forms gives exactly the stated weight strings and parity
conditions on each $U_d$; see \cite[Ch.~5]{CM}.
Conversely, these conditions construct a triple whose neutral element
has the spectrum of $z_d$. Semisimple elements with this spectrum
are conjugate under the general linear or full isometry group, so
we may take the neutral element to be $z_d$ itself.
The resulting triples assemble in $\mathfrak c_1$; in the special
linear case, each has trace zero on every $U_d$.

Finally, for $j\geq0$, the weight strings give
$$
 \mu_d(j)=\sum_{\ell\geq0}n_{d,j+1+2\ell}.
$$
Taking successive differences gives
\eqref{eq:string-multiplicities}; conversely, telescoping recovers
$\mu_d(j)$, and symmetry recovers the negative weights.
This also proves uniqueness of the multiplicities.
\end{proof}

\subsubsection{Enumerating Arthur pairs}
For an Arthur orbit, decomposing each multiplicity space $U_d$ with respect to the second, partner Lie triple gives
\begin{equation}\label{eq:two-factor-classification}
 U_d=\bigoplus_b L_b\otimes E_{d,b},\qquad
 W=\bigoplus_{d,b}L_d\otimes L_b\otimes E_{d,b},\qquad
 \dim E_{d,b}=n_{d,b}.
\end{equation}
The two $\SL_2$ factors act on $L_d$ and $L_b$, respectively.
In the orthogonal or symplectic case, the form on $E_{d,b}$ has sign
\begin{equation}\label{eq:two-factor-sign}
 \epsilon(-1)^{d+b-2}.
\end{equation}
Hence $n_{d,b}$ must be even when $d+b$ is odd in the orthogonal
case, and when $d+b$ is even in the symplectic case. This recovers
the restrictions in Theorem~\ref{thm:classical-test}.

To enumerate the arrays for the prescribed diagonal, write
$W|_{\varphi_\Delta}=\bigoplus_{k\geq1}L_k^{\oplus m_k}$.
The Clebsch--Gordan decomposition of $L_d\otimes L_b$ gives
\begin{equation}\label{eq:diagonal-multiplicities}
 m_k=\sum_{d,b\geq1}n_{d,b}\,\eta_k(d,b),\qquad
 \eta_k(d,b)=
 \begin{cases}
 1,&k\in\{d+b-1,d+b-3,\ldots,|d-b|+1\},\\
 0,&\text{otherwise}.
 \end{cases}
\end{equation}
This is a finite system: if $k_{\max}=\max\{k:m_k\ne0\}$,
then $n_{d,b}>0$ forces $d+b-1\leq k_{\max}$, since its
largest diagonal summand cannot cancel against other nonnegative
contributions. Taking dimensions also gives
$\sum_{d,b}d\,b\,n_{d,b}=\dim W$, which bounds every entry.

\begin{proposition}[Multiplicity arrays]
\label{prop:classical-completeness}
For $\Gd=\mathrm{GL}(W)$, $\SL(W)$, or $\mathrm{Sp}(W)$,
the nonnegative arrays $(n_{d,b})$ satisfying
\eqref{eq:diagonal-multiplicities} and the parity
conditions parametrize the Arthur pairs for the fixed diagonal.

For $\Gd=\mathrm{SO}(W)$, the arrays classify the Arthur pairs up to conjugacy by $\mathrm{O}(W)$; Proposition~\ref{prop:orthogonal-refinement} gives the refinement
for the fixed diagonal under $\mathrm{SO}(W)$-conjugation.

On arrays, the partner of an Arthur orbit is given by transposition,
$n^t_{d,b}=n_{b,d}$. The action on special orthogonal classes is described in
Proposition~\ref{prop:orthogonal-refinement}.
\end{proposition}

\begin{proof}
The decomposition \eqref{eq:two-factor-classification} and uniqueness
of the multiplicity-space forms \eqref{eq:two-factor-sign} classify
the homomorphisms up to general linear, symplectic, or orthogonal
conjugacy. Their determinants are trivial, and a general linear
conjugator can be made special linear by multiplying it by a scalar.
By Clebsch--Gordan, \eqref{eq:diagonal-multiplicities} selects the
prescribed diagonal, up to the same conjugacy; the special orthogonal
refinement is given below. Corollary~\ref{cor:arthur-orbit-parameter}
then identifies these classes with Arthur pairs. Exchanging the two
factors transposes the array.
\end{proof}

\subsubsection{The special orthogonal refinement}
An admissible array determines an $\mathrm O(W)$-class. To obtain
Arthur pairs for the fixed diagonal in $\mathrm{SO}(W)$, we must
resolve its splitting and the action of factor exchange.

\begin{proposition}[Special orthogonal classes]
\label{prop:orthogonal-refinement}
Let $(n_{d,b})$ be an admissible array for an orthogonal $W$.
\begin{enumerate}
\item Its $\mathrm O(W)$-class splits into two $\mathrm{SO}(W)$-classes
if and only if
\begin{equation}\label{eq:orthogonal-splitting}
 n_{d,b}=0\quad\text{for all odd }d,b.
\end{equation}
Otherwise it gives one $\mathrm{SO}(W)$-class.
\item If the diagonal partition is very even, exactly one of the
array's two classes lies over the prescribed diagonal class.
Otherwise all its classes lie over that diagonal class.
\item Factor exchange sends the classes for $n$ to those for $n^t$.
If $n=n^t$ and \eqref{eq:orthogonal-splitting} holds, it preserves
the two classes individually when
$$
 (-1)^{\sum_{d\equiv2\!\!\pmod4}n_{d,d}}=1,
$$
and interchanges them when this sign is $-1$. For a very-even
diagonal, a self-transpose array gives a fixed point of the partner operation.
\end{enumerate}
\end{proposition}
\begin{proof}
An orthogonal class splits precisely when its centralizer contains
no element of determinant $-1$. By
\eqref{eq:two-factor-classification}, a centralizer element has
$$
 \det(g|_W)=\prod_{d,b}\det(g_{d,b})^{db}.
$$
This can be $-1$ exactly when an odd--odd block occurs: its symmetric
multiplicity space admits a reflection. This proves (1).

For (2), the diagonal orthogonal class splits precisely when its
partition is very even; see \cite[Ch.~5]{CM}.
In that case, \eqref{eq:diagonal-multiplicities} forces $d+b$ odd
on every nonzero block, so the array also splits. A determinant-$-1$
conjugation exchanges both pairs of classes, leaving exactly one
array class over each diagonal class. Otherwise the diagonal has
one special orthogonal class, over which all array classes lie.

For (3), factor exchange always maps the classes for $n$ bijectively
to those for $n^t$. Suppose $n=n^t$ and the array splits. Choose
isometries between paired multiplicity spaces and implement exchange
by tensor flips. Swapping the blocks with indices $d<b$ contributes
$(-1)^{db n_{d,b}}$; on $L_d\otimes L_d$, the flip is $-1$ on
$\bigwedge^2 L_d$. Hence the resulting isometry $q$ has
$$
 \det q
 =(-1)^{\sum_{d<b}db n_{d,b}+\sum_d n_{d,d}d(d-1)/2}
 =(-1)^{\sum_{d\equiv2\!\!\pmod4}n_{d,d}},
$$
since every occupied block has $db$ even. This sign is independent
of the choices, as centralizer elements have determinant one;
it determines whether the two classes are fixed or exchanged.
For a very-even diagonal there are no diagonal blocks, so the
selected class is fixed.
\end{proof}

\begin{example}
In $\mathrm{SO}_{12}$ the module
$W=(L_2\otimes L_2)^{\oplus3}$ has $n_{2,2}=3$ and diagonal
partition $(3^3,1^3)$. The array is self-transpose and satisfies
\eqref{eq:orthogonal-splitting}; both special orthogonal classes
belong to the same diagonal fiber. The exchange determinant is
$(-1)^3=-1$, so the two Arthur pairs are exchanged by the
partner operation.
\end{example}

\subsection{Exceptional groups}\label{subsec:exceptional-classification}
We now take $\gd$ of type $G_2$, $F_4$, $E_6$, $E_7$, or $E_8$.
We classify Arthur pairs for all Dynkin gradings in these types,
including the zero grading. For a fixed $\hd$, put
$$
 \mathcal C_{\hd}=H\backslash V,\qquad
 \mathcal A_{\hd}=
       \{C\in\mathcal C_{\hd}:C\text{ is Arthur}\}.
$$
By Corollary~\ref{cor:arthur-orbit-parameter}, the set
$\mathcal A_{\hd}$ also indexes the Arthur parameters in the
prescribed diagonal fibre. The involution on this set is the partner
operation of Lemma \ref{lem:cunningham}. 
Write $N_{\hd}=|\mathcal C_{\hd}|$ and
$A_{\hd}=|\mathcal A_{\hd}|$ for their cardinalities.
We first describe an exhaustive orbit enumeration valid for all
$\hd$, then the residual test, and finally the complete numerical
classification.

\medskip

The orbit-enumeration algorithms of Vinberg and de Graaf apply to
the $1$-component of a $\mathbb Z$-graded semisimple Lie
algebra; see \cite{Vinberg1979,dGtheta,SLA}. The following reduction
places $V=\gd(2)$ in that setting for every $\hd$.

\begin{proposition}\label{prop:exceptional-regrading}
Define the grading cocharacter by
$\gamma(t)=\varphi_\Delta(\operatorname{diag}(t,t^{-1}))$, and put
$$
 L^\vee=Z_{\Gd}(\gamma(-1))^\circ,\qquad
 \mathfrak l^\vee=\Lie L^\vee
                       =\bigoplus_{j\in\ZZ}\gd(2j),
 \qquad
 \mathfrak l_{\mathrm{der}}^\vee
                       =[\mathfrak l^\vee,\mathfrak l^\vee].
$$
The semisimple algebra $\mathfrak l_{\mathrm{der}}^\vee$ has the $\mathbb Z$-grading
$$
 \mathfrak l_{\mathrm{der}}^\vee(j)
       =\mathfrak l_{\mathrm{der}}^\vee\cap\gd(2j),
             \qquad\mathfrak l_{\mathrm{der}}^\vee(1)=V.
$$
Its $0$-component group
$Z_{(L^\vee)_{\mathrm{der}}}(\hd)$ has the same orbits on $V$
as $H$. Thus the $1$-component orbit enumeration in
$\mathfrak l_{\mathrm{der}}^\vee$ enumerates all $H$-orbits in $V$.
\end{proposition}
\begin{proof}
On $\gd(k)$ the adjoint action of $\gamma(-1)$ is multiplication
by $(-1)^k$, so its fixed point algebra is 
$\mathfrak l^\vee$. The centralizer of a semisimple group element
has reductive identity component; hence $L^\vee$ is reductive.
The center $\mathfrak z(\mathfrak l^\vee)$ commutes with
$\hd\in\mathfrak l^\vee$, lies in $\gd(0)$, and acts trivially
on $V\subset\mathfrak l^\vee$. For $x\in V$,
$$
                 x=\tfrac12[\hd,x]\in
                         [\mathfrak l^\vee,\mathfrak l^\vee].
$$
This proves $\mathfrak l_{\mathrm{der}}^\vee(1)=V$ and the
integral grading assertion.

The connected group $H$ centralizes $\gamma(-1)$, so it lies in
$L^\vee$ and equals $Z_{L^\vee}(\hd)$.
The central-product isogeny
$Z(L^\vee)^\circ\times(L^\vee)_{\mathrm{der}}\to L^\vee$ gives
$$
              H=Z(L^\vee)^\circ
                       Z_{(L^\vee)_{\mathrm{der}}}(\hd).
$$
Indeed, the central factor fixes $\hd$, so a product centralizes
$\hd$ precisely when its derived-group factor does. The central
factor acts trivially on $V$, proving equality of the orbit sets.
To identify the degree-zero group, use the reductive decomposition
$$
 \mathfrak l^\vee=\mathfrak z(\mathfrak l^\vee)
                     \oplus\mathfrak l_{\mathrm{der}}^\vee,
 \qquad
 \hd=h_{\mathrm{cen}}^\vee+h_{\mathrm{der}}^\vee.
$$
Here the last two elements are the components of $\hd$ in these
summands. The adjoint action of $\hd$ on
$\mathfrak l_{\mathrm{der}}^\vee$ is that of the semisimple
element $h_{\mathrm{der}}^\vee$. The grading element is
$h_{\mathrm{der}}^\vee/2$. Its centralizer in
$(L^\vee)_{\mathrm{der}}$ is connected and is the degree-zero
group of the regrading. The integral orbit-enumeration algorithm
therefore applies \cite[\S4]{dGtheta}, \cite[\S5.1]{SLA}.
\end{proof}

For $\hd=0$, $V=0$ and there is one orbit, which is Arthur and
fixed by the partner operation. Otherwise the proposition gives
a finite exhaustive enumeration, using the same reduction for every
Dynkin grading.

\medskip

For an orbit $C=H e_1^\vee\subset V$, choose a homogeneous triple
$(e_1^\vee,h_1^\vee,f_1^\vee)$ and retain
$\mathfrak c_1=Z_{\gd}(e_1^\vee,h_1^\vee,f_1^\vee)$ and
$z^\vee=\hd-h_1^\vee$ from \eqref{eq:residual-centralizer}.
For the computations, we use the following polynomial formulation of
Proposition~\ref{thm:centralizer-criterion}. An explicit second triple
certifies that $C$ is Arthur, while a Gr\"obner basis containing $1$
certifies that it is not. These calculations can be performed in
\textsc{GAP}/\textsc{SLA} and \textsc{Singular} \cite{GAP,SLA,Singular}.

\begin{lemma}[Polynomial criterion]\label{lem:exceptional-polynomial-criterion}
Choose bases $(x_i)_{i=1}^r$ and $(y_j)_{j=1}^s$ of the
$2$- and $-2$-eigenspaces of $\operatorname{ad}(z^\vee)$ in
$\mathfrak c_1$. Let
$I_C\subset\CC[a_1,\ldots,a_r,b_1,\ldots,b_s]$ be generated by
the coordinate equations of
\begin{equation}\label{eq:exceptional-bilinear-test}
 \sum_{i,j}a_ib_j[x_i,y_j]=z^\vee.
\end{equation}
Then
\begin{equation}\label{eq:exceptional-ideal-test}
 C\text{ is Arthur}\quad\Longleftrightarrow\quad
 I_C\ne\CC[a_1,\ldots,a_r,b_1,\ldots,b_s].
\end{equation}
\end{lemma}
\begin{proof}
A solution gives a triple
$(\sum_i a_ix_i,z^\vee,\sum_j b_jy_j)$ in $\mathfrak c_1$,
and every triple with neutral element $z^\vee$ arises this way.
The assertion follows from Proposition~\ref{thm:centralizer-criterion}
and the Nullstellensatz. If the input is rational, the bases may
be chosen over $\mathbb Q$, and properness can be checked over
$\mathbb Q$ since field extension preserves proper ideals.
\end{proof}

\begin{theorem}\label{thm:exceptional-classification}
For each Dynkin grading of an exceptional simple Lie algebra, the
numbers of graded orbits, Arthur orbits, and $p$-fixed Arthur orbits
are given in Table~\ref{tab:exceptional-by-grading} in
Appendix~\ref{app:tables}. The rows are indexed by the diagonal
nilpotent orbit. An Arthur orbit $C$ is $p$-fixed if $p(C)=C$;
the remaining Arthur orbits occur in pairs $\{C,p(C)\}$.
\end{theorem}

\section{Special pieces and weak Arthur packets}
\label{sec:weak}

Fix a nilpotent orbit $\mathcal O^\vee\subset\gd$, a triple
$(e^\vee,\hd,f^\vee)$ with $e^\vee\in\mathcal O^\vee$, and its neutral
infinitesimal parameter $\lambda$. The bounded Weil part is trivial, and
$$
 H=Z_{\Gd}(\hd),\qquad V=\gd(2),\qquad
 C_{\mathrm{op}}=He^\vee,\qquad \Gd C_{\mathrm{op}}=\mathcal O^\vee.
$$
Here $C_{\mathrm{op}}$ is open in $V$, so
$V\subseteq\overline{\mathcal O^\vee}$. Let $d$ denote
Lusztig--Spaltenstein duality from nilpotent orbits in $\gd$ to
nilpotent orbits in the complex Lie algebra of $G$.

For a special nilpotent orbit $\mathcal O_{\mathrm{sp}}^\vee$, its special
piece is \cite[\S2.2]{AcharSage}
$$
 \mathscr S(\mathcal O_{\mathrm{sp}}^\vee)
 =\overline{\mathcal O_{\mathrm{sp}}^\vee}
   \setminus\!\bigcup_{\substack{\mathcal O'^\vee<\mathcal O_{\mathrm{sp}}^\vee\\
                                \mathcal O'^\vee\ \mathrm{special}}}
                  \overline{\mathcal O'^\vee}.
$$
Write $\mathscr S(\mathcal O^\vee)$ for the unique special piece containing
$\mathcal O^\vee$. Its \emph{relative special piece} is
\begin{equation}\label{eq:relative-special-piece}
 \mathscr S_{\mathrm{rel}}(\mathcal O^\vee)
 :=\overline{\mathcal O^\vee}\cap\mathscr S(\mathcal O^\vee)
 =\!\bigcup_{\substack{\mathcal O_1^\vee\leq\mathcal O^\vee\\
               d(\mathcal O_1^\vee)=d(\mathcal O^\vee)}}\mathcal O_1^\vee.
\end{equation}

Set
\begin{equation}\label{eq:relative-graded-locus}
 V_{\mathrm{rel}}=V\cap\mathscr S_{\mathrm{rel}}(\mathcal O^\vee),
 \qquad \mathscr C_{\mathrm{rel}}=H\backslash V_{\mathrm{rel}}.
\end{equation}
Weak Arthur packets in the $p$-adic setting were introduced by
Ciubotaru--Mason-Brown--Okada
\cite{CMBOweak}. The weak packet attached to $\mathcal O^\vee$ is
\begin{equation}\label{eq:weak-wavefront-definition}
 \Pi_{\mathcal O^\vee}^{\mathrm{weak}}
 :=\{\pi_a:a\in\mathcal I_\lambda,\ 
          \mathrm{WF}_{\mathrm{alg}}(\pi_a)=\{d(\mathcal O^\vee)\}\}.
\end{equation}
We regard $\mathrm{WF}_{\mathrm{alg}}$ as the set of maximal wavefront
orbits. For this interpretation we assume $G$ is inner to split,
$\operatorname{char}F=0$, sufficiently large residual characteristic
\cite{CMBOWFII}, and
Hypotheses~\ref{hyp:enhanced-llc} and~\ref{hyp:aubert-compatibility}.
The geometric description \eqref{eq:weak-fourier-set} below is independent
of these representation-theoretic assumptions.

Write $\FT P_a\simeq P_{f(a)}$. Fourier--Aubert compatibility and the
unipotent wavefront formula \cite[Theorem~1.4.1(1)]{CMBOWFII} give
$$
 \operatorname{AZ}(\pi_a)=\pi_{f(a)},\qquad
 \mathrm{WF}_{\mathrm{alg}}(\pi_a)=\{d(\Gd C_{f(a)})\}.
$$
Since $\Gd C_{f(a)}\leq\mathcal O^\vee$ and $d$ reverses closure order,
$d(\mathcal O^\vee)\leq d(\Gd C_{f(a)})$. Hence
\begin{equation}\label{eq:weak-fourier-set}
 \begin{aligned}
 \Pi_{\mathcal O^\vee}^{\mathrm{weak}}
 &=\{\pi_a:d(\Gd C_{f(a)})=d(\mathcal O^\vee)\}\\
 &=\{\pi_a:C_{f(a)}\in\mathscr C_{\mathrm{rel}}\}.
 \end{aligned}
\end{equation}
On a chosen inner form, retain only relevant enhancements.

\subsection{The special-piece Arthur locus}
\label{subsec:weak-locus}
Since the special piece is open in the closure of its special orbit
and $V\subseteq\overline{\mathcal O^\vee}$, the intersection
$V_{\mathrm{rel}}$ is open in $V$. We study the saturation map
$$
 s:\mathscr C_{\mathrm{rel}}\longrightarrow
       \Gd\backslash\mathscr S_{\mathrm{rel}}(\mathcal O^\vee),
 \qquad C\longmapsto\Gd C.
$$

\begin{theorem}[Lifts on the relative special piece]
\label{thm:weak-injectivity}
For a neutral grading, every orbit in $\mathscr C_{\mathrm{rel}}$ is
Arthur. The saturation map
$$
 s:\mathscr C_{\mathrm{rel}}\longrightarrow
 \Gd\backslash\mathscr S_{\mathrm{rel}}(\mathcal O^\vee)
$$
has the following properties:
\begin{enumerate}
\item\label{item:weak-classical-lifts} In classical types, $s$ is a bijection.
\item\label{item:weak-exceptional-lifts} In exceptional types, $s$ is injective,
and is a bijection when $\mathcal O^\vee$ is even.
The orbits missing from its image are listed in
Table~\ref{tab:exceptional-relative-lifts} in Appendix~\ref{app:tables}.
\end{enumerate}
\end{theorem}
\begin{proof}
In type $A$, special pieces are singletons. The open orbit is Arthur,
and Lemma~\ref{lem:strict-saturation} excludes any other graded lift
of its ambient orbit.

For $\Gd=\mathrm{Sp}_{2n}(\CC)$ or $\mathrm{SO}_{2n+1}(\CC)$,
this is \cite[Proposition~1.3.1]{GO}, expressed in terms of
$H$-orbits via Section~\ref{sec:setup}.

For the even orthogonal case, let $\lambda$ be the
partition of $\mathcal O^\vee$ and $\mu$ that of an orbit in its
relative special piece. If $\mu$ is not very even, the proof is
analogous to \cite[Lemma~3.4.1 and Theorem~3.4.2]{GO}, using the
type-$D$ special-partition rules. We must additionally pass from
orthogonal to special-orthogonal conjugacy. In odd dimension this
is automatic by multiplying a conjugator by $-I$; in even dimension,
$-I$ has determinant $1$. Instead, the Arthur parameter constructed by that argument
has an odd-dimensional summand $L_a\boxtimes L_1$, in the notation
of Section~\ref{subsec:classical-arthur-pairs}. Acting by $-1$ on this
summand and by $1$ on its orthogonal complement centralizes both
the grading and the nilpotent element and corrects the determinant
of any orthogonal conjugator.

If $\mu$ is very even, it is special, so $\mu=\lambda$.
The other very-even orbit has the same dimension
and does not lie in the closure of the chosen orbit. Its relative
special piece is therefore a singleton, whose unique graded lift is
$C_{\mathrm{op}}$ by Lemma~\ref{lem:strict-saturation}; this orbit
is Arthur.

For exceptional types, the computed orbit lists show that every
orbit in $\mathscr C_{\mathrm{rel}}$ is Arthur and that their
saturations are distinct. We check this for every neutral grading
of the simply connected dual group, using CHEVIE \cite{CHEVIE}
to determine the special pieces. For even gradings, comparison with
\cite[Tables~1--3 and \S3.5]{JLSYspecial} shows that every orbit in
the special piece occurs, proving surjectivity. The data are provided
in the supplementary material.

The choice of isogeny type does not affect the result: central
isogenies identify the graded and ambient orbits and preserve their
closure relations. They also preserve the Arthur condition, since
homomorphisms from $\SL_2\times\SL_2$ lift through them.
\end{proof}

Surjectivity fails in exceptional type: for the $\widetilde A_1$ grading
in $G_2$,
$$
 A_1\subset\mathscr S_{\mathrm{rel}}(\widetilde A_1),
 \qquad A_1\cap V=\varnothing.
$$

We next compare the graded and ambient component groups, using
the notation $A_K(u)=\pi_0Z_K(u)$ of
Section~\ref{subsec:microlocal-data} for either acting group $K$.

\begin{lemma}[The component-group map]\label{lem:even-component-map}
For any commuting pair of triples with sum of neutral elements $\hd$,
inclusion induces an injection
$$
 \iota:A_H(e_1^\vee)\hookrightarrow A_{\Gd}(e_1^\vee).
$$
If $Q=Z_{\Gd}(e_1^\vee,h_1^\vee,f_1^\vee)$, its image is the
stabilizer in $\pi_0Q$ of the $Q^\circ$-conjugacy class of $h_2^\vee$.
In particular, $\iota$ is an isomorphism if $h_2^\vee=0$, if $Q$
is connected, or if $h_2^\vee$ is a principal Dynkin element in
the derived algebra of $\Lie Q$.
\end{lemma}
\begin{proof}
The graded Jacobson--Morozov argument of
\cite[Proposition~3.3]{Lusztig1995Graded}, together with the connected
unipotent radical of $Z_{\Gd}(e_1^\vee)$, gives
$$
 A_H(e_1^\vee)\simeq\pi_0Z_Q(h_2^\vee),\qquad
 A_{\Gd}(e_1^\vee)\simeq\pi_0Q.
$$
Since $Z_{Q^\circ}(h_2^\vee)$ is connected, inclusion induces an
injection whose image consists precisely of the components preserving
the $Q^\circ$-conjugacy class of $h_2^\vee$. The surjectivity criteria
follow; for the principal case, automorphisms preserve the principal
Dynkin class. The case $h_2^\vee=0$ is also
\cite[Proposition~4.2(c)]{Lusztig1995Graded}.
\end{proof}

\begin{theorem}[Equality of component groups]\label{thm:even-ambient-components}
Assume that $\gd$ is simple and $\mathcal O^\vee$ is even.
For every $e_1^\vee\in V_{\mathrm{rel}}$, inclusion induces an isomorphism
\begin{equation}\label{eq:even-ambient-components}
 A_H(e_1^\vee)\xrightarrow{\sim}
                         A_{\Gd}(e_1^\vee).
\end{equation}
\end{theorem}
\begin{proof}
By Theorem~\ref{thm:weak-injectivity}, choose commuting triples with
first nilpositive element $e_1^\vee$ and $h_1^\vee+h_2^\vee=\hd$.
Put $Q=Z_{\Gd}(e_1^\vee,h_1^\vee,f_1^\vee)$. By
Lemma~\ref{lem:even-component-map}, it suffices to show that every
component of $Q$ meets $Z_Q(h_2^\vee)$. For classical matrix groups,
the lifts of \cite[\S3.4.1, Lemma~3.4.1]{GO}, with the type-$D$
adaptation in Theorem~\ref{thm:weak-injectivity}, have nontrivial
second factor only on alternating multiplicity spaces. The second
triple therefore fixes the orthogonal multiplicity spaces, whose
determinant signs determine the components of $Q$
\cite[\S1.2.2]{JLSYspecial}. This proves surjectivity. In type $A$
and the singleton special-piece cases, take $h_2^\vee=0$.

The image criterion in Lemma~\ref{lem:even-component-map} is
unchanged by central isogeny: the connected conjugacy classes of
$h_2^\vee$ are identified, and any additional central components
act trivially. Thus the classical argument applies to every isogeny
type, and we may treat exceptional groups in adjoint form.

For exceptional types, the computed lifts have $h_2^\vee=0$,
connected $Q$, or principal $h_2^\vee$ in $\Lie Q$.
The latter cases are verified in Table~\ref{tab:even-component-check};
the first two require no further calculation.
Lemma~\ref{lem:even-component-map} therefore proves the assertion.
The zero grading is immediate.
\end{proof}

\subsection{Parametrisation for even dual orbits}
\label{subsec:weak-even-parametrisation}
Throughout this subsection, assume that $\gd$ is simple and that
$\mathcal O^\vee$ is even, meaning that $\gd(j)=0$ for odd $j$.
An even orbit is Richardson for its Jacobson--Morozov parabolic and
is special; see \cite{CM}. Thus $\mathscr S_{\mathrm{rel}}(\mathcal O^\vee)
=\mathscr S(\mathcal O^\vee)$. By Theorem~\ref{thm:weak-injectivity},
every orbit in this special piece has a unique graded lift, which is
Arthur, and Theorem~\ref{thm:even-ambient-components} identifies
its graded and ambient component groups. We use these results to
parametrise $\Pi_{\mathcal O^\vee}^{\mathrm{weak}}$.
We retain the representation-theoretic hypotheses of
Section~\ref{sec:weak} and its choice of enhanced correspondence.
The geometric assertions hold for every isogeny type of $\Gd$.

\subsubsection{Finite labels}

\begin{definition}[Canonical quotient]
For an ambient nilpotent orbit $\mathcal U^\vee=\Gd u^\vee$,
let $\mathcal F_{\mathcal U^\vee}$ be the Weyl-group family containing
the Springer representation of $(\mathcal U^\vee,1)$, and put
$$
 \mathcal J(\mathcal U^\vee)
 =\{\rho\in\Irr A_{\Gd}(u^\vee):
    \operatorname{Spr}(\mathcal U^\vee,\rho)
                 \in\mathcal F_{\mathcal U^\vee}\},\qquad
 \bar A(\mathcal U^\vee)
 =A_{\Gd}(u^\vee)\big/\!\bigcap_{\rho\in\mathcal J(\mathcal U^\vee)}
                                                   \ker\rho.
$$
Here $\operatorname{Spr}(\mathcal U^\vee,\rho)$ denotes the ordinary
Springer representation; the set $\mathcal J(\mathcal U^\vee)$
includes only pairs in the ordinary Springer correspondence. The
image of $Z(\Gd)$ is trivial in $\bar A(\mathcal U^\vee)$, so this
quotient agrees with that for the adjoint group; see \cite[\S2.1]{AcharSage}.
\end{definition}

Put $\Gamma=\bar A(\mathcal O^\vee)$. Its canonical Coxeter structure
is understood up to conjugacy.

\begin{proposition}[{Special-piece labels, \cite[Theorem~2.1]{AcharSage}}]
\label{prop:even-special-piece-labels}
There is a Coxeter subgroup
$\Gamma_{\mathrm{sp}}\subseteq\Gamma$ and a bijection
\begin{equation}\label{eq:even-special-piece-index}
 \operatorname{Cl}(\Gamma_{\mathrm{sp}})
       \xrightarrow{\sim}
          \Gd\backslash\mathscr S(\mathcal O^\vee),
 \qquad [x]\longmapsto\mathcal O_x^\vee,
 \qquad \mathcal O_1^\vee=\mathcal O^\vee.
\end{equation}
Here $\operatorname{Cl}$ denotes the set of conjugacy classes.
Distinct classes in $\Gamma_{\mathrm{sp}}$ remain distinct in $\Gamma$. The subgroup
$\Gamma_{\mathrm{sp}}$ is called the special-piece group.
\end{proposition}
The group $\Gamma_{\mathrm{sp}}$ can be smaller than $\Gamma$, even
when the special piece has only one orbit.

For $x\in\Gamma_{\mathrm{sp}}$, let $P_x$ be the smallest Coxeter
parabolic subgroup of $\Gamma$ containing $x$; equivalently, it is
the intersection of all such parabolics. Its conjugacy class is the
\emph{Coxeter support} of $[x]$. Here $\Gamma$ is a product of symmetric groups, and the classes
$[P_x]$ distinguish the conjugacy classes in \eqref{eq:even-special-piece-index}.
Write $\mathfrak S_m$ for the symmetric group on $m$ letters.
For $\Gamma=\Gamma_{\mathrm{sp}}=\mathfrak S_m$, a permutation of cycle
type $(\ell_1,\ldots,\ell_r)$ has
$P_x\simeq\mathfrak S_{\ell_1}\times\cdots\times\mathfrak S_{\ell_r}$.
For an elementary abelian special-piece group, $P_x$ is generated
by the Coxeter generators occurring in $x$.

\medskip

\begin{remark}[Explicit special-piece groups]
For classical types, $\Gamma_{\mathrm{sp}}$ is an elementary
abelian $2$-group; its generators and the corresponding partition
moves are described in \cite[\S1.3.1, (2)--(3)]{JLSYspecial}.
For exceptional even orbits, the nontrivial special-piece groups
are listed below \cite[Table~1]{JLSYspecial}.
$$
\begin{array}{c|l|c}
 \text{type}&\mathcal O^\vee&\Gamma_{\mathrm{sp}}\\\hline
 G_2&G_2(a_1)&\mathfrak S_3\\
 F_4&F_4(a_3)&\mathfrak S_4\\
 E_6&A_2,\ E_6(a_3)&\mathfrak S_2\\
    &D_4(a_1)&\mathfrak S_3\\
 E_7&A_2,\ E_6(a_3),\ E_7(a_3)&\mathfrak S_2\\
    &D_4(a_1),\ E_7(a_5)&\mathfrak S_3\\
 E_8&A_2,\ 2A_2,\ D_4(a_1)+A_2,\ E_6(a_3)&\mathfrak S_2\\
    &E_8(b_6),\ E_8(a_5),\ E_8(a_3)&\mathfrak S_2\\
    &D_4(a_1),\ E_8(b_5)&\mathfrak S_3\\
    &E_8(a_7)&\mathfrak S_5
\end{array}
$$
For all other even exceptional orbits $\Gamma_{\mathrm{sp}}=1$.
\end{remark}

For $[x]\in\operatorname{Cl}(\Gamma_{\mathrm{sp}})$, put
\begin{equation}\label{eq:even-two-orbit-labels}
 E_x=\mathcal O_x^\vee\cap V,\qquad C_x=p(E_x),
\end{equation}
and choose $e_{1,x}^\vee\in E_x$.
By Theorem~\ref{thm:weak-injectivity}, $E_x$ is a single $H$-orbit.

\subsubsection{Packet members and canonical enhancements}

Use \eqref{eq:even-ambient-components} to regard an irreducible
representation of $A_{\Gd}(e_{1,x}^\vee)$ as an enhancement of $E_x$.
The notation $\pi_{(E_x,\rho)}$ refers to the enhanced correspondence
fixed in Section~\ref{sec:setup}.

\begin{theorem}[Parametrisation of the weak packet]\label{thm:even-weak-parametrisation}
Assume that $\gd$ is simple and that $\mathcal O^\vee$ is even.
Taking representations over all pure inner twists, with the twist
included in the label, there is a bijection
\begin{equation}\label{eq:even-weak-parametrisation}
 \bigsqcup_{[x]\in\operatorname{Cl}(\Gamma_{\mathrm{sp}})}
       \Irr A_{\Gd}(e_{1,x}^\vee)
 \xrightarrow{\sim}\Pi_{\mathcal O^\vee}^{\mathrm{weak}},
 \qquad ([x],\rho)\longmapsto
                       \operatorname{AZ}(\pi_{(E_x,\rho)}).
\end{equation}
It is canonical relative to the chosen enhanced correspondence.
On the inner twist $G^\delta(F)$, retain precisely those $\rho$
whose restriction along $Z(\Gd)\to A_{\Gd}(e_{1,x}^\vee)$ is
$\kappa_\delta\,\mathrm{id}_\rho$, where $\kappa_\delta$ is pulled
back from $\pi_0Z(\Gd)$ as in Section~\ref{sec:setup}. In particular,
$$
 |\Pi_{\mathcal O^\vee}^{\mathrm{weak}}|
   =\sum_{[x]\in\operatorname{Cl}(\Gamma_{\mathrm{sp}})}
                        |\Irr A_{\Gd}(e_{1,x}^\vee)|.
$$
\end{theorem}
\begin{proof}
By \eqref{eq:weak-fourier-set}, a representation belongs to the weak
packet exactly when the orbit of its Aubert dual lies in
$\mathscr C_{\mathrm{rel}}$. Theorem~\ref{thm:weak-injectivity}
and \eqref{eq:even-special-piece-index} identify those orbits with
the $E_x$. The enhanced correspondence gives exactly one irreducible
representation for each pair $(E_x,\rho)$, and
Theorem~\ref{thm:even-ambient-components} identifies its enhancement
group with the ambient one. Aubert duality is an involution on
irreducibles, so these labels remain distinct and exhaust the weak
packet. It preserves the underlying inner form. The stabilizer
isomorphism is induced by inclusion and commutes with the maps from
$Z(\Gd)$, proving the stated relevance condition.
\end{proof}

\begin{remark}
The fibres in \eqref{eq:even-weak-parametrisation} are disjoint:
$[x]$ records the ambient nilpotent orbit of the Aubert dual.
The Arthur packets based at $C_x=p(E_x)$ can overlap.
Section~\ref{subsec:weak-decomposition} proves that their union is
the weak packet. Thus $[x]$ indexes both disjoint enhancement fibres and possibly
overlapping Arthur packets.
\end{remark}

\begin{definition}[Canonical enhancement labels]
For $x\in\Gamma_{\mathrm{sp}}$ define
$$
 I_x^{\mathrm{can}}
   =Z_\Gamma(x)/(P_x\cap Z_\Gamma(x))
   \simeq N_\Gamma(P_x)/P_x.
$$
Achar--Sage \cite[\S2.4 and Theorem~2.1(2)]{AcharSage} give
$\bar A(\mathcal O_x^\vee)\simeq I_x^{\mathrm{can}}$.
Recall Lusztig's finite set
$$
 \mathcal M(\Gamma)
 =\{(x,\tau):x\in\Gamma,\ \tau\in\Irr Z_\Gamma(x)\}/\Gamma,
$$
where $\Gamma$ acts by simultaneous conjugation. Define the subset
\begin{equation}\label{eq:even-restricted-M-set}
 \mathcal M_{\mathrm{sp}}^0(\Gamma)
 =\left\{[(x,\tau)]\in\mathcal M(\Gamma):
 \begin{array}{l}
 [x]\cap\Gamma_{\mathrm{sp}}\ne\varnothing,\\
 P_x\cap Z_\Gamma(x)\subseteq\ker\tau
 \end{array}\right\}.
\end{equation}
The parabolic $P_x$ is defined in $\Gamma$, so these conditions are
conjugacy invariant. Inflation gives an identification with
$\bigsqcup_{[x]\in\operatorname{Cl}(\Gamma_{\mathrm{sp}})}
\Irr I_x^{\mathrm{can}}$.
\end{definition}

\begin{corollary}[Canonical enhancements]\label{cor:even-canonical-parametrisation}
Let $\Pi_{\mathcal O^\vee}^{\mathrm{weak,can}}$ consist of the
representations whose labels in \eqref{eq:even-weak-parametrisation}
have $\rho$ factoring through $\bar A(\mathcal O_x^\vee)$. Then
$$
 \Pi_{\mathcal O^\vee}^{\mathrm{weak,can}}
       \xrightarrow{\sim}\mathcal M_{\mathrm{sp}}^0(\Gamma).
$$
This parametrises the whole weak packet if every map
$A_{\Gd}(e_{1,x}^\vee)\to\bar A(\mathcal O_x^\vee)$ is an
isomorphism. In general the full parametrisation is
\eqref{eq:even-weak-parametrisation}.
\end{corollary}
\begin{proof}
Factoring through the canonical quotient replaces each enhancement
fibre by $\Irr\bar A(\mathcal O_x^\vee)$.
The Achar--Sage identification and inflation to $Z_\Gamma(x)$ give
exactly the kernel condition in \eqref{eq:even-restricted-M-set}.
\end{proof}

\begin{example}[$G_2(a_1)$]
For $\mathcal O^\vee=G_2(a_1)$, both $\Gamma$ and
$\Gamma_{\mathrm{sp}}$ are $\mathfrak S_3$. The identity,
transposition, and three-cycle label respectively
$G_2(a_1),\widetilde A_1,A_1$. Their ambient component groups are
$\mathfrak S_3,1,1$, giving $3+1+1=5$ members of the weak packet.
All enhancements are canonical. The ordinary set
$\mathcal M(\mathfrak S_3)$ has $8$ elements; its centralizer groups
are replaced here by their effective quotients.

\end{example}

\begin{example}[$F_4(a_3)$]
For $F_4(a_3)$ the two finite groups are $\mathfrak S_4$.
Using the orbit numbering of Section~\ref{subsec:example-f4}, the
classification is
$$
\begin{array}{c|l|cc|c}
 \text{cycle type of }x&\mathcal O_x^\vee&E_x&C_x=p(E_x)
                            &A_{\Gd}(e_{1,x}^\vee)\\\hline
 1^4&F_4(a_3)&C_{12}&C_1&\mathfrak S_4\\
 2\,1^2&C_3(a_1)&C_{11}&C_2&\mathfrak S_2\\
 2^2&B_2&C_{10}&C_3&\mathfrak S_2\\
 3\,1&\widetilde A_2+A_1&C_9&C_5&1\\
 4&A_2+\widetilde A_1&C_8&C_8&1
\end{array}
$$
Here the $C_i$ use the orbit indices of Section~\ref{subsec:example-f4}.
The five enhancement fibres have sizes $5,2,2,1,1$, so the weak
packet has $11$ members. Again every enhancement is canonical,
whereas $|\mathcal M(\mathfrak S_4)|=21$.

\end{example}

\begin{example}[$D_4(a_1)$ in type $E_6$]
For $D_4(a_1)$ in the simply connected group of type $E_6$,
$\Gamma_{\mathrm{sp}}=\Gamma=\mathfrak S_3$ and the three
special-piece orbits are $D_4(a_1),A_3+A_1,2A_2+A_1$.
Their actual component groups are
$\mathfrak S_3,1,\mathbb Z/3\mathbb Z$, respectively.
Thus the full weak packet has $3+1+3=7$ members, in agreement with
\eqref{eq:e6-weak-example}. The last orbit has trivial canonical
quotient, so the canonical-enhancement subset has $3+1+1=5$ members.
Thus \eqref{eq:even-weak-parametrisation} includes enhancements
that do not factor through the canonical quotient.

\end{example}

\begin{example}[$A_2$ in type $F_4$]
The distinction is not confined to nontrivial central characters.
The even orbit $A_2$ in $F_4$ has a singleton special piece, ambient
component group $\mathfrak S_2$, and trivial canonical quotient.
Its weak packet therefore has two members, only one of which has a
canonical enhancement. This is also the orbit used in
Section~\ref{subsec:weak-decomposition} to distinguish canonical
enhancements from weak sphericity.
\end{example}

\subsection{Decomposition into microlocal Arthur packets}
\label{subsec:weak-decomposition}
Ciubotaru--Mason-Brown--Okada conjecture that weak Arthur packets
are unions of Arthur packets \cite[Conjecture~3.1.2]{CMBOweak}.
Their description in terms of Aubert duals
\cite[Corollary~3.1.1]{CMBOweak} takes the form
\eqref{eq:weak-fourier-set} in our notation. The following theorem
establishes the microlocal version of the conjecture and identifies
exactly which packets occur.

\begin{theorem}[Weak Arthur packet conjecture for microlocal packets]
\label{thm:weak-packet-decomposition}
Under the hypotheses of this section, for every neutral grading of
classical or exceptional type,
\begin{equation}\label{eq:weak-packet-decomposition}
 \Pi_{\mathcal O^\vee}^{\mathrm{weak}}
 =\bigcup_{C\in\mathscr C_{\mathrm{rel}}}\Pi_{p(C)}^\mic,
\end{equation}
and every packet in this union is a microlocal Arthur packet.
The equality holds on each pure inner twist. With packets taken
over all pure inner twists, a microlocal Arthur packet satisfies
$$
 \Pi_B^\mic\subseteq\Pi_{\mathcal O^\vee}^{\mathrm{weak}}
 \quad\Longleftrightarrow\quad p(B)\in\mathscr C_{\mathrm{rel}}.
$$
\end{theorem}
\begin{proof}
If $C\in\mathscr C_{\mathrm{rel}}$ and $\pi_a\in\Pi_{p(C)}^\mic$,
Fourier symmetry gives $m_C(P_{f(a)})>0$, hence
$C\leq C_{f(a)}$. Since
$\Gd C\leq\Gd C_{f(a)}\leq\mathcal O^\vee$, order reversal under
$d$ gives $d(\Gd C_{f(a)})=d(\mathcal O^\vee)$. Thus
\eqref{eq:weak-fourier-set} puts $\pi_a$ in the weak packet.
Conversely, for a member $\pi_a$ of the weak packet, put
$C=C_{f(a)}\in\mathscr C_{\mathrm{rel}}$. Then
$$
 m_{p(C)}(P_a)=m_C(P_{f(a)})=\dim\rho_{f(a)}>0.
$$
This proves the union formula, including its restriction to each
pure inner twist.

For the containment criterion, the reverse implication follows
from this formula. For the forward implication, put $D=p(B)$ and
choose $a$ with $P_{f(a)}=\IC(\overline D,\mathcal L_1)$.
Then $m_B(P_a)=1$, so containment and
\eqref{eq:weak-fourier-set} force $D\in\mathscr C_{\mathrm{rel}}$.
Finally, Theorem~\ref{thm:weak-injectivity} and factor exchange show
that all the packet bases $p(C)$ in the union are Arthur.
\end{proof}

Endoscopic versions for split classical groups are treated in
\cite{LiuLoWeak,GO}. The theorem above concerns microlocal packets;
comparison with endoscopic packets requires the corresponding
packet identification.

\subsection{Weak sphericity}
\label{subsec:weak-sphericity}
Following Gurevich--Okada \cite[Definition~1.2.1]{GO}, a representation
is \emph{weakly spherical} if it has a nonzero vector fixed by some
maximal compact subgroup. We ask whether every microlocal Arthur
packet contained in a weak packet has such a member. Let
$$
 \mathcal S_\lambda
 :=\{\pi_a:a\in\mathcal I_\lambda,\ \pi_a^K\ne0
       \text{ for some maximal compact }K\subset G^{\delta_a}(F)\}.
$$

For split adjoint exceptional groups, the same method reduces weak
sphericity of anti-tempered representations to a calculation in
Springer cohomology
\cite[Propositions~5.1.5 and~5.2.5, \S5.3]{GO}.
Fix an Iwahori subgroup $I$. For each maximal compact $K$ containing
$I$, let $\widehat W_K\subseteq W$ be the image in the finite Weyl
group of the subgroup of the extended affine Weyl group
associated with $K$, including its length-zero elements. For a principal-series enhancement $\rho$,
\begin{equation}\label{eq:weak-spherical-occurrence}
 \pi_{\mathrm{at},\rho}\text{ is weakly spherical}
 \quad\Longleftrightarrow\quad
 \operatorname{Hom}_A\!\left(E_\rho,
 H^*(\mathcal B_{e^\vee},\CC)\right)^{\widehat W_K}\ne0
 \quad\text{for some }K,
\end{equation}
where $\mathcal B_{e^\vee}$ is the Springer fibre. We checked for
these fixed vectors using GAP; the normalization is recorded in
the appendix. Only principal-series
enhancements need be tested, since $I\subseteq K$.

The calculations give the following result in exceptional types.

\begin{proposition}[Weak sphericity of contained packets]
\label{prop:contained-sphericity}
Let $G$ be split, simple and adjoint of exceptional type. In the
neutral unipotent setting, for every neutral grading,
$$
 \Pi^\mic\subseteq\Pi_{\mathcal O^\vee}^{\mathrm{weak}}
 \quad\Longrightarrow\quad \Pi^\mic\cap\mathcal S_\lambda\ne\varnothing
$$
for every microlocal Arthur packet $\Pi^\mic$.
\end{proposition}
\begin{proof}
For $\hd=0$, the graded space is a point and $H=\Gd$ is connected.
There is a single equivariant IC complex, and its microlocal packet
consists of the spherical representation with trivial unramified
Langlands parameter. The assertion follows in this case. For the
nonzero neutral gradings, the GAP calculations verify the assertion
for every contained packet.
\end{proof}
In fact, for each
split adjoint exceptional $G$, one maximal compact subgroup $K_G$
works for every neutral grading:
$$
 \Pi^\mic\subseteq\Pi_{\mathcal O^\vee}^{\mathrm{weak}}
 \quad\Longrightarrow\quad
 \exists\,\pi_a\in\Pi^\mic:\quad
       \pi_a^{K_G}\ne0,\qquad\dim\rho_a=1.
$$
The corresponding subgroups $\widehat W_{K_G}$ have reflection types
$$
\begin{array}{cccccc}
\text{dual type}&G_2&F_4&E_6&E_7&E_8\\\hline
\text{reflection type}&A_1\times A_1&A_2\times A_2&A_5\times A_1&
D_6\times A_1&A_4\times A_4.
\end{array}
$$
The enhancement is one-dimensional, but the space of $K_G$-fixed
vectors need not be. At $E_7(a_3)$ and $E_8(a_3)$, one must choose
a member whose support is larger than the defining Arthur orbit.
\section{Examples and extremality counterexamples}
\label{sec:examples}

We apply Theorem~\ref{thm:arthur-row-reconstruction} to four neutral
gradings. Throughout this section the Weil part is trivial,
$V=\gd(2)$ is the degree-two subspace for the indicated neutral element
$\hd$, and $H=Z_{\Gd}(\hd)$. The dual groups are simply connected.
We include all equivariant local systems; the packets on a fixed inner
form are obtained by imposing the relevance condition of
Section~\ref{sec:setup}.

We use the notation of Section~\ref{sec:reconstruction}. Thus $p(i)$ is
the index of $p(C_i)$, and $\FT P_a\simeq P_{f(a)}$. Orbit and IC indices
are assigned separately in each example. In the $F_4$ and $E_6$
examples, these indices differ from those in the online atlas.
For an Arthur orbit $C_i$, put
$$
 R_i=\sum_a\mathbf M_{ia}\pi_a,\qquad
 \Pi_i^{\mic}:=\Pi_{C_i}^{\mic}
       =\{\pi_a:\mathbf M_{ia}>0\}.
$$
The formal sum $R_i$ records the ordinary microlocal multiplicities.
It is distinct from the signed distribution of
Section~\ref{sec:reconstruction}, and does not determine the
representations of the microlocal component groups.

In the tables, $\Gd C_i$ denotes the saturation of $C_i$ and $A_H(C_i)$
its stabilizer component group. The IC columns specify
$P_a=\IC(\overline{C_i},\mathcal L_{\rho_a})$, together with
$\dim\rho_a$ and $f(a)$. Partitions label irreducible representations
of symmetric groups; in particular, $2$ and $11$ label the trivial
and nontrivial characters of $S_2$. We write $1$ for a trivial
component group or its unique character.

\subsection{The \texorpdfstring{$F_4(a_3)$}{F4(a3)} grading}
\label{subsec:example-f4}

Here $\hd=(0,2,0,0)$ and $\dim V=12$. There are twelve graded
orbits, twenty IC parameters, and eleven Arthur orbits. The orbits
$C_4$ and $C_5$ have the same dimension and saturation, but only $C_5$
is Arthur. The orbit and enhancement data are as follows.
\begingroup\small
\setlength{\tabcolsep}{4pt}
\setlength{\arrayrulewidth}{0.25pt}
\renewcommand{\arraystretch}{1.2}
\begin{longtable}{@{}rrlcrc@{\hspace{12pt}}rcrr@{}}
\hline

\multicolumn{6}{c}{\textbf{Orbit data}} & \multicolumn{4}{c}{\textbf{IC dictionary}}\\[3pt]
$i$ & $\dim C_i$ & $\Gd C_i$ & $A_H(C_i)$ & $p(i)$ & Arthur type & $a$ & $\rho_a$ & $\dim\rho_a$ & $f(a)$\\[3pt]
\hline

\endhead
\hline\endfoot
1 & 0 & $0$ & $1$ & 12 & Yes & \begin{tabular}[c]{@{}r@{}}1\end{tabular} & \begin{tabular}[c]{@{}c@{}}$1$\end{tabular} & \begin{tabular}[c]{@{}r@{}}1\end{tabular} & \begin{tabular}[c]{@{}r@{}}16\end{tabular}\\
\hline
2 & 4 & $A_1$ & $1$ & 11 & Yes & \begin{tabular}[c]{@{}r@{}}2\end{tabular} & \begin{tabular}[c]{@{}c@{}}$1$\end{tabular} & \begin{tabular}[c]{@{}r@{}}1\end{tabular} & \begin{tabular}[c]{@{}r@{}}17\end{tabular}\\
\hline
3 & 6 & $\tilde A_1$ & $\ZZ/2$ & 10 & Yes & \begin{tabular}[c]{@{}r@{}}3 \\ 4\end{tabular} & \begin{tabular}[c]{@{}c@{}}$2$ \\ $11$\end{tabular} & \begin{tabular}[c]{@{}r@{}}1 \\ 1\end{tabular} & \begin{tabular}[c]{@{}r@{}}18 \\ 11\end{tabular}\\
\hline
4 & 7 & $A_1{+}\tilde A_1$ & $1$ & 4 & No & \begin{tabular}[c]{@{}r@{}}6\end{tabular} & \begin{tabular}[c]{@{}c@{}}$1$\end{tabular} & \begin{tabular}[c]{@{}r@{}}1\end{tabular} & \begin{tabular}[c]{@{}r@{}}6\end{tabular}\\
\hline
5 & 7 & $A_1{+}\tilde A_1$ & $1$ & 9 & Yes & \begin{tabular}[c]{@{}r@{}}5\end{tabular} & \begin{tabular}[c]{@{}c@{}}$1$\end{tabular} & \begin{tabular}[c]{@{}r@{}}1\end{tabular} & \begin{tabular}[c]{@{}r@{}}14\end{tabular}\\
\hline
6 & 8 & $\tilde A_2$ & $1$ & 7 & Yes & \begin{tabular}[c]{@{}r@{}}8\end{tabular} & \begin{tabular}[c]{@{}c@{}}$1$\end{tabular} & \begin{tabular}[c]{@{}r@{}}1\end{tabular} & \begin{tabular}[c]{@{}r@{}}7\end{tabular}\\
\hline
7 & 8 & $A_2$ & $\ZZ/2$ & 6 & Yes & \begin{tabular}[c]{@{}r@{}}7 \\ 9\end{tabular} & \begin{tabular}[c]{@{}c@{}}$2$ \\ $11$\end{tabular} & \begin{tabular}[c]{@{}r@{}}1 \\ 1\end{tabular} & \begin{tabular}[c]{@{}r@{}}8 \\ 19\end{tabular}\\
\hline
8 & 9 & $A_2{+}\tilde A_1$ & $1$ & 8 & Yes & \begin{tabular}[c]{@{}r@{}}10\end{tabular} & \begin{tabular}[c]{@{}c@{}}$1$\end{tabular} & \begin{tabular}[c]{@{}r@{}}1\end{tabular} & \begin{tabular}[c]{@{}r@{}}15\end{tabular}\\
\hline
9 & 10 & $\tilde A_2{+}A_1$ & $1$ & 5 & Yes & \begin{tabular}[c]{@{}r@{}}13\end{tabular} & \begin{tabular}[c]{@{}c@{}}$1$\end{tabular} & \begin{tabular}[c]{@{}r@{}}1\end{tabular} & \begin{tabular}[c]{@{}r@{}}13\end{tabular}\\
\hline
10 & 10 & $B_2$ & $\ZZ/2$ & 3 & Yes & \begin{tabular}[c]{@{}r@{}}11 \\ 12\end{tabular} & \begin{tabular}[c]{@{}c@{}}$2$ \\ $11$\end{tabular} & \begin{tabular}[c]{@{}r@{}}1 \\ 1\end{tabular} & \begin{tabular}[c]{@{}r@{}}4 \\ 12\end{tabular}\\
\hline
11 & 11 & $C_3(a_1)$ & $\ZZ/2$ & 2 & Yes & \begin{tabular}[c]{@{}r@{}}14 \\ 15\end{tabular} & \begin{tabular}[c]{@{}c@{}}$2$ \\ $11$\end{tabular} & \begin{tabular}[c]{@{}r@{}}1 \\ 1\end{tabular} & \begin{tabular}[c]{@{}r@{}}5 \\ 10\end{tabular}\\
\hline
12 & 12 & $F_4(a_3)$ & $S_4$ & 1 & Yes & \begin{tabular}[c]{@{}r@{}}16 \\ 17 \\ 18 \\ 19 \\ 20\end{tabular} & \begin{tabular}[c]{@{}c@{}}$4$ \\ $31$ \\ $22$ \\ $211$ \\ $1111$\end{tabular} & \begin{tabular}[c]{@{}r@{}}1 \\ 3 \\ 2 \\ 3 \\ 1\end{tabular} & \begin{tabular}[c]{@{}r@{}}1 \\ 2 \\ 3 \\ 9 \\ 20\end{tabular}\\
\end{longtable}
\endgroup

\clearpage
\begin{figure}[!htbp]
\centering
\begin{tikzpicture}[x=1cm,y=1cm, every node/.style={font=\scriptsize},
 art/.style={circle,draw,fill=black,text=white,inner sep=1pt,minimum size=4.5mm},
 ordinary/.style={circle,draw,fill=white,inner sep=1pt,minimum size=4.5mm}]
\node[art] (n1) at (0.000,0.000) {$ 1 $};
\node[art] (n2) at (0.000,0.620) {$ 2 $};
\node[art] (n3) at (0.000,1.240) {$ 3 $};
\node[ordinary] (n4) at (-0.650,1.860) {$ 4 $};
\node[art] (n5) at (0.650,1.860) {$ 5 $};
\coordinate (nd4l4) at (-1.300,2.480);
\node[art] (n7) at (0.000,2.480) {$ 7 $};
\node[art] (n6) at (1.300,2.480) {$ 6 $};
\node[art] (n8) at (-0.650,3.100) {$ 8 $};
\coordinate (nd7l5) at (0.650,3.100);
\node[art] (n10) at (-0.650,3.720) {$ 10 $};
\node[art] (n9) at (0.650,3.720) {$ 9 $};
\node[art] (n11) at (0.000,4.340) {$ 11 $};
\node[art] (n12) at (0.000,4.960) {$ 12 $};
\draw[line width=.35pt] (n1) -- (n2);
\draw[line width=.35pt] (n2) -- (n3);
\draw[line width=.35pt] (n3) -- (n4);
\draw[line width=.35pt] (n3) -- (n5);
\draw[line width=.35pt] (n4) -- (nd4l4) -- (n8);
\draw[line width=.35pt] (n5) -- (n6);
\draw[line width=.35pt] (n5) -- (n7);
\draw[line width=.35pt] (n6) -- (nd7l5) -- (n9);
\draw[line width=.35pt] (n7) -- (n8);
\draw[line width=.35pt] (n8) -- (n9);
\draw[line width=.35pt] (n8) -- (n10);
\draw[line width=.35pt] (n9) -- (n11);
\draw[line width=.35pt] (n10) -- (n11);
\draw[line width=.35pt] (n11) -- (n12);
\node[anchor=east] at (-2.100,0.000) {$ 0 $};
\node[anchor=east] at (-2.100,0.620) {$ 4 $};
\node[anchor=east] at (-2.100,1.240) {$ 6 $};
\node[anchor=east] at (-2.100,1.860) {$ 7 $};
\node[anchor=east] at (-2.100,2.480) {$ 8 $};
\node[anchor=east] at (-2.100,3.100) {$ 9 $};
\node[anchor=east] at (-2.100,3.720) {$ 10 $};
\node[anchor=east] at (-2.100,4.340) {$ 11 $};
\node[anchor=east] at (-2.100,4.960) {$ 12 $};
\node[anchor=east] at (-2.100,5.580) {$\dim C$};
\end{tikzpicture}
\caption{Closure order for the $F_4(a_3)$ grading. Vertices are labelled by $i$, and filled vertices denote Arthur orbits. The order increases upwards; the left column gives $\dim C_i$. }\label{fig:hasse-f4}
\end{figure}

Figure~\ref{fig:hasse-f4} gives the closure order. The enhancement
$a=20$ is the cuspidal local system on the open orbit; the remaining
nineteen parameters belong to the principal series. The Arthur rows are
\begin{longtable}{rl}
$i$ & $R_i$\\ \hline
\endhead
1 & $\pi_{1}+3\pi_{2}+2\pi_{3}+3\pi_{9}+\pi_{20}$\\
2 & $\pi_{2}+\pi_{3}+\pi_{5}+2\pi_{9}+\pi_{10}+\pi_{20}$\\
3 & $\pi_{3}+\pi_{4}+\pi_{9}+2\pi_{10}+\pi_{12}+\pi_{20}$\\
5 & $\pi_{5}+\pi_{9}+\pi_{10}+\pi_{13}+\pi_{20}$\\
6 & $\pi_{8}+2\pi_{13}+\pi_{19}+\pi_{20}$\\
7 & $\pi_{7}+\pi_{9}+2\pi_{13}+\pi_{20}$\\
8 & $\pi_{10}+\pi_{12}+\pi_{13}+\pi_{15}+\pi_{20}$\\
9 & $\pi_{13}+\pi_{14}+\pi_{15}+\pi_{19}+\pi_{20}$\\
10 & $\pi_{11}+\pi_{12}+2\pi_{15}+\pi_{18}+\pi_{19}+\pi_{20}$\\
11 & $\pi_{14}+\pi_{15}+\pi_{17}+\pi_{18}+2\pi_{19}+\pi_{20}$\\
12 & $\pi_{16}+3\pi_{17}+2\pi_{18}+3\pi_{19}+\pi_{20}$\\
\end{longtable}
Thus every $\pi_a$ except $\pi_6$ belongs to an Arthur packet.
Moreover, $\mathbf M_{i,20}=1$ for every Arthur orbit $C_i$.

For comparison with \cite[Tables~1--2 and Theorem~3.2]{BL}, the orbit dictionary is
$$
 (C_1,\ldots,C_{12})=
 (\mathcal S_0,\mathcal S_1,\mathcal S_2,\mathcal S_4,
 \mathcal S_3,\mathcal S_6,\mathcal S_5,\mathcal S_7,
 \mathcal S_8,\mathcal S_9,\mathcal S_{10},\mathcal S_{11}).
$$
Under this identification, our multiplicities on the Arthur conormals
agree with the nineteen principal-series characteristic-cycle formulas
of \cite[Theorem~3.2]{BL}. The undetermined coefficient in those formulas
concerns $\mathcal S_4=C_4$, which is not Arthur, and hence does not enter
the packet table. For example, writing
$\Lambda_i=\overline{T^*_{C_i}V}$, the Arthur part of the column $a=19$ is
$$
 [\Lambda_6]+[\Lambda_9]+[\Lambda_{10}]
       +2[\Lambda_{11}]+3[\Lambda_{12}].
$$
We use this column in Section~\ref{subsec:extremality}.

\clearpage
\subsection{The \texorpdfstring{$D_4(a_1)$}{D4(a1)} grading in \texorpdfstring{$E_6$}{E6}}
\label{subsec:example-e6}

Here $\hd=(0,0,0,2,0,0)$ and $\dim V=18$. There are eighteen graded
orbits, twenty-five IC parameters, and eleven Arthur orbits. Their
orbit and enhancement data are as follows.
\begingroup\small
\setlength{\tabcolsep}{4pt}
\setlength{\arrayrulewidth}{0.25pt}
\renewcommand{\arraystretch}{1.2}
\begin{longtable}{@{}rrlcrc@{\hspace{12pt}}rcrr@{}}
\hline

\multicolumn{6}{c}{\textbf{Orbit data}} & \multicolumn{4}{c}{\textbf{IC dictionary}}\\[3pt]
$i$ & $\dim C_i$ & $\Gd C_i$ & $A_H(C_i)$ & $p(i)$ & Arthur type & $a$ & $\rho_a$ & $\dim\rho_a$ & $f(a)$\\[3pt]
\hline

\endhead
\hline\endfoot
1 & 0 & $0$ & $1$ & 18 & Yes & \begin{tabular}[c]{@{}r@{}}1\end{tabular} & \begin{tabular}[c]{@{}c@{}}$1$\end{tabular} & \begin{tabular}[c]{@{}r@{}}1\end{tabular} & \begin{tabular}[c]{@{}r@{}}19\end{tabular}\\
\hline
2 & 6 & $A_1$ & $1$ & 17 & Yes & \begin{tabular}[c]{@{}r@{}}2\end{tabular} & \begin{tabular}[c]{@{}c@{}}$1$\end{tabular} & \begin{tabular}[c]{@{}r@{}}1\end{tabular} & \begin{tabular}[c]{@{}r@{}}20\end{tabular}\\
\hline
3 & 8 & $2A_1$ & $1$ & 13 & No & \begin{tabular}[c]{@{}r@{}}4\end{tabular} & \begin{tabular}[c]{@{}c@{}}$1$\end{tabular} & \begin{tabular}[c]{@{}r@{}}1\end{tabular} & \begin{tabular}[c]{@{}r@{}}13\end{tabular}\\
\hline
4 & 8 & $2A_1$ & $1$ & 11 & No & \begin{tabular}[c]{@{}r@{}}3\end{tabular} & \begin{tabular}[c]{@{}c@{}}$1$\end{tabular} & \begin{tabular}[c]{@{}r@{}}1\end{tabular} & \begin{tabular}[c]{@{}r@{}}12\end{tabular}\\
\hline
5 & 9 & $2A_1$ & $1$ & 15 & Yes & \begin{tabular}[c]{@{}r@{}}5\end{tabular} & \begin{tabular}[c]{@{}c@{}}$1$\end{tabular} & \begin{tabular}[c]{@{}r@{}}1\end{tabular} & \begin{tabular}[c]{@{}r@{}}16\end{tabular}\\
\hline
6 & 10 & $3A_1$ & $1$ & 6 & No & \begin{tabular}[c]{@{}r@{}}6\end{tabular} & \begin{tabular}[c]{@{}c@{}}$1$\end{tabular} & \begin{tabular}[c]{@{}r@{}}1\end{tabular} & \begin{tabular}[c]{@{}r@{}}6\end{tabular}\\
\hline
7 & 11 & $3A_1$ & $1$ & 16 & Yes & \begin{tabular}[c]{@{}r@{}}7\end{tabular} & \begin{tabular}[c]{@{}c@{}}$1$\end{tabular} & \begin{tabular}[c]{@{}r@{}}1\end{tabular} & \begin{tabular}[c]{@{}r@{}}18\end{tabular}\\
\hline
8 & 12 & $A_2$ & $\ZZ/2$ & 12 & Yes & \begin{tabular}[c]{@{}r@{}}8 \\ 9\end{tabular} & \begin{tabular}[c]{@{}c@{}}$2$ \\ $11$\end{tabular} & \begin{tabular}[c]{@{}r@{}}1 \\ 1\end{tabular} & \begin{tabular}[c]{@{}r@{}}15 \\ 21\end{tabular}\\
\hline
9 & 13 & $A_2{+}A_1$ & $1$ & 9 & No & \begin{tabular}[c]{@{}r@{}}11\end{tabular} & \begin{tabular}[c]{@{}c@{}}$1$\end{tabular} & \begin{tabular}[c]{@{}r@{}}1\end{tabular} & \begin{tabular}[c]{@{}r@{}}11\end{tabular}\\
\hline
10 & 13 & $A_2{+}A_1$ & $1$ & 10 & No & \begin{tabular}[c]{@{}r@{}}10\end{tabular} & \begin{tabular}[c]{@{}c@{}}$1$\end{tabular} & \begin{tabular}[c]{@{}r@{}}1\end{tabular} & \begin{tabular}[c]{@{}r@{}}10\end{tabular}\\
\hline
11 & 14 & $A_2{+}2A_1$ & $1$ & 4 & No & \begin{tabular}[c]{@{}r@{}}12\end{tabular} & \begin{tabular}[c]{@{}c@{}}$1$\end{tabular} & \begin{tabular}[c]{@{}r@{}}1\end{tabular} & \begin{tabular}[c]{@{}r@{}}3\end{tabular}\\
\hline
12 & 14 & $2A_2$ & $\ZZ/3$ & 8 & Yes & \begin{tabular}[c]{@{}r@{}}15 \\ 22 \\ 24\end{tabular} & \begin{tabular}[c]{@{}c@{}}$1$ \\ $\zeta_3$ \\ $\zeta_3^2$\end{tabular} & \begin{tabular}[c]{@{}r@{}}1 \\ 1 \\ 1\end{tabular} & \begin{tabular}[c]{@{}r@{}}8 \\ 23 \\ 25\end{tabular}\\
\hline
13 & 14 & $A_2{+}2A_1$ & $1$ & 3 & No & \begin{tabular}[c]{@{}r@{}}13\end{tabular} & \begin{tabular}[c]{@{}c@{}}$1$\end{tabular} & \begin{tabular}[c]{@{}r@{}}1\end{tabular} & \begin{tabular}[c]{@{}r@{}}4\end{tabular}\\
\hline
14 & 14 & $A_2{+}2A_1$ & $1$ & 14 & Yes & \begin{tabular}[c]{@{}r@{}}14\end{tabular} & \begin{tabular}[c]{@{}c@{}}$1$\end{tabular} & \begin{tabular}[c]{@{}r@{}}1\end{tabular} & \begin{tabular}[c]{@{}r@{}}14\end{tabular}\\
\hline
15 & 15 & $A_3$ & $1$ & 5 & Yes & \begin{tabular}[c]{@{}r@{}}16\end{tabular} & \begin{tabular}[c]{@{}c@{}}$1$\end{tabular} & \begin{tabular}[c]{@{}r@{}}1\end{tabular} & \begin{tabular}[c]{@{}r@{}}5\end{tabular}\\
\hline
16 & 16 & $2A_2{+}A_1$ & $\ZZ/3$ & 7 & Yes & \begin{tabular}[c]{@{}r@{}}17 \\ 23 \\ 25\end{tabular} & \begin{tabular}[c]{@{}c@{}}$1$ \\ $\zeta_3$ \\ $\zeta_3^2$\end{tabular} & \begin{tabular}[c]{@{}r@{}}1 \\ 1 \\ 1\end{tabular} & \begin{tabular}[c]{@{}r@{}}17 \\ 22 \\ 24\end{tabular}\\
\hline
17 & 17 & $A_3{+}A_1$ & $1$ & 2 & Yes & \begin{tabular}[c]{@{}r@{}}18\end{tabular} & \begin{tabular}[c]{@{}c@{}}$1$\end{tabular} & \begin{tabular}[c]{@{}r@{}}1\end{tabular} & \begin{tabular}[c]{@{}r@{}}7\end{tabular}\\
\hline
18 & 18 & $D_4(a_1)$ & $S_3$ & 1 & Yes & \begin{tabular}[c]{@{}r@{}}19 \\ 20 \\ 21\end{tabular} & \begin{tabular}[c]{@{}c@{}}$3$ \\ $21$ \\ $111$\end{tabular} & \begin{tabular}[c]{@{}r@{}}1 \\ 2 \\ 1\end{tabular} & \begin{tabular}[c]{@{}r@{}}1 \\ 2 \\ 9\end{tabular}\\
\end{longtable}
\endgroup

\clearpage
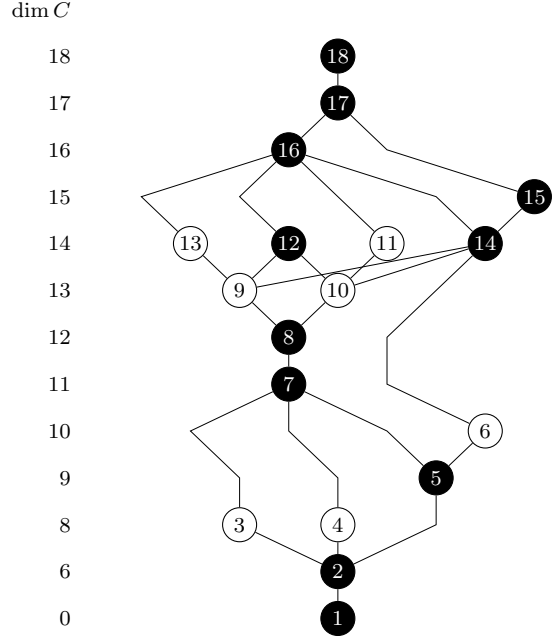
\begin{figure}[!htbp]
\centering
\begin{tikzpicture}[x=1cm,y=1cm, every node/.style={font=\scriptsize},
 art/.style={circle,draw,fill=black,text=white,inner sep=1pt,minimum size=4.5mm},
 ordinary/.style={circle,draw,fill=white,inner sep=1pt,minimum size=4.5mm}]
\node[art] (n1) at (0.000,0.000) {$ 1 $};
\node[art] (n2) at (0.000,0.620) {$ 2 $};
\node[ordinary] (n3) at (-1.300,1.240) {$ 3 $};
\node[ordinary] (n4) at (0.000,1.240) {$ 4 $};
\coordinate (nd3l2) at (1.300,1.240);
\coordinate (nd4l3) at (-1.300,1.860);
\coordinate (nd5l3) at (0.000,1.860);
\node[art] (n5) at (1.300,1.860) {$ 5 $};
\coordinate (nd4l4) at (-1.950,2.480);
\coordinate (nd5l4) at (-0.650,2.480);
\coordinate (nd7l4) at (0.650,2.480);
\node[ordinary] (n6) at (1.950,2.480) {$ 6 $};
\node[art] (n7) at (-0.650,3.100) {$ 7 $};
\coordinate (nd8l5) at (0.650,3.100);
\node[art] (n8) at (-0.650,3.720) {$ 8 $};
\coordinate (nd8l6) at (0.650,3.720);
\node[ordinary] (n9) at (-1.300,4.340) {$ 9 $};
\node[ordinary] (n10) at (0.000,4.340) {$ 10 $};
\coordinate (nd8l7) at (1.300,4.340);
\node[ordinary] (n13) at (-1.950,4.960) {$ 13 $};
\node[art] (n12) at (-0.650,4.960) {$ 12 $};
\node[ordinary] (n11) at (0.650,4.960) {$ 11 $};
\node[art] (n14) at (1.950,4.960) {$ 14 $};
\coordinate (nd20l9) at (-2.600,5.580);
\coordinate (nd19l9) at (-1.300,5.580);
\coordinate (nd18l9) at (0.000,5.580);
\coordinate (nd22l9) at (1.300,5.580);
\node[art] (n15) at (2.600,5.580) {$ 15 $};
\node[art] (n16) at (-0.650,6.200) {$ 16 $};
\coordinate (nd23l10) at (0.650,6.200);
\node[art] (n17) at (0.000,6.820) {$ 17 $};
\node[art] (n18) at (0.000,7.440) {$ 18 $};
\draw[line width=.35pt] (n1) -- (n2);
\draw[line width=.35pt] (n2) -- (n3);
\draw[line width=.35pt] (n2) -- (n4);
\draw[line width=.35pt] (n2) -- (nd3l2) -- (n5);
\draw[line width=.35pt] (n3) -- (nd4l3) -- (nd4l4) -- (n7);
\draw[line width=.35pt] (n4) -- (nd5l3) -- (nd5l4) -- (n7);
\draw[line width=.35pt] (n5) -- (n6);
\draw[line width=.35pt] (n5) -- (nd7l4) -- (n7);
\draw[line width=.35pt] (n6) -- (nd8l5) -- (nd8l6) -- (nd8l7) -- (n14);
\draw[line width=.35pt] (n7) -- (n8);
\draw[line width=.35pt] (n8) -- (n9);
\draw[line width=.35pt] (n8) -- (n10);
\draw[line width=.35pt] (n9) -- (n12);
\draw[line width=.35pt] (n9) -- (n13);
\draw[line width=.35pt] (n9) -- (n14);
\draw[line width=.35pt] (n10) -- (n11);
\draw[line width=.35pt] (n10) -- (n12);
\draw[line width=.35pt] (n10) -- (n14);
\draw[line width=.35pt] (n11) -- (nd18l9) -- (n16);
\draw[line width=.35pt] (n12) -- (nd19l9) -- (n16);
\draw[line width=.35pt] (n13) -- (nd20l9) -- (n16);
\draw[line width=.35pt] (n14) -- (n15);
\draw[line width=.35pt] (n14) -- (nd22l9) -- (n16);
\draw[line width=.35pt] (n15) -- (nd23l10) -- (n17);
\draw[line width=.35pt] (n16) -- (n17);
\draw[line width=.35pt] (n17) -- (n18);
\node[anchor=east] at (-3.400,0.000) {$ 0 $};
\node[anchor=east] at (-3.400,0.620) {$ 6 $};
\node[anchor=east] at (-3.400,1.240) {$ 8 $};
\node[anchor=east] at (-3.400,1.860) {$ 9 $};
\node[anchor=east] at (-3.400,2.480) {$ 10 $};
\node[anchor=east] at (-3.400,3.100) {$ 11 $};
\node[anchor=east] at (-3.400,3.720) {$ 12 $};
\node[anchor=east] at (-3.400,4.340) {$ 13 $};
\node[anchor=east] at (-3.400,4.960) {$ 14 $};
\node[anchor=east] at (-3.400,5.580) {$ 15 $};
\node[anchor=east] at (-3.400,6.200) {$ 16 $};
\node[anchor=east] at (-3.400,6.820) {$ 17 $};
\node[anchor=east] at (-3.400,7.440) {$ 18 $};
\node[anchor=east] at (-3.400,8.060) {$\dim C$};
\end{tikzpicture}
\caption{Closure order for the $D_4(a_1)$ grading in $E_6$. Vertices are labelled by $i$, and filled vertices denote Arthur orbits. The order increases upwards; the left column gives $\dim C_i$. }\label{fig:hasse-e6}
\end{figure}

The closure order is shown in Figure~\ref{fig:hasse-e6}. The parameters
$22,23,24,25$ belong to the two non-principal series with cuspidal support
of type $A_2\times A_2$. We label their local systems by restriction of
the ambient characters: $\zeta_3$ and $\zeta_3^2$ correspond respectively
to the CHEVIE labels \texttt{E3} and \texttt{E3\^{}2}.
The complete packet table is
\begin{longtable}{rl}
$i$ & $R_i$\\ \hline
\endhead
1 & $\pi_{1}+2\pi_{2}+\pi_{9}$\\
2 & $\pi_{2}+\pi_{7}+\pi_{9}$\\
5 & $\pi_{5}+\pi_{7}$\\
7 & $\pi_{7}+\pi_{9}+\pi_{17}+\pi_{22}+\pi_{24}$\\
8 & $\pi_{8}+\pi_{9}+2\pi_{17}+2\pi_{22}+\pi_{23}+2\pi_{24}+\pi_{25}$\\
12 & $\pi_{15}+2\pi_{17}+\pi_{21}+\pi_{22}+2\pi_{23}+\pi_{24}+2\pi_{25}$\\
14 & $\pi_{14}+\pi_{17}$\\
15 & $\pi_{16}+\pi_{18}$\\
16 & $\pi_{17}+\pi_{18}+\pi_{21}+\pi_{23}+\pi_{25}$\\
17 & $\pi_{18}+\pi_{20}+\pi_{21}$\\
18 & $\pi_{19}+2\pi_{20}+\pi_{21}$\\
\end{longtable}
The multiplicity $\mathbf M_{18,20}=2$ is the rank of the standard
$S_3$ local system on the open orbit. Since $p(18)=1$ and $f(20)=2$,
Fourier symmetry gives $\mathbf M_{1,2}=2$. The other multiplicities
two are
$$
 \mathbf M_{8,17}=\mathbf M_{8,22}=\mathbf M_{8,24}=2,
 \qquad
 \mathbf M_{12,17}=\mathbf M_{12,23}=\mathbf M_{12,25}=2.
$$
These occur along conormals to proper boundary orbits of IC complexes
with rank-one local systems. In particular, they are not top-conormal
ranks. The identities $p(8)=12$, $f(17)=17$, $f(22)=23$, and $f(24)=25$
identify the two triples by Fourier symmetry.

The special piece of $D_4(a_1)$ consists of
$D_4(a_1)$, $A_3+A_1$, and $2A_2+A_1$.
The graded orbits with these saturations are $C_{18},C_{17},C_{16}$,
and their images under $p$ are $C_1,C_2,C_7$. With the Fourier--Aubert
identification used in Section~\ref{sec:weak}, the weak packet is
\begin{equation}\label{eq:e6-weak-example}
 \Pi_{D_4(a_1)}^{\mathrm{weak}}
 =\Pi_1^{\mic}\cup\Pi_2^{\mic}\cup\Pi_7^{\mic}
 =\{\pi_1,\pi_2,\pi_7,\pi_9,\pi_{17},\pi_{22},\pi_{24}\}.
\end{equation}
Indeed, the enhancements supported on $C_{18},C_{17},C_{16}$ have indices
$\{17,18,19,20,21,23,25\}$. Applying $f$ gives the set on the right-hand
side, which is also the union of the supports of $R_1,R_2,R_7$.
The equality includes the non-principal enhancements and restricts to
each inner form by relevance.

\subsection{The \texorpdfstring{$A_4+A_1$}{A4+A1} grading in \texorpdfstring{$E_7$}{E7}}
\label{subsec:example-e7}

Here $\hd=(1,0,0,1,0,1,0)$ and $\dim V=13$. There are fifty graded
orbits, sixty-six IC parameters, and six Arthur orbits. We retain the
orbit and IC indices of the all-series atlas. The Arthur orbits are
$C_1,C_2,C_{25},C_{34},C_{48},C_{50}$, and Piasetskii duality pairs them as
$$
 (1,50),\qquad (2,48),\qquad (25,34).
$$
Every irreducible equivariant local system has rank one. At $C_{49}$,
the ordered pairs of partitions label the characters of
$A_H(C_{49})\simeq S_2\times S_2$, with the second factor containing
the image of $Z(\Gd)$. The orbit and enhancement data are as follows.

\begingroup\small
\setlength{\tabcolsep}{4pt}
\setlength{\arrayrulewidth}{0.25pt}
\renewcommand{\arraystretch}{1.2}
\begin{longtable}{@{}rrlcrc@{\hspace{12pt}}rcrr@{}}
\hline
\multicolumn{6}{c}{\textbf{Orbit data}} & \multicolumn{4}{c}{\textbf{IC dictionary}}\\[3pt]
$i$ & $\dim C_i$ & $\Gd C_i$ & $A_H(C_i)$ & $p(i)$ & Arthur type & $a$ & $\rho_a$ & $\dim\rho_a$ & $f(a)$\\[3pt]
\hline
\endhead
\hline\endfoot
1 & 0 & $0$ & $1$ & 50 & Yes & \begin{tabular}[c]{@{}r@{}}1\end{tabular} & \begin{tabular}[c]{@{}c@{}}$1$\end{tabular} & \begin{tabular}[c]{@{}r@{}}1\end{tabular} & \begin{tabular}[c]{@{}r@{}}55\end{tabular}\\
\hline
2 & 1 & $A_1$ & $1$ & 48 & Yes & \begin{tabular}[c]{@{}r@{}}2\end{tabular} & \begin{tabular}[c]{@{}c@{}}$1$\end{tabular} & \begin{tabular}[c]{@{}r@{}}1\end{tabular} & \begin{tabular}[c]{@{}r@{}}52\end{tabular}\\
\hline
3 & 3 & $A_1$ & $1$ & 49 & No & \begin{tabular}[c]{@{}r@{}}3\end{tabular} & \begin{tabular}[c]{@{}c@{}}$1$\end{tabular} & \begin{tabular}[c]{@{}r@{}}1\end{tabular} & \begin{tabular}[c]{@{}r@{}}54\end{tabular}\\
\hline
4 & 4 & $2A_1$ & $1$ & 44 & No & \begin{tabular}[c]{@{}r@{}}5\end{tabular} & \begin{tabular}[c]{@{}c@{}}$1$\end{tabular} & \begin{tabular}[c]{@{}r@{}}1\end{tabular} & \begin{tabular}[c]{@{}r@{}}47\end{tabular}\\
\hline
5 & 4 & $2A_1$ & $1$ & 33 & No & \begin{tabular}[c]{@{}r@{}}4\end{tabular} & \begin{tabular}[c]{@{}c@{}}$1$\end{tabular} & \begin{tabular}[c]{@{}r@{}}1\end{tabular} & \begin{tabular}[c]{@{}r@{}}37\end{tabular}\\
\hline
6 & 4 & $A_1$ & $1$ & 47 & No & \begin{tabular}[c]{@{}r@{}}6\end{tabular} & \begin{tabular}[c]{@{}c@{}}$1$\end{tabular} & \begin{tabular}[c]{@{}r@{}}1\end{tabular} & \begin{tabular}[c]{@{}r@{}}56\end{tabular}\\
\hline
7 & 5 & $2A_1$ & $1$ & 40 & No & \begin{tabular}[c]{@{}r@{}}7\end{tabular} & \begin{tabular}[c]{@{}c@{}}$1$\end{tabular} & \begin{tabular}[c]{@{}r@{}}1\end{tabular} & \begin{tabular}[c]{@{}r@{}}41\end{tabular}\\
\hline
8 & 5 & $3A_1'$ & $1$ & 24 & No & \begin{tabular}[c]{@{}r@{}}8\end{tabular} & \begin{tabular}[c]{@{}c@{}}$1$\end{tabular} & \begin{tabular}[c]{@{}r@{}}1\end{tabular} & \begin{tabular}[c]{@{}r@{}}25\end{tabular}\\
\hline
9 & 5 & $2A_1$ & $1$ & 46 & No & \begin{tabular}[c]{@{}r@{}}10\end{tabular} & \begin{tabular}[c]{@{}c@{}}$1$\end{tabular} & \begin{tabular}[c]{@{}r@{}}1\end{tabular} & \begin{tabular}[c]{@{}r@{}}50\end{tabular}\\
\hline
10 & 5 & $2A_1$ & $1$ & 37 & No & \begin{tabular}[c]{@{}r@{}}9\end{tabular} & \begin{tabular}[c]{@{}c@{}}$1$\end{tabular} & \begin{tabular}[c]{@{}r@{}}1\end{tabular} & \begin{tabular}[c]{@{}r@{}}39\end{tabular}\\
\hline
11 & 5 & $2A_1$ & $1$ & 38 & No & \begin{tabular}[c]{@{}r@{}}11\end{tabular} & \begin{tabular}[c]{@{}c@{}}$1$\end{tabular} & \begin{tabular}[c]{@{}r@{}}1\end{tabular} & \begin{tabular}[c]{@{}r@{}}44\end{tabular}\\
\hline
12 & 6 & $3A_1'$ & $1$ & 35 & No & \begin{tabular}[c]{@{}r@{}}13\end{tabular} & \begin{tabular}[c]{@{}c@{}}$1$\end{tabular} & \begin{tabular}[c]{@{}r@{}}1\end{tabular} & \begin{tabular}[c]{@{}r@{}}29\end{tabular}\\
\hline
13 & 6 & $3A_1''$ & $\ZZ/2$ & 31 & No & \begin{tabular}[c]{@{}r@{}}12 \\ 57\end{tabular} & \begin{tabular}[c]{@{}c@{}}$2$ \\ $11$\end{tabular} & \begin{tabular}[c]{@{}r@{}}1 \\ 1\end{tabular} & \begin{tabular}[c]{@{}r@{}}31 \\ 66\end{tabular}\\
\hline
14 & 6 & $3A_1'$ & $1$ & 28 & No & \begin{tabular}[c]{@{}r@{}}14\end{tabular} & \begin{tabular}[c]{@{}c@{}}$1$\end{tabular} & \begin{tabular}[c]{@{}r@{}}1\end{tabular} & \begin{tabular}[c]{@{}r@{}}36\end{tabular}\\
\hline
15 & 6 & $2A_1$ & $1$ & 43 & No & \begin{tabular}[c]{@{}r@{}}15\end{tabular} & \begin{tabular}[c]{@{}c@{}}$1$\end{tabular} & \begin{tabular}[c]{@{}r@{}}1\end{tabular} & \begin{tabular}[c]{@{}r@{}}49\end{tabular}\\
\hline
16 & 7 & $A_2$ & $1$ & 45 & No & \begin{tabular}[c]{@{}r@{}}18\end{tabular} & \begin{tabular}[c]{@{}c@{}}$1$\end{tabular} & \begin{tabular}[c]{@{}r@{}}1\end{tabular} & \begin{tabular}[c]{@{}r@{}}53\end{tabular}\\
\hline
17 & 7 & $3A_1'$ & $1$ & 32 & No & \begin{tabular}[c]{@{}r@{}}17\end{tabular} & \begin{tabular}[c]{@{}c@{}}$1$\end{tabular} & \begin{tabular}[c]{@{}r@{}}1\end{tabular} & \begin{tabular}[c]{@{}r@{}}51\end{tabular}\\
\hline
18 & 7 & $3A_1'$ & $1$ & 42 & No & \begin{tabular}[c]{@{}r@{}}16\end{tabular} & \begin{tabular}[c]{@{}c@{}}$1$\end{tabular} & \begin{tabular}[c]{@{}r@{}}1\end{tabular} & \begin{tabular}[c]{@{}r@{}}40\end{tabular}\\
\hline
19 & 7 & $3A_1'$ & $1$ & 23 & No & \begin{tabular}[c]{@{}r@{}}19\end{tabular} & \begin{tabular}[c]{@{}c@{}}$1$\end{tabular} & \begin{tabular}[c]{@{}r@{}}1\end{tabular} & \begin{tabular}[c]{@{}r@{}}28\end{tabular}\\
\hline
20 & 8 & $A_2{+}A_1$ & $1$ & 27 & No & \begin{tabular}[c]{@{}r@{}}21\end{tabular} & \begin{tabular}[c]{@{}c@{}}$1$\end{tabular} & \begin{tabular}[c]{@{}r@{}}1\end{tabular} & \begin{tabular}[c]{@{}r@{}}32\end{tabular}\\
\hline
21 & 8 & $A_2{+}A_1$ & $1$ & 36 & No & \begin{tabular}[c]{@{}r@{}}24\end{tabular} & \begin{tabular}[c]{@{}c@{}}$1$\end{tabular} & \begin{tabular}[c]{@{}r@{}}1\end{tabular} & \begin{tabular}[c]{@{}r@{}}33\end{tabular}\\
\hline
22 & 8 & $A_2{+}A_1$ & $1$ & 30 & No & \begin{tabular}[c]{@{}r@{}}27\end{tabular} & \begin{tabular}[c]{@{}c@{}}$1$\end{tabular} & \begin{tabular}[c]{@{}r@{}}1\end{tabular} & \begin{tabular}[c]{@{}r@{}}30\end{tabular}\\
\hline
23 & 8 & $4A_1$ & $\ZZ/2$ & 19 & No & \begin{tabular}[c]{@{}r@{}}28 \\ 58\end{tabular} & \begin{tabular}[c]{@{}c@{}}$2$ \\ $11$\end{tabular} & \begin{tabular}[c]{@{}r@{}}1 \\ 1\end{tabular} & \begin{tabular}[c]{@{}r@{}}19 \\ 64\end{tabular}\\
\hline
24 & 8 & $A_2$ & $\ZZ/2$ & 8 & No & \begin{tabular}[c]{@{}r@{}}23 \\ 25\end{tabular} & \begin{tabular}[c]{@{}c@{}}$11$ \\ $2$\end{tabular} & \begin{tabular}[c]{@{}r@{}}1 \\ 1\end{tabular} & \begin{tabular}[c]{@{}r@{}}35 \\ 8\end{tabular}\\
\hline
25 & 8 & $A_3$ & $1$ & 34 & Yes & \begin{tabular}[c]{@{}r@{}}20\end{tabular} & \begin{tabular}[c]{@{}c@{}}$1$\end{tabular} & \begin{tabular}[c]{@{}r@{}}1\end{tabular} & \begin{tabular}[c]{@{}r@{}}34\end{tabular}\\
\hline
26 & 8 & $A_2{+}A_1$ & $1$ & 41 & No & \begin{tabular}[c]{@{}r@{}}26\end{tabular} & \begin{tabular}[c]{@{}c@{}}$1$\end{tabular} & \begin{tabular}[c]{@{}r@{}}1\end{tabular} & \begin{tabular}[c]{@{}r@{}}46\end{tabular}\\
\hline
27 & 8 & $4A_1$ & $\ZZ/2$ & 20 & No & \begin{tabular}[c]{@{}r@{}}22 \\ 59\end{tabular} & \begin{tabular}[c]{@{}c@{}}$2$ \\ $11$\end{tabular} & \begin{tabular}[c]{@{}r@{}}1 \\ 1\end{tabular} & \begin{tabular}[c]{@{}r@{}}45 \\ 62\end{tabular}\\
\hline
28 & 9 & $2A_2$ & $1$ & 14 & No & \begin{tabular}[c]{@{}r@{}}36\end{tabular} & \begin{tabular}[c]{@{}c@{}}$1$\end{tabular} & \begin{tabular}[c]{@{}r@{}}1\end{tabular} & \begin{tabular}[c]{@{}r@{}}14\end{tabular}\\
\hline
29 & 9 & $A_2{+}2A_1$ & $1$ & 39 & No & \begin{tabular}[c]{@{}r@{}}38\end{tabular} & \begin{tabular}[c]{@{}c@{}}$1$\end{tabular} & \begin{tabular}[c]{@{}r@{}}1\end{tabular} & \begin{tabular}[c]{@{}r@{}}43\end{tabular}\\
\hline
30 & 9 & $A_2{+}2A_1$ & $1$ & 22 & No & \begin{tabular}[c]{@{}r@{}}30\end{tabular} & \begin{tabular}[c]{@{}c@{}}$1$\end{tabular} & \begin{tabular}[c]{@{}r@{}}1\end{tabular} & \begin{tabular}[c]{@{}r@{}}27\end{tabular}\\
\hline
31 & 9 & $(A_3{+}A_1)'$ & $1$ & 13 & No & \begin{tabular}[c]{@{}r@{}}31\end{tabular} & \begin{tabular}[c]{@{}c@{}}$1$\end{tabular} & \begin{tabular}[c]{@{}r@{}}1\end{tabular} & \begin{tabular}[c]{@{}r@{}}12\end{tabular}\\
\hline
32 & 9 & $A_2{+}2A_1$ & $1$ & 17 & No & \begin{tabular}[c]{@{}r@{}}35\end{tabular} & \begin{tabular}[c]{@{}c@{}}$1$\end{tabular} & \begin{tabular}[c]{@{}r@{}}1\end{tabular} & \begin{tabular}[c]{@{}r@{}}23\end{tabular}\\
\hline
33 & 9 & $A_2{+}A_1$ & $\ZZ/2$ & 5 & No & \begin{tabular}[c]{@{}r@{}}32 \\ 37\end{tabular} & \begin{tabular}[c]{@{}c@{}}$11$ \\ $2$\end{tabular} & \begin{tabular}[c]{@{}r@{}}1 \\ 1\end{tabular} & \begin{tabular}[c]{@{}r@{}}21 \\ 4\end{tabular}\\
\hline
34 & 9 & $(A_3{+}A_1)''$ & $\ZZ/2$ & 25 & Yes & \begin{tabular}[c]{@{}r@{}}34 \\ 60\end{tabular} & \begin{tabular}[c]{@{}c@{}}$2$ \\ $11$\end{tabular} & \begin{tabular}[c]{@{}r@{}}1 \\ 1\end{tabular} & \begin{tabular}[c]{@{}r@{}}20 \\ 63\end{tabular}\\
\hline
35 & 9 & $2A_2$ & $1$ & 12 & No & \begin{tabular}[c]{@{}r@{}}29\end{tabular} & \begin{tabular}[c]{@{}c@{}}$1$\end{tabular} & \begin{tabular}[c]{@{}r@{}}1\end{tabular} & \begin{tabular}[c]{@{}r@{}}13\end{tabular}\\
\hline
36 & 9 & $A_2{+}2A_1$ & $1$ & 21 & No & \begin{tabular}[c]{@{}r@{}}33\end{tabular} & \begin{tabular}[c]{@{}c@{}}$1$\end{tabular} & \begin{tabular}[c]{@{}r@{}}1\end{tabular} & \begin{tabular}[c]{@{}r@{}}24\end{tabular}\\
\hline
37 & 10 & $2A_2{+}A_1$ & $1$ & 10 & No & \begin{tabular}[c]{@{}r@{}}39\end{tabular} & \begin{tabular}[c]{@{}c@{}}$1$\end{tabular} & \begin{tabular}[c]{@{}r@{}}1\end{tabular} & \begin{tabular}[c]{@{}r@{}}9\end{tabular}\\
\hline
38 & 10 & $2A_2{+}A_1$ & $1$ & 11 & No & \begin{tabular}[c]{@{}r@{}}44\end{tabular} & \begin{tabular}[c]{@{}c@{}}$1$\end{tabular} & \begin{tabular}[c]{@{}r@{}}1\end{tabular} & \begin{tabular}[c]{@{}r@{}}11\end{tabular}\\
\hline
39 & 10 & $A_2{+}3A_1$ & $\ZZ/2$ & 29 & No & \begin{tabular}[c]{@{}r@{}}43 \\ 61\end{tabular} & \begin{tabular}[c]{@{}c@{}}$2$ \\ $11$\end{tabular} & \begin{tabular}[c]{@{}r@{}}1 \\ 1\end{tabular} & \begin{tabular}[c]{@{}r@{}}38 \\ 65\end{tabular}\\
\hline
40 & 10 & $A_3{+}2A_1$ & $\ZZ/2$ & 7 & No & \begin{tabular}[c]{@{}r@{}}41 \\ 62\end{tabular} & \begin{tabular}[c]{@{}c@{}}$2$ \\ $11$\end{tabular} & \begin{tabular}[c]{@{}r@{}}1 \\ 1\end{tabular} & \begin{tabular}[c]{@{}r@{}}7 \\ 59\end{tabular}\\
\hline
41 & 10 & $(A_3{+}A_1)'$ & $1$ & 26 & No & \begin{tabular}[c]{@{}r@{}}42\end{tabular} & \begin{tabular}[c]{@{}c@{}}$1$\end{tabular} & \begin{tabular}[c]{@{}r@{}}1\end{tabular} & \begin{tabular}[c]{@{}r@{}}48\end{tabular}\\
\hline
42 & 10 & $A_3$ & $1$ & 18 & No & \begin{tabular}[c]{@{}r@{}}40\end{tabular} & \begin{tabular}[c]{@{}c@{}}$1$\end{tabular} & \begin{tabular}[c]{@{}r@{}}1\end{tabular} & \begin{tabular}[c]{@{}r@{}}16\end{tabular}\\
\hline
43 & 11 & $(A_3{+}A_1)'$ & $1$ & 15 & No & \begin{tabular}[c]{@{}r@{}}49\end{tabular} & \begin{tabular}[c]{@{}c@{}}$1$\end{tabular} & \begin{tabular}[c]{@{}r@{}}1\end{tabular} & \begin{tabular}[c]{@{}r@{}}15\end{tabular}\\
\hline
44 & 11 & $D_4(a_1)$ & $\ZZ/2$ & 4 & No & \begin{tabular}[c]{@{}r@{}}46 \\ 47\end{tabular} & \begin{tabular}[c]{@{}c@{}}$11$ \\ $2$\end{tabular} & \begin{tabular}[c]{@{}r@{}}1 \\ 1\end{tabular} & \begin{tabular}[c]{@{}r@{}}26 \\ 5\end{tabular}\\
\hline
45 & 11 & $A_3{+}2A_1$ & $\ZZ/2$ & 16 & No & \begin{tabular}[c]{@{}r@{}}48 \\ 63\end{tabular} & \begin{tabular}[c]{@{}c@{}}$2$ \\ $11$\end{tabular} & \begin{tabular}[c]{@{}r@{}}1 \\ 1\end{tabular} & \begin{tabular}[c]{@{}r@{}}42 \\ 60\end{tabular}\\
\hline
46 & 11 & $A_3{+}A_2$ & $1$ & 9 & No & \begin{tabular}[c]{@{}r@{}}45\end{tabular} & \begin{tabular}[c]{@{}c@{}}$1$\end{tabular} & \begin{tabular}[c]{@{}r@{}}1\end{tabular} & \begin{tabular}[c]{@{}r@{}}22\end{tabular}\\
\hline
47 & 12 & $A_3{+}A_2{+}A_1$ & $\ZZ/2$ & 6 & No & \begin{tabular}[c]{@{}r@{}}51 \\ 66\end{tabular} & \begin{tabular}[c]{@{}c@{}}$2$ \\ $11$\end{tabular} & \begin{tabular}[c]{@{}r@{}}1 \\ 1\end{tabular} & \begin{tabular}[c]{@{}r@{}}17 \\ 57\end{tabular}\\
\hline
48 & 12 & $A_4$ & $\ZZ/2$ & 2 & Yes & \begin{tabular}[c]{@{}r@{}}50 \\ 52\end{tabular} & \begin{tabular}[c]{@{}c@{}}$11$ \\ $2$\end{tabular} & \begin{tabular}[c]{@{}r@{}}1 \\ 1\end{tabular} & \begin{tabular}[c]{@{}r@{}}10 \\ 2\end{tabular}\\
\hline
49 & 12 & $D_4(a_1){+}A_1$ & $(\ZZ/2)^2$ & 3 & No & \begin{tabular}[c]{@{}r@{}}53 \\ 54 \\ 64 \\ 65\end{tabular} & \begin{tabular}[c]{@{}c@{}}$(11,2)$ \\ $(2,2)$ \\ $(11,11)$ \\ $(2,11)$\end{tabular} & \begin{tabular}[c]{@{}r@{}}1 \\ 1 \\ 1 \\ 1\end{tabular} & \begin{tabular}[c]{@{}r@{}}18 \\ 3 \\ 58 \\ 61\end{tabular}\\
\hline
50 & 13 & $A_4{+}A_1$ & $\ZZ/2$ & 1 & Yes & \begin{tabular}[c]{@{}r@{}}55 \\ 56\end{tabular} & \begin{tabular}[c]{@{}c@{}}$2$ \\ $11$\end{tabular} & \begin{tabular}[c]{@{}r@{}}1 \\ 1\end{tabular} & \begin{tabular}[c]{@{}r@{}}1 \\ 6\end{tabular}\\
\end{longtable}
\endgroup

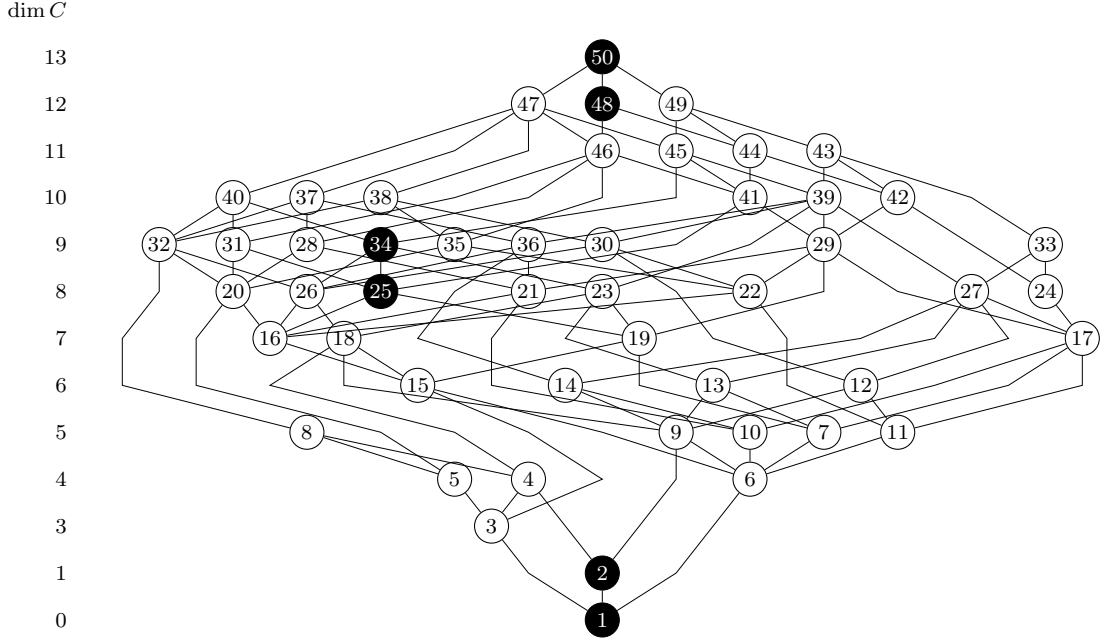
\begin{figure}[!htbp]
\centering
\begin{tikzpicture}[x=1cm,y=0.97cm,every node/.style={font=\scriptsize},
 art/.style={circle,draw,fill=black,text=white,inner sep=1pt,minimum size=4.5mm},
 ordinary/.style={circle,draw,fill=white,inner sep=1pt,minimum size=4.5mm}]
\node[art] (n1) at (0.0000,0.0000) {$ 1 $};
\coordinate (nd1l1) at (-0.9769,0.6400);
\node[art] (n2) at (0.0000,0.6400) {$ 2 $};
\coordinate (nd2l1) at (0.9769,0.6400);
\node[ordinary] (n3) at (-1.4654,1.2800) {$ 3 $};
\coordinate (nd3l2) at (-0.4885,1.2800);
\coordinate (nd4l2) at (0.4885,1.2800);
\coordinate (nd2l2) at (1.4654,1.2800);
\node[ordinary] (n5) at (-1.9538,1.9200) {$ 5 $};
\node[ordinary] (n4) at (-0.9769,1.9200) {$ 4 $};
\coordinate (nd7l3) at (0.0000,1.9200);
\coordinate (nd4l3) at (0.9769,1.9200);
\node[ordinary] (n6) at (1.9538,1.9200) {$ 6 $};
\node[ordinary] (n8) at (-3.9077,2.5600) {$ 8 $};
\coordinate (nd11l4) at (-2.9308,2.5600);
\coordinate (nd9l4) at (-1.9538,2.5600);
\coordinate (nd7l4) at (-0.9769,2.5600);
\coordinate (nd16l4) at (0.0000,2.5600);
\node[ordinary] (n9) at (0.9769,2.5600) {$ 9 $};
\node[ordinary] (n10) at (1.9538,2.5600) {$ 10 $};
\node[ordinary] (n7) at (2.9308,2.5600) {$ 7 $};
\node[ordinary] (n11) at (3.9077,2.5600) {$ 11 $};
\coordinate (nd20l5) at (-6.3500,3.2000);
\coordinate (nd11l5) at (-5.3731,3.2000);
\coordinate (nd9l5) at (-4.3962,3.2000);
\coordinate (nd24l5) at (-3.4192,3.2000);
\node[ordinary] (n15) at (-2.4423,3.2000) {$ 15 $};
\coordinate (nd27l5) at (-1.4654,3.2000);
\node[ordinary] (n14) at (-0.4885,3.2000) {$ 14 $};
\coordinate (nd19l5) at (0.4885,3.2000);
\node[ordinary] (n13) at (1.4654,3.2000) {$ 13 $};
\coordinate (nd30l5) at (2.4423,3.2000);
\node[ordinary] (n12) at (3.4192,3.2000) {$ 12 $};
\coordinate (nd26l5) at (4.3962,3.2000);
\coordinate (nd18l5) at (5.3731,3.2000);
\coordinate (nd29l5) at (6.3500,3.2000);
\coordinate (nd20l6) at (-6.3500,3.8400);
\coordinate (nd11l6) at (-5.3731,3.8400);
\node[ordinary] (n16) at (-4.3962,3.8400) {$ 16 $};
\node[ordinary] (n18) at (-3.4192,3.8400) {$ 18 $};
\coordinate (nd36l6) at (-2.4423,3.8400);
\coordinate (nd27l6) at (-1.4654,3.8400);
\coordinate (nd33l6) at (-0.4885,3.8400);
\node[ordinary] (n19) at (0.4885,3.8400) {$ 19 $};
\coordinate (nd32l6) at (1.4654,3.8400);
\coordinate (nd30l6) at (2.4423,3.8400);
\coordinate (nd35l6) at (3.4192,3.8400);
\coordinate (nd34l6) at (4.3962,3.8400);
\coordinate (nd31l6) at (5.3731,3.8400);
\node[ordinary] (n17) at (6.3500,3.8400) {$ 17 $};
\coordinate (nd20l7) at (-5.8615,4.4800);
\node[ordinary] (n20) at (-4.8846,4.4800) {$ 20 $};
\node[ordinary] (n26) at (-3.9077,4.4800) {$ 26 $};
\node[art] (n25) at (-2.9308,4.4800) {$ 25 $};
\coordinate (nd36l7) at (-1.9538,4.4800);
\node[ordinary] (n21) at (-0.9769,4.4800) {$ 21 $};
\node[ordinary] (n23) at (0.0000,4.4800) {$ 23 $};
\coordinate (nd32l7) at (0.9769,4.4800);
\node[ordinary] (n22) at (1.9538,4.4800) {$ 22 $};
\coordinate (nd52l7) at (2.9308,4.4800);
\coordinate (nd47l7) at (3.9077,4.4800);
\node[ordinary] (n27) at (4.8846,4.4800) {$ 27 $};
\node[ordinary] (n24) at (5.8615,4.4800) {$ 24 $};
\node[ordinary] (n32) at (-5.8615,5.1200) {$ 32 $};
\node[ordinary] (n31) at (-4.8846,5.1200) {$ 31 $};
\node[ordinary] (n28) at (-3.9077,5.1200) {$ 28 $};
\node[art] (n34) at (-2.9308,5.1200) {$ 34 $};
\node[ordinary] (n35) at (-1.9538,5.1200) {$ 35 $};
\node[ordinary] (n36) at (-0.9769,5.1200) {$ 36 $};
\node[ordinary] (n30) at (0.0000,5.1200) {$ 30 $};
\coordinate (nd69l8) at (0.9769,5.1200);
\coordinate (nd64l8) at (1.9538,5.1200);
\node[ordinary] (n29) at (2.9308,5.1200) {$ 29 $};
\coordinate (nd75l8) at (3.9077,5.1200);
\coordinate (nd66l8) at (4.8846,5.1200);
\node[ordinary] (n33) at (5.8615,5.1200) {$ 33 $};
\node[ordinary] (n40) at (-4.8846,5.7600) {$ 40 $};
\node[ordinary] (n37) at (-3.9077,5.7600) {$ 37 $};
\node[ordinary] (n38) at (-2.9308,5.7600) {$ 38 $};
\coordinate (nd84l9) at (-1.9538,5.7600);
\coordinate (nd77l9) at (-0.9769,5.7600);
\coordinate (nd92l9) at (0.0000,5.7600);
\coordinate (nd90l9) at (0.9769,5.7600);
\node[ordinary] (n41) at (1.9538,5.7600) {$ 41 $};
\node[ordinary] (n39) at (2.9308,5.7600) {$ 39 $};
\node[ordinary] (n42) at (3.9077,5.7600) {$ 42 $};
\coordinate (nd88l9) at (4.8846,5.7600);
\coordinate (nd99l10) at (-2.9308,6.4000);
\coordinate (nd95l10) at (-1.9538,6.4000);
\coordinate (nd96l10) at (-0.9769,6.4000);
\node[ordinary] (n46) at (0.0000,6.4000) {$ 46 $};
\node[ordinary] (n45) at (0.9769,6.4000) {$ 45 $};
\node[ordinary] (n44) at (1.9538,6.4000) {$ 44 $};
\node[ordinary] (n43) at (2.9308,6.4000) {$ 43 $};
\node[ordinary] (n47) at (-0.9769,7.0400) {$ 47 $};
\node[art] (n48) at (0.0000,7.0400) {$ 48 $};
\node[ordinary] (n49) at (0.9769,7.0400) {$ 49 $};
\node[art] (n50) at (0.0000,7.6800) {$ 50 $};
\draw[line width=.35pt] (n1) -- (n2);
\draw[line width=.35pt] (n1) -- (nd1l1) -- (n3);
\draw[line width=.35pt] (n1) -- (nd2l1) -- (nd2l2) -- (n6);
\draw[line width=.35pt] (n2) -- (nd3l2) -- (n4);
\draw[line width=.35pt] (n2) -- (nd4l2) -- (nd4l3) -- (n9);
\draw[line width=.35pt] (n3) -- (n4);
\draw[line width=.35pt] (n3) -- (n5);
\draw[line width=.35pt] (n3) -- (nd7l3) -- (nd7l4) -- (n15);
\draw[line width=.35pt] (n4) -- (n8);
\draw[line width=.35pt] (n4) -- (nd9l4) -- (nd9l5) -- (n18);
\draw[line width=.35pt] (n5) -- (n8);
\draw[line width=.35pt] (n5) -- (nd11l4) -- (nd11l5) -- (nd11l6) -- (n20);
\draw[line width=.35pt] (n6) -- (n7);
\draw[line width=.35pt] (n6) -- (n9);
\draw[line width=.35pt] (n6) -- (n10);
\draw[line width=.35pt] (n6) -- (n11);
\draw[line width=.35pt] (n6) -- (nd16l4) -- (n15);
\draw[line width=.35pt] (n7) -- (n13);
\draw[line width=.35pt] (n7) -- (nd18l5) -- (n17);
\draw[line width=.35pt] (n7) -- (nd19l5) -- (n19);
\draw[line width=.35pt] (n8) -- (nd20l5) -- (nd20l6) -- (nd20l7) -- (n32);
\draw[line width=.35pt] (n9) -- (n12);
\draw[line width=.35pt] (n9) -- (n13);
\draw[line width=.35pt] (n9) -- (n14);
\draw[line width=.35pt] (n9) -- (nd24l5) -- (n18);
\draw[line width=.35pt] (n10) -- (n14);
\draw[line width=.35pt] (n10) -- (nd26l5) -- (n17);
\draw[line width=.35pt] (n10) -- (nd27l5) -- (nd27l6) -- (n21);
\draw[line width=.35pt] (n11) -- (n12);
\draw[line width=.35pt] (n11) -- (nd29l5) -- (n17);
\draw[line width=.35pt] (n11) -- (nd30l5) -- (nd30l6) -- (n22);
\draw[line width=.35pt] (n12) -- (nd31l6) -- (n27);
\draw[line width=.35pt] (n12) -- (nd32l6) -- (nd32l7) -- (n30);
\draw[line width=.35pt] (n13) -- (nd33l6) -- (n23);
\draw[line width=.35pt] (n13) -- (nd34l6) -- (n27);
\draw[line width=.35pt] (n14) -- (nd35l6) -- (n27);
\draw[line width=.35pt] (n14) -- (nd36l6) -- (nd36l7) -- (n36);
\draw[line width=.35pt] (n15) -- (n16);
\draw[line width=.35pt] (n15) -- (n18);
\draw[line width=.35pt] (n15) -- (n19);
\draw[line width=.35pt] (n16) -- (n20);
\draw[line width=.35pt] (n16) -- (n21);
\draw[line width=.35pt] (n16) -- (n22);
\draw[line width=.35pt] (n16) -- (n25);
\draw[line width=.35pt] (n16) -- (n26);
\draw[line width=.35pt] (n17) -- (n24);
\draw[line width=.35pt] (n17) -- (n27);
\draw[line width=.35pt] (n17) -- (nd47l7) -- (n29);
\draw[line width=.35pt] (n18) -- (n23);
\draw[line width=.35pt] (n18) -- (n26);
\draw[line width=.35pt] (n19) -- (n23);
\draw[line width=.35pt] (n19) -- (n25);
\draw[line width=.35pt] (n19) -- (nd52l7) -- (n29);
\draw[line width=.35pt] (n20) -- (n28);
\draw[line width=.35pt] (n20) -- (n31);
\draw[line width=.35pt] (n20) -- (n32);
\draw[line width=.35pt] (n20) -- (n35);
\draw[line width=.35pt] (n21) -- (n28);
\draw[line width=.35pt] (n21) -- (n29);
\draw[line width=.35pt] (n21) -- (n36);
\draw[line width=.35pt] (n22) -- (n29);
\draw[line width=.35pt] (n22) -- (n30);
\draw[line width=.35pt] (n22) -- (n35);
\draw[line width=.35pt] (n23) -- (n34);
\draw[line width=.35pt] (n23) -- (nd64l8) -- (n39);
\draw[line width=.35pt] (n24) -- (n33);
\draw[line width=.35pt] (n24) -- (nd66l8) -- (n42);
\draw[line width=.35pt] (n25) -- (n31);
\draw[line width=.35pt] (n25) -- (n34);
\draw[line width=.35pt] (n25) -- (nd69l8) -- (n41);
\draw[line width=.35pt] (n26) -- (n30);
\draw[line width=.35pt] (n26) -- (n32);
\draw[line width=.35pt] (n26) -- (n34);
\draw[line width=.35pt] (n26) -- (n36);
\draw[line width=.35pt] (n27) -- (n33);
\draw[line width=.35pt] (n27) -- (nd75l8) -- (n39);
\draw[line width=.35pt] (n28) -- (n37);
\draw[line width=.35pt] (n28) -- (nd77l9) -- (n46);
\draw[line width=.35pt] (n29) -- (n39);
\draw[line width=.35pt] (n29) -- (n41);
\draw[line width=.35pt] (n29) -- (n42);
\draw[line width=.35pt] (n30) -- (n38);
\draw[line width=.35pt] (n30) -- (n39);
\draw[line width=.35pt] (n31) -- (n40);
\draw[line width=.35pt] (n31) -- (nd84l9) -- (n46);
\draw[line width=.35pt] (n32) -- (n37);
\draw[line width=.35pt] (n32) -- (n38);
\draw[line width=.35pt] (n32) -- (n40);
\draw[line width=.35pt] (n33) -- (nd88l9) -- (n43);
\draw[line width=.35pt] (n34) -- (n40);
\draw[line width=.35pt] (n34) -- (nd90l9) -- (n45);
\draw[line width=.35pt] (n35) -- (n38);
\draw[line width=.35pt] (n35) -- (nd92l9) -- (n46);
\draw[line width=.35pt] (n36) -- (n37);
\draw[line width=.35pt] (n36) -- (n39);
\draw[line width=.35pt] (n37) -- (nd95l10) -- (n47);
\draw[line width=.35pt] (n38) -- (nd96l10) -- (n47);
\draw[line width=.35pt] (n39) -- (n43);
\draw[line width=.35pt] (n39) -- (n45);
\draw[line width=.35pt] (n40) -- (nd99l10) -- (n47);
\draw[line width=.35pt] (n41) -- (n44);
\draw[line width=.35pt] (n41) -- (n45);
\draw[line width=.35pt] (n41) -- (n46);
\draw[line width=.35pt] (n42) -- (n43);
\draw[line width=.35pt] (n42) -- (n44);
\draw[line width=.35pt] (n43) -- (n49);
\draw[line width=.35pt] (n44) -- (n48);
\draw[line width=.35pt] (n44) -- (n49);
\draw[line width=.35pt] (n45) -- (n47);
\draw[line width=.35pt] (n45) -- (n49);
\draw[line width=.35pt] (n46) -- (n47);
\draw[line width=.35pt] (n46) -- (n48);
\draw[line width=.35pt] (n47) -- (n50);
\draw[line width=.35pt] (n48) -- (n50);
\draw[line width=.35pt] (n49) -- (n50);
\node[anchor=east] at (-6.9500,0.0000) {$ 0 $};
\node[anchor=east] at (-6.9500,0.6400) {$ 1 $};
\node[anchor=east] at (-6.9500,1.2800) {$ 3 $};
\node[anchor=east] at (-6.9500,1.9200) {$ 4 $};
\node[anchor=east] at (-6.9500,2.5600) {$ 5 $};
\node[anchor=east] at (-6.9500,3.2000) {$ 6 $};
\node[anchor=east] at (-6.9500,3.8400) {$ 7 $};
\node[anchor=east] at (-6.9500,4.4800) {$ 8 $};
\node[anchor=east] at (-6.9500,5.1200) {$ 9 $};
\node[anchor=east] at (-6.9500,5.7600) {$ 10 $};
\node[anchor=east] at (-6.9500,6.4000) {$ 11 $};
\node[anchor=east] at (-6.9500,7.0400) {$ 12 $};
\node[anchor=east] at (-6.9500,7.6800) {$ 13 $};
\node[anchor=east] at (-6.9500,8.3200) {$\dim C$};
\end{tikzpicture}
\caption{Closure order for the $A_4+A_1$ grading in $E_7$.
Vertices are labelled by $i$, and filled vertices denote Arthur orbits.
The order increases upwards; the left column gives $\dim C_i$.}
\label{fig:hasse-e7}
\end{figure}

Figure~\ref{fig:hasse-e7} gives the full closure order. The parameters
$1,\ldots,56$ belong to the principal series. The remaining ten
parameters belong to the generalized Springer series whose cuspidal
support has derived Levi subgroup of type $A_1\times A_1\times A_1$,
with simple roots $\alpha_2,\alpha_5,\alpha_7$. Their ambient central
character is nontrivial. In particular, for the split adjoint form one
retains precisely the parameters $1,\ldots,56$.

Using the IC polynomials of \cite{CiubotaruOkadaKL}, the reconstruction
gives the following complete table of Arthur rows, including both
central characters:
\begin{longtable}{rl}
$i$ & $R_i$\\ \hline
\endhead
1 & $\pi_{1}+\pi_{6}$\\
2 & $\pi_{2}+\pi_{10}$\\
25 & $\pi_{20}+\pi_{42}+\pi_{60}+\pi_{63}$\\
34 & $\pi_{34}+\pi_{48}+\pi_{60}+\pi_{63}$\\
48 & $\pi_{50}+\pi_{52}$\\
50 & $\pi_{55}+\pi_{56}$\\
\end{longtable}
Thus the union of the six packets contains fourteen representations,
and all nonzero ordinary microlocal multiplicities are one. The
non-principal contributions are $\pi_{60}$ and $\pi_{63}$, supported
on $C_{34}$ and $C_{45}$, respectively; both occur in $R_{25}$ and
$R_{34}$. Applying $f$ to the constituent indices exchanges the rows
indexed by each of the three Piasetskii-dual pairs. In particular,
$f(60)=63$ exchanges the two non-principal constituents.

\subsection{The \texorpdfstring{$E_8(a_7)$}{E8(a7)} grading}
\label{subsec:example-e8}

Here $\hd=(0,0,0,0,2,0,0,0)$ and $\dim V=40$. There are $63$ graded
orbits, $88$ IC parameters, and $34$ Arthur orbits. The union of the
Arthur packets contains $58$ representations. The open orbit $C_{63}$ has component group $S_5$;
its seven enhancements have indices $82,\ldots,88$, with respective labels
$5,41,32,311,221,2111,11111$ and ranks $1,4,5,6,5,4,1$.
The enhancement $a=88$ is cuspidal, and $\mathbf M_{i,88}=1$ for every
Arthur orbit $C_i$.

Figure~\ref{fig:hasse-e8} gives the full closure order on all $63$ graded orbits.

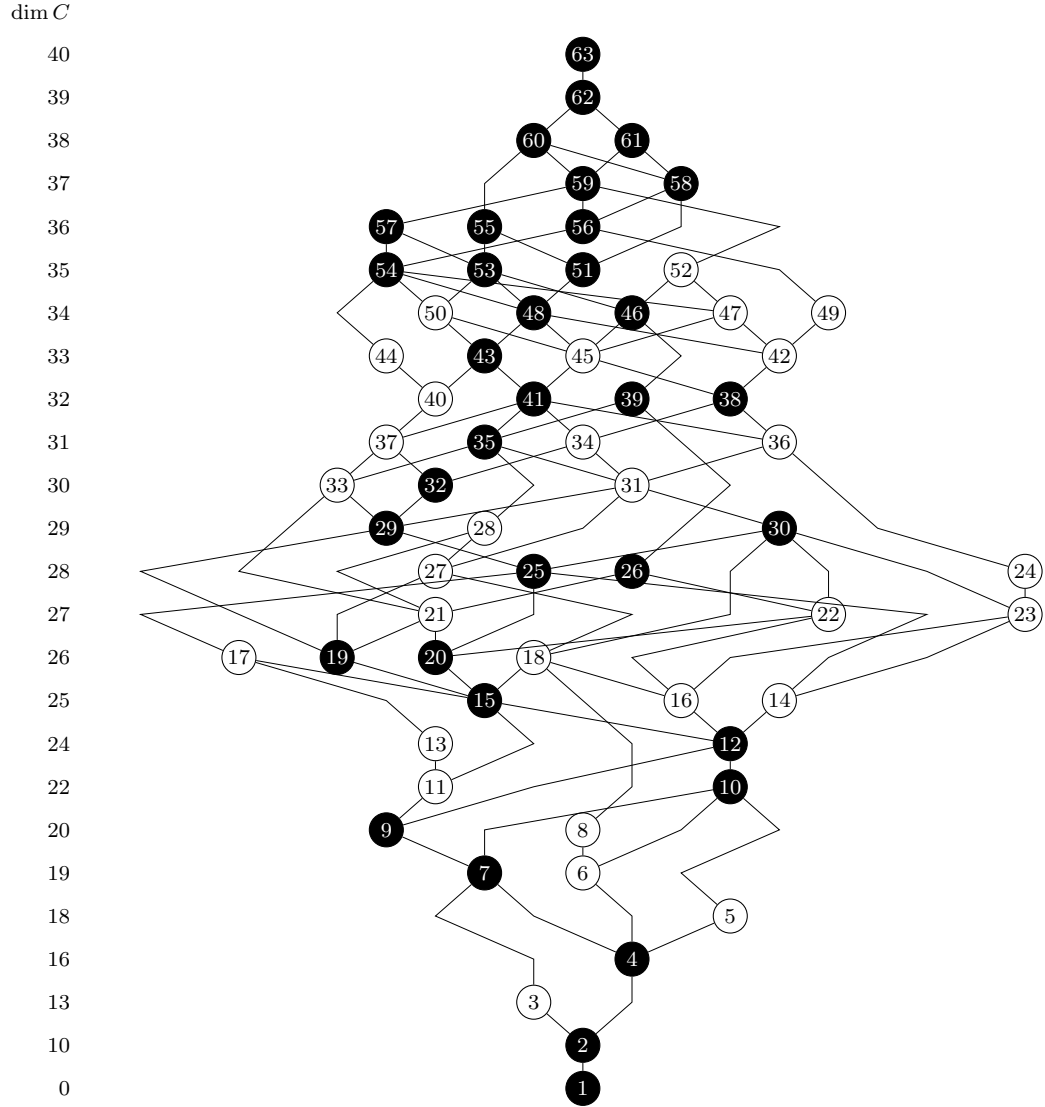
\begin{figure}[p]
\centering
\begin{tikzpicture}[x=1cm,y=1cm, every node/.style={font=\scriptsize},
 art/.style={circle,draw,fill=black,text=white,inner sep=1pt,minimum size=4.5mm},
 ordinary/.style={circle,draw,fill=white,inner sep=1pt,minimum size=4.5mm}]
\node[art] (n1) at (0.000,0.000) {$ 1 $};
\node[art] (n2) at (0.000,0.570) {$ 2 $};
\node[ordinary] (n3) at (-0.650,1.140) {$ 3 $};
\coordinate (nd2l2) at (0.650,1.140);
\coordinate (nd3l3) at (-0.650,1.710);
\node[art] (n4) at (0.650,1.710) {$ 4 $};
\coordinate (nd3l4) at (-1.950,2.280);
\coordinate (nd6l4) at (-0.650,2.280);
\coordinate (nd5l4) at (0.650,2.280);
\node[ordinary] (n5) at (1.950,2.280) {$ 5 $};
\node[art] (n7) at (-1.300,2.850) {$ 7 $};
\node[ordinary] (n6) at (0.000,2.850) {$ 6 $};
\coordinate (nd7l5) at (1.300,2.850);
\node[art] (n9) at (-2.600,3.420) {$ 9 $};
\coordinate (nd11l6) at (-1.300,3.420);
\node[ordinary] (n8) at (0.000,3.420) {$ 8 $};
\coordinate (nd9l6) at (1.300,3.420);
\coordinate (nd7l6) at (2.600,3.420);
\node[ordinary] (n11) at (-1.950,3.990) {$ 11 $};
\coordinate (nd14l7) at (-0.650,3.990);
\coordinate (nd12l7) at (0.650,3.990);
\node[art] (n10) at (1.950,3.990) {$ 10 $};
\node[ordinary] (n13) at (-1.950,4.560) {$ 13 $};
\coordinate (nd17l8) at (-0.650,4.560);
\coordinate (nd12l8) at (0.650,4.560);
\node[art] (n12) at (1.950,4.560) {$ 12 $};
\coordinate (nd21l9) at (-2.600,5.130);
\node[art] (n15) at (-1.300,5.130) {$ 15 $};
\coordinate (nd12l9) at (0.000,5.130);
\node[ordinary] (n16) at (1.300,5.130) {$ 16 $};
\node[ordinary] (n14) at (2.600,5.130) {$ 14 $};
\node[ordinary] (n17) at (-4.550,5.700) {$ 17 $};
\node[art] (n19) at (-3.250,5.700) {$ 19 $};
\node[art] (n20) at (-1.950,5.700) {$ 20 $};
\node[ordinary] (n18) at (-0.650,5.700) {$ 18 $};
\coordinate (nd29l10) at (0.650,5.700);
\coordinate (nd30l10) at (1.950,5.700);
\coordinate (nd23l10) at (3.250,5.700);
\coordinate (nd22l10) at (4.550,5.700);
\coordinate (nd31l11) at (-5.850,6.270);
\coordinate (nd36l11) at (-4.550,6.270);
\coordinate (nd35l11) at (-3.250,6.270);
\node[ordinary] (n21) at (-1.950,6.270) {$ 21 $};
\coordinate (nd39l11) at (-0.650,6.270);
\coordinate (nd32l11) at (0.650,6.270);
\coordinate (nd33l11) at (1.950,6.270);
\node[ordinary] (n22) at (3.250,6.270) {$ 22 $};
\coordinate (nd23l11) at (4.550,6.270);
\node[ordinary] (n23) at (5.850,6.270) {$ 23 $};
\coordinate (nd36l12) at (-5.850,6.840);
\coordinate (nd42l12) at (-4.550,6.840);
\coordinate (nd41l12) at (-3.250,6.840);
\node[ordinary] (n27) at (-1.950,6.840) {$ 27 $};
\node[art] (n25) at (-0.650,6.840) {$ 25 $};
\node[art] (n26) at (0.650,6.840) {$ 26 $};
\coordinate (nd33l12) at (1.950,6.840);
\coordinate (nd44l12) at (3.250,6.840);
\coordinate (nd46l12) at (4.550,6.840);
\node[ordinary] (n24) at (5.850,6.840) {$ 24 $};
\coordinate (nd42l13) at (-3.900,7.410);
\node[art] (n29) at (-2.600,7.410) {$ 29 $};
\node[ordinary] (n28) at (-1.300,7.410) {$ 28 $};
\coordinate (nd52l13) at (0.000,7.410);
\coordinate (nd50l13) at (1.300,7.410);
\node[art] (n30) at (2.600,7.410) {$ 30 $};
\coordinate (nd47l13) at (3.900,7.410);
\node[ordinary] (n33) at (-3.250,7.980) {$ 33 $};
\node[art] (n32) at (-1.950,7.980) {$ 32 $};
\coordinate (nd53l14) at (-0.650,7.980);
\node[ordinary] (n31) at (0.650,7.980) {$ 31 $};
\coordinate (nd50l14) at (1.950,7.980);
\coordinate (nd47l14) at (3.250,7.980);
\node[ordinary] (n37) at (-2.600,8.550) {$ 37 $};
\node[art] (n35) at (-1.300,8.550) {$ 35 $};
\node[ordinary] (n34) at (0.000,8.550) {$ 34 $};
\coordinate (nd50l15) at (1.300,8.550);
\node[ordinary] (n36) at (2.600,8.550) {$ 36 $};
\node[ordinary] (n40) at (-1.950,9.120) {$ 40 $};
\node[art] (n41) at (-0.650,9.120) {$ 41 $};
\node[art] (n39) at (0.650,9.120) {$ 39 $};
\node[art] (n38) at (1.950,9.120) {$ 38 $};
\node[ordinary] (n44) at (-2.600,9.690) {$ 44 $};
\node[art] (n43) at (-1.300,9.690) {$ 43 $};
\node[ordinary] (n45) at (0.000,9.690) {$ 45 $};
\coordinate (nd75l17) at (1.300,9.690);
\node[ordinary] (n42) at (2.600,9.690) {$ 42 $};
\coordinate (nd85l18) at (-3.250,10.260);
\node[ordinary] (n50) at (-1.950,10.260) {$ 50 $};
\node[art] (n48) at (-0.650,10.260) {$ 48 $};
\node[art] (n46) at (0.650,10.260) {$ 46 $};
\node[ordinary] (n47) at (1.950,10.260) {$ 47 $};
\node[ordinary] (n49) at (3.250,10.260) {$ 49 $};
\node[art] (n54) at (-2.600,10.830) {$ 54 $};
\node[art] (n53) at (-1.300,10.830) {$ 53 $};
\node[art] (n51) at (0.000,10.830) {$ 51 $};
\node[ordinary] (n52) at (1.300,10.830) {$ 52 $};
\coordinate (nd97l19) at (2.600,10.830);
\node[art] (n57) at (-2.600,11.400) {$ 57 $};
\node[art] (n55) at (-1.300,11.400) {$ 55 $};
\node[art] (n56) at (0.000,11.400) {$ 56 $};
\coordinate (nd101l20) at (1.300,11.400);
\coordinate (nd102l20) at (2.600,11.400);
\coordinate (nd107l21) at (-1.300,11.970);
\node[art] (n59) at (0.000,11.970) {$ 59 $};
\node[art] (n58) at (1.300,11.970) {$ 58 $};
\node[art] (n60) at (-0.650,12.540) {$ 60 $};
\node[art] (n61) at (0.650,12.540) {$ 61 $};
\node[art] (n62) at (0.000,13.110) {$ 62 $};
\node[art] (n63) at (0.000,13.680) {$ 63 $};
\draw[line width=.35pt] (n1) -- (n2);
\draw[line width=.35pt] (n2) -- (n3);
\draw[line width=.35pt] (n2) -- (nd2l2) -- (n4);
\draw[line width=.35pt] (n3) -- (nd3l3) -- (nd3l4) -- (n7);
\draw[line width=.35pt] (n4) -- (n5);
\draw[line width=.35pt] (n4) -- (nd5l4) -- (n6);
\draw[line width=.35pt] (n4) -- (nd6l4) -- (n7);
\draw[line width=.35pt] (n5) -- (nd7l5) -- (nd7l6) -- (n10);
\draw[line width=.35pt] (n6) -- (n8);
\draw[line width=.35pt] (n6) -- (nd9l6) -- (n10);
\draw[line width=.35pt] (n7) -- (n9);
\draw[line width=.35pt] (n7) -- (nd11l6) -- (n10);
\draw[line width=.35pt] (n8) -- (nd12l7) -- (nd12l8) -- (nd12l9) -- (n18);
\draw[line width=.35pt] (n9) -- (n11);
\draw[line width=.35pt] (n9) -- (nd14l7) -- (n12);
\draw[line width=.35pt] (n10) -- (n12);
\draw[line width=.35pt] (n11) -- (n13);
\draw[line width=.35pt] (n11) -- (nd17l8) -- (n15);
\draw[line width=.35pt] (n12) -- (n14);
\draw[line width=.35pt] (n12) -- (n15);
\draw[line width=.35pt] (n12) -- (n16);
\draw[line width=.35pt] (n13) -- (nd21l9) -- (n17);
\draw[line width=.35pt] (n14) -- (nd22l10) -- (n23);
\draw[line width=.35pt] (n14) -- (nd23l10) -- (nd23l11) -- (n25);
\draw[line width=.35pt] (n15) -- (n17);
\draw[line width=.35pt] (n15) -- (n18);
\draw[line width=.35pt] (n15) -- (n19);
\draw[line width=.35pt] (n15) -- (n20);
\draw[line width=.35pt] (n16) -- (n18);
\draw[line width=.35pt] (n16) -- (nd29l10) -- (n22);
\draw[line width=.35pt] (n16) -- (nd30l10) -- (n23);
\draw[line width=.35pt] (n17) -- (nd31l11) -- (n25);
\draw[line width=.35pt] (n18) -- (nd32l11) -- (n27);
\draw[line width=.35pt] (n18) -- (nd33l11) -- (nd33l12) -- (n30);
\draw[line width=.35pt] (n19) -- (n21);
\draw[line width=.35pt] (n19) -- (nd35l11) -- (n27);
\draw[line width=.35pt] (n19) -- (nd36l11) -- (nd36l12) -- (n29);
\draw[line width=.35pt] (n20) -- (n21);
\draw[line width=.35pt] (n20) -- (n22);
\draw[line width=.35pt] (n20) -- (nd39l11) -- (n25);
\draw[line width=.35pt] (n21) -- (n26);
\draw[line width=.35pt] (n21) -- (nd41l12) -- (n28);
\draw[line width=.35pt] (n21) -- (nd42l12) -- (nd42l13) -- (n33);
\draw[line width=.35pt] (n22) -- (n26);
\draw[line width=.35pt] (n22) -- (nd44l12) -- (n30);
\draw[line width=.35pt] (n23) -- (n24);
\draw[line width=.35pt] (n23) -- (nd46l12) -- (n30);
\draw[line width=.35pt] (n24) -- (nd47l13) -- (nd47l14) -- (n36);
\draw[line width=.35pt] (n25) -- (n29);
\draw[line width=.35pt] (n25) -- (n30);
\draw[line width=.35pt] (n26) -- (nd50l13) -- (nd50l14) -- (nd50l15) -- (n39);
\draw[line width=.35pt] (n27) -- (n28);
\draw[line width=.35pt] (n27) -- (nd52l13) -- (n31);
\draw[line width=.35pt] (n28) -- (nd53l14) -- (n35);
\draw[line width=.35pt] (n29) -- (n31);
\draw[line width=.35pt] (n29) -- (n32);
\draw[line width=.35pt] (n29) -- (n33);
\draw[line width=.35pt] (n30) -- (n31);
\draw[line width=.35pt] (n31) -- (n34);
\draw[line width=.35pt] (n31) -- (n35);
\draw[line width=.35pt] (n31) -- (n36);
\draw[line width=.35pt] (n32) -- (n34);
\draw[line width=.35pt] (n32) -- (n37);
\draw[line width=.35pt] (n33) -- (n35);
\draw[line width=.35pt] (n33) -- (n37);
\draw[line width=.35pt] (n34) -- (n38);
\draw[line width=.35pt] (n34) -- (n41);
\draw[line width=.35pt] (n35) -- (n39);
\draw[line width=.35pt] (n35) -- (n41);
\draw[line width=.35pt] (n36) -- (n38);
\draw[line width=.35pt] (n36) -- (n41);
\draw[line width=.35pt] (n37) -- (n40);
\draw[line width=.35pt] (n37) -- (n41);
\draw[line width=.35pt] (n38) -- (n42);
\draw[line width=.35pt] (n38) -- (n45);
\draw[line width=.35pt] (n39) -- (nd75l17) -- (n46);
\draw[line width=.35pt] (n40) -- (n43);
\draw[line width=.35pt] (n40) -- (n44);
\draw[line width=.35pt] (n41) -- (n43);
\draw[line width=.35pt] (n41) -- (n45);
\draw[line width=.35pt] (n42) -- (n47);
\draw[line width=.35pt] (n42) -- (n48);
\draw[line width=.35pt] (n42) -- (n49);
\draw[line width=.35pt] (n43) -- (n48);
\draw[line width=.35pt] (n43) -- (n50);
\draw[line width=.35pt] (n44) -- (nd85l18) -- (n54);
\draw[line width=.35pt] (n45) -- (n46);
\draw[line width=.35pt] (n45) -- (n47);
\draw[line width=.35pt] (n45) -- (n48);
\draw[line width=.35pt] (n45) -- (n50);
\draw[line width=.35pt] (n46) -- (n52);
\draw[line width=.35pt] (n46) -- (n53);
\draw[line width=.35pt] (n47) -- (n52);
\draw[line width=.35pt] (n47) -- (n54);
\draw[line width=.35pt] (n48) -- (n51);
\draw[line width=.35pt] (n48) -- (n53);
\draw[line width=.35pt] (n48) -- (n54);
\draw[line width=.35pt] (n49) -- (nd97l19) -- (n56);
\draw[line width=.35pt] (n50) -- (n53);
\draw[line width=.35pt] (n50) -- (n54);
\draw[line width=.35pt] (n51) -- (n55);
\draw[line width=.35pt] (n51) -- (nd101l20) -- (n58);
\draw[line width=.35pt] (n52) -- (nd102l20) -- (n59);
\draw[line width=.35pt] (n53) -- (n55);
\draw[line width=.35pt] (n53) -- (n57);
\draw[line width=.35pt] (n54) -- (n56);
\draw[line width=.35pt] (n54) -- (n57);
\draw[line width=.35pt] (n55) -- (nd107l21) -- (n60);
\draw[line width=.35pt] (n56) -- (n58);
\draw[line width=.35pt] (n56) -- (n59);
\draw[line width=.35pt] (n57) -- (n59);
\draw[line width=.35pt] (n58) -- (n60);
\draw[line width=.35pt] (n58) -- (n61);
\draw[line width=.35pt] (n59) -- (n60);
\draw[line width=.35pt] (n59) -- (n61);
\draw[line width=.35pt] (n60) -- (n62);
\draw[line width=.35pt] (n61) -- (n62);
\draw[line width=.35pt] (n62) -- (n63);
\node[anchor=east] at (-6.650,0.000) {$ 0 $};
\node[anchor=east] at (-6.650,0.570) {$ 10 $};
\node[anchor=east] at (-6.650,1.140) {$ 13 $};
\node[anchor=east] at (-6.650,1.710) {$ 16 $};
\node[anchor=east] at (-6.650,2.280) {$ 18 $};
\node[anchor=east] at (-6.650,2.850) {$ 19 $};
\node[anchor=east] at (-6.650,3.420) {$ 20 $};
\node[anchor=east] at (-6.650,3.990) {$ 22 $};
\node[anchor=east] at (-6.650,4.560) {$ 24 $};
\node[anchor=east] at (-6.650,5.130) {$ 25 $};
\node[anchor=east] at (-6.650,5.700) {$ 26 $};
\node[anchor=east] at (-6.650,6.270) {$ 27 $};
\node[anchor=east] at (-6.650,6.840) {$ 28 $};
\node[anchor=east] at (-6.650,7.410) {$ 29 $};
\node[anchor=east] at (-6.650,7.980) {$ 30 $};
\node[anchor=east] at (-6.650,8.550) {$ 31 $};
\node[anchor=east] at (-6.650,9.120) {$ 32 $};
\node[anchor=east] at (-6.650,9.690) {$ 33 $};
\node[anchor=east] at (-6.650,10.260) {$ 34 $};
\node[anchor=east] at (-6.650,10.830) {$ 35 $};
\node[anchor=east] at (-6.650,11.400) {$ 36 $};
\node[anchor=east] at (-6.650,11.970) {$ 37 $};
\node[anchor=east] at (-6.650,12.540) {$ 38 $};
\node[anchor=east] at (-6.650,13.110) {$ 39 $};
\node[anchor=east] at (-6.650,13.680) {$ 40 $};
\node[anchor=east] at (-6.650,14.250) {$\dim C$};
\end{tikzpicture}
\caption{Closure order for the $E_8(a_7)$ grading. Vertices are labelled by $i$, and filled vertices denote Arthur orbits. The order increases upwards; the left column gives $\dim C_i$. }\label{fig:hasse-e8}
\end{figure}

We display the rows for the zero and open orbits, the orbit $C_{38}$
fixed by $p$, and the five Arthur orbits used in
Section~\ref{subsec:extremality}:
\begin{longtable}{rl}
$i$ & $R_i$\\ \hline
\endhead
1 & $\pi_{1}+4\pi_{2}+5\pi_{4}+6\pi_{10}+5\pi_{13}+4\pi_{36}+\pi_{88}$\\
4 & $\pi_{4}+\pi_{10}+\pi_{12}+2\pi_{13}+2\pi_{16}+\pi_{35}+2\pi_{36}+2\pi_{45}+\pi_{88}$\\
10 & $\pi_{12}+\pi_{13}+\pi_{16}+\pi_{33}+\pi_{35}+\pi_{36}+2\pi_{45}+\pi_{51}+\pi_{88}$\\
15 & $\pi_{16}+\pi_{26}+\pi_{35}+\pi_{36}+\pi_{45}+\pi_{51}+\pi_{56}+\pi_{88}$\\
19 & $\pi_{20}+\pi_{34}+\pi_{35}+\pi_{36}+2\pi_{56}+\pi_{68}+\pi_{88}$\\
30 & $\pi_{33}+\pi_{35}+\pi_{45}+\pi_{51}+\pi_{56}+\pi_{69}+\pi_{88}$\\
38 & $\pi_{47}+\pi_{48}+4\pi_{69}+\pi_{88}$\\
63 & $\pi_{82}+4\pi_{83}+5\pi_{84}+6\pi_{85}+5\pi_{86}+4\pi_{87}+\pi_{88}$\\
\end{longtable}
\clearpage
The enhancements occurring in these rows are listed below, together
with their supporting orbits. The complete tables are included in the
supplementary data.
\begingroup\small
\setlength{\tabcolsep}{4pt}
\setlength{\arrayrulewidth}{0.25pt}
\renewcommand{\arraystretch}{1.2}
\begin{longtable}{@{}rrlcrc@{\hspace{12pt}}rcrr@{}}
\hline

\multicolumn{6}{c}{\textbf{Orbit data}} & \multicolumn{4}{c}{\textbf{IC dictionary}}\\[3pt]
$i$ & $\dim C_i$ & $\Gd C_i$ & $A_H(C_i)$ & $p(i)$ & Arthur type & $a$ & $\rho_a$ & $\dim\rho_a$ & $f(a)$\\[3pt]
\hline

\endhead
\hline\endfoot
1 & 0 & $0$ & $1$ & 63 & Yes & \begin{tabular}[c]{@{}r@{}}1\end{tabular} & \begin{tabular}[c]{@{}c@{}}$1$\end{tabular} & \begin{tabular}[c]{@{}r@{}}1\end{tabular} & \begin{tabular}[c]{@{}r@{}}82\end{tabular}\\
\hline
2 & 10 & $A_1$ & $1$ & 62 & Yes & \begin{tabular}[c]{@{}r@{}}2\end{tabular} & \begin{tabular}[c]{@{}c@{}}$1$\end{tabular} & \begin{tabular}[c]{@{}r@{}}1\end{tabular} & \begin{tabular}[c]{@{}r@{}}83\end{tabular}\\
\hline
4 & 16 & $2A_1$ & $1$ & 61 & Yes & \begin{tabular}[c]{@{}r@{}}4\end{tabular} & \begin{tabular}[c]{@{}c@{}}$1$\end{tabular} & \begin{tabular}[c]{@{}r@{}}1\end{tabular} & \begin{tabular}[c]{@{}r@{}}84\end{tabular}\\
\hline
9 & 20 & $A_2$ & $\ZZ/2$ & 55 & Yes & \begin{tabular}[c]{@{}r@{}}10\end{tabular} & \begin{tabular}[c]{@{}c@{}}$11$\end{tabular} & \begin{tabular}[c]{@{}r@{}}1\end{tabular} & \begin{tabular}[c]{@{}r@{}}85\end{tabular}\\
\hline
10 & 22 & $4A_1$ & $1$ & 58 & Yes & \begin{tabular}[c]{@{}r@{}}12\end{tabular} & \begin{tabular}[c]{@{}c@{}}$1$\end{tabular} & \begin{tabular}[c]{@{}r@{}}1\end{tabular} & \begin{tabular}[c]{@{}r@{}}78\end{tabular}\\
\hline
12 & 24 & $A_2{+}A_1$ & $\ZZ/2$ & 51 & Yes & \begin{tabular}[c]{@{}r@{}}13\end{tabular} & \begin{tabular}[c]{@{}c@{}}$2$\end{tabular} & \begin{tabular}[c]{@{}r@{}}1\end{tabular} & \begin{tabular}[c]{@{}r@{}}86\end{tabular}\\
\hline
15 & 25 & $A_2{+}2A_1$ & $1$ & 59 & Yes & \begin{tabular}[c]{@{}r@{}}16\end{tabular} & \begin{tabular}[c]{@{}c@{}}$1$\end{tabular} & \begin{tabular}[c]{@{}r@{}}1\end{tabular} & \begin{tabular}[c]{@{}r@{}}80\end{tabular}\\
\hline
19 & 26 & $A_3$ & $1$ & 46 & Yes & \begin{tabular}[c]{@{}r@{}}20\end{tabular} & \begin{tabular}[c]{@{}c@{}}$1$\end{tabular} & \begin{tabular}[c]{@{}r@{}}1\end{tabular} & \begin{tabular}[c]{@{}r@{}}58\end{tabular}\\
\hline
25 & 28 & $2A_2{+}A_1$ & $1$ & 53 & Yes & \begin{tabular}[c]{@{}r@{}}26\end{tabular} & \begin{tabular}[c]{@{}c@{}}$1$\end{tabular} & \begin{tabular}[c]{@{}r@{}}1\end{tabular} & \begin{tabular}[c]{@{}r@{}}77\end{tabular}\\
\hline
30 & 29 & $2A_2{+}2A_1$ & $1$ & 56 & Yes & \begin{tabular}[c]{@{}r@{}}33\end{tabular} & \begin{tabular}[c]{@{}c@{}}$1$\end{tabular} & \begin{tabular}[c]{@{}r@{}}1\end{tabular} & \begin{tabular}[c]{@{}r@{}}73\end{tabular}\\
\hline
31 & 30 & $A_3{+}2A_1$ & $1$ & 45 & No & \begin{tabular}[c]{@{}r@{}}35\end{tabular} & \begin{tabular}[c]{@{}c@{}}$1$\end{tabular} & \begin{tabular}[c]{@{}r@{}}1\end{tabular} & \begin{tabular}[c]{@{}r@{}}75\end{tabular}\\
\hline
32 & 30 & $D_4(a_1)$ & $S_3$ & 26 & Yes & \begin{tabular}[c]{@{}r@{}}34 \\ 36\end{tabular} & \begin{tabular}[c]{@{}c@{}}$21$ \\ $111$\end{tabular} & \begin{tabular}[c]{@{}r@{}}2 \\ 1\end{tabular} & \begin{tabular}[c]{@{}r@{}}63 \\ 87\end{tabular}\\
\hline
34 & 31 & $D_4(a_1){+}A_1$ & $S_3$ & 22 & No & \begin{tabular}[c]{@{}r@{}}45\end{tabular} & \begin{tabular}[c]{@{}c@{}}$111$\end{tabular} & \begin{tabular}[c]{@{}r@{}}1\end{tabular} & \begin{tabular}[c]{@{}r@{}}79\end{tabular}\\
\hline
38 & 32 & $A_4$ & $\ZZ/2$ & 38 & Yes & \begin{tabular}[c]{@{}r@{}}47 \\ 48\end{tabular} & \begin{tabular}[c]{@{}c@{}}$11$ \\ $2$\end{tabular} & \begin{tabular}[c]{@{}r@{}}1 \\ 1\end{tabular} & \begin{tabular}[c]{@{}r@{}}47 \\ 48\end{tabular}\\
\hline
41 & 32 & $D_4(a_1){+}A_2$ & $\ZZ/2$ & 48 & Yes & \begin{tabular}[c]{@{}r@{}}51\end{tabular} & \begin{tabular}[c]{@{}c@{}}$11$\end{tabular} & \begin{tabular}[c]{@{}r@{}}1\end{tabular} & \begin{tabular}[c]{@{}r@{}}76\end{tabular}\\
\hline
43 & 33 & $2A_3$ & $1$ & 43 & Yes & \begin{tabular}[c]{@{}r@{}}56\end{tabular} & \begin{tabular}[c]{@{}c@{}}$1$\end{tabular} & \begin{tabular}[c]{@{}r@{}}1\end{tabular} & \begin{tabular}[c]{@{}r@{}}74\end{tabular}\\
\hline
56 & 36 & $A_4{+}A_3$ & $1$ & 30 & Yes & \begin{tabular}[c]{@{}r@{}}69\end{tabular} & \begin{tabular}[c]{@{}c@{}}$1$\end{tabular} & \begin{tabular}[c]{@{}r@{}}1\end{tabular} & \begin{tabular}[c]{@{}r@{}}69\end{tabular}\\
\hline
57 & 36 & $D_4{+}A_2$ & $\ZZ/2$ & 20 & Yes & \begin{tabular}[c]{@{}r@{}}68\end{tabular} & \begin{tabular}[c]{@{}c@{}}$11$\end{tabular} & \begin{tabular}[c]{@{}r@{}}1\end{tabular} & \begin{tabular}[c]{@{}r@{}}68\end{tabular}\\
\hline
63 & 40 & $E_8(a_7)$ & $S_5$ & 1 & Yes & \begin{tabular}[c]{@{}r@{}}82 \\ 83 \\ 84 \\ 85 \\ 86 \\ 87 \\ 88\end{tabular} & \begin{tabular}[c]{@{}c@{}}$5$ \\ $41$ \\ $32$ \\ $311$ \\ $221$ \\ $2111$ \\ $11111$\end{tabular} & \begin{tabular}[c]{@{}r@{}}1 \\ 4 \\ 5 \\ 6 \\ 5 \\ 4 \\ 1\end{tabular} & \begin{tabular}[c]{@{}r@{}}1 \\ 2 \\ 4 \\ 10 \\ 13 \\ 36 \\ 88\end{tabular}\\
\end{longtable}
\endgroup
The row $R_{38}$ gives $\mathbf M_{38,69}=4$, although $\rho_{69}$
has rank one. Here $C_{38}$ has dimension $32$ and saturation $A_4$,
whereas $P_{69}$ is supported on $\overline{C_{56}}$. The coefficients of $R_{63}$ are the irreducible
degrees of $S_5$, and Fourier symmetry gives the corresponding
coefficients of $R_1$.

\clearpage
\subsection{Failure of extremality}
\label{subsec:extremality}

For a fixed enhanced parameter $a$, define the finite ordered set
$$
 \mathscr A(a)=\{i:C_i\text{ is Arthur and }\mathbf M_{ia}>0\},
 \qquad i\leq j\ \Longleftrightarrow\ C_i\subseteq\overline{C_j}.
$$
Hazeltine--Liu--Lo--Zhang \cite[Conjecture~1.5]{HLLZ} predict, for
endoscopic Arthur packets of a connected reductive group, that all
parameters whose packets contain a fixed Arthur-type representation
have the same infinitesimal parameter and possess a unique least and a
unique greatest element in the closure order. The microlocal analogue
would require a least and greatest element of every nonempty
$\mathscr A(a)$. The following examples show that this need not hold.

\begin{example}[No least element in type $F_4$]
Take $a=19$, the $(211)$ local system on $C_{12}$.
The packet table gives
$$
 \mathscr A(19)=\{6,9,10,11,12\},\qquad
 (\mathbf M_{6,19},\mathbf M_{9,19},\mathbf M_{10,19},
  \mathbf M_{11,19},\mathbf M_{12,19})=(1,1,1,2,3).
$$
The covering relations in Figure~\ref{fig:hasse-f4} are
$$
\begin{gathered}
1\to2,\quad2\to3,\quad3\to4,5,\quad4\to8,\quad
5\to6,7,\\
6\to9,\quad7\to8,\quad8\to9,10,\quad
9\to11,\quad10\to11,\quad11\to12.
\end{gathered}
$$
The induced poset on $\mathscr A(19)$ has chains
$6<9<11<12$ and $10<11<12$, with $10$ incomparable to both
$6$ and $9$. Hence $6$ and $10$ are its two minimal elements.
Their saturations are $\widetilde A_2$ and $B_2$ and their dimensions
are $8$ and $10$. In particular no least element exists.
\end{example}

\begin{example}[No greatest element in type $E_8$]
Take $a=35$, the trivial local system on $C_{31}$. This orbit has
dimension $30$ and saturation $A_3+2A_1$. The packet table gives
$$
 \mathscr A(35)=\{4,10,15,19,30\};
 \qquad\mathbf M_{i,35}=1\quad(i\in\mathscr A(35)).
$$
Its induced closure order consists of $4<10<15<19$ and
$4<10<15<30$, with $19$ and $30$ incomparable.
Here $C_{19}$ has dimension $26$ and saturation $A_3$, whereas
$C_{30}$ has dimension $29$ and saturation $2A_2+2A_1$.
The full lower ideal of $C_{30}$ in the closure order is
$$
 \{1,2,3,4,5,6,7,8,9,10,11,12,13,14,15,16,17,18,
       20,22,23,25,30\}.
$$
Since $19$ is absent, $C_{19}\not\subseteq\overline{C_{30}}$.
The reverse inclusion is excluded by dimension. Hence $19$ and $30$
are the two maximal elements of $\mathscr A(35)$, so no greatest
element exists.
\end{example}

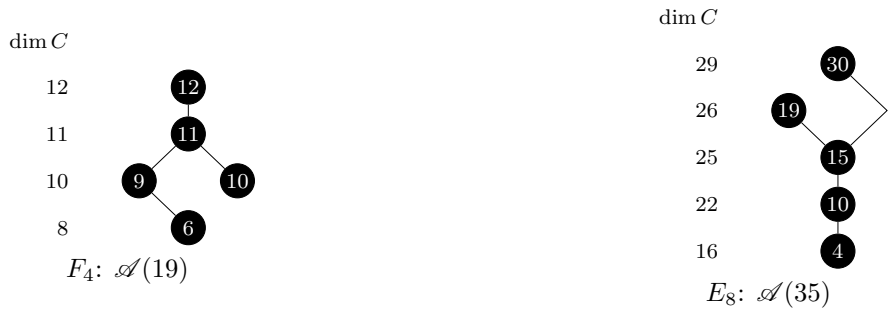
\begin{figure}[htbp]
\centering
\begin{minipage}{.47\textwidth}\centering
\begin{tikzpicture}[x=1cm,y=1cm, every node/.style={font=\scriptsize},
 art/.style={circle,draw,fill=black,text=white,inner sep=1pt,minimum size=4.5mm},
 ordinary/.style={circle,draw,fill=white,inner sep=1pt,minimum size=4.5mm}]
\node[art] (n6) at (0.000,0.000) {$ 6 $};
\node[art] (n9) at (-0.650,0.620) {$ 9 $};
\node[art] (n10) at (0.650,0.620) {$ 10 $};
\node[art] (n11) at (0.000,1.240) {$ 11 $};
\node[art] (n12) at (0.000,1.860) {$ 12 $};
\draw[line width=.35pt] (n6) -- (n9);
\draw[line width=.35pt] (n9) -- (n11);
\draw[line width=.35pt] (n10) -- (n11);
\draw[line width=.35pt] (n11) -- (n12);
\node[anchor=east] at (-1.450,0.000) {$ 8 $};
\node[anchor=east] at (-1.450,0.620) {$ 10 $};
\node[anchor=east] at (-1.450,1.240) {$ 11 $};
\node[anchor=east] at (-1.450,1.860) {$ 12 $};
\node[anchor=east] at (-1.450,2.480) {$\dim C$};
\end{tikzpicture}\par\small $F_4$: $\mathscr A(19)$
\end{minipage}\hfill
\begin{minipage}{.47\textwidth}\centering
\begin{tikzpicture}[x=1cm,y=1cm, every node/.style={font=\scriptsize},
 art/.style={circle,draw,fill=black,text=white,inner sep=1pt,minimum size=4.5mm},
 ordinary/.style={circle,draw,fill=white,inner sep=1pt,minimum size=4.5mm}]
\node[art] (n4) at (0.000,0.000) {$ 4 $};
\node[art] (n10) at (0.000,0.620) {$ 10 $};
\node[art] (n15) at (0.000,1.240) {$ 15 $};
\node[art] (n19) at (-0.650,1.860) {$ 19 $};
\coordinate (nd3l3) at (0.650,1.860);
\node[art] (n30) at (0.000,2.480) {$ 30 $};
\draw[line width=.35pt] (n4) -- (n10);
\draw[line width=.35pt] (n10) -- (n15);
\draw[line width=.35pt] (n15) -- (n19);
\draw[line width=.35pt] (n15) -- (nd3l3) -- (n30);
\node[anchor=east] at (-1.450,0.000) {$ 16 $};
\node[anchor=east] at (-1.450,0.620) {$ 22 $};
\node[anchor=east] at (-1.450,1.240) {$ 25 $};
\node[anchor=east] at (-1.450,1.860) {$ 26 $};
\node[anchor=east] at (-1.450,2.480) {$ 29 $};
\node[anchor=east] at (-1.450,3.100) {$\dim C$};
\end{tikzpicture}\par\small $E_8$: $\mathscr A(35)$
\end{minipage}
\caption{Induced closure orders for the extremality counterexamples.
The left diagram has two minimal vertices, $6$ and $10$; the right has
two maximal vertices, $19$ and $30$. Edges are covers in the indicated
induced posets, and may represent several steps in the full closure poset.
The left columns give orbit dimensions.}\label{fig:hasse-extremality}
\end{figure}
The two induced Hasse diagrams are displayed in Figure~\ref{fig:hasse-extremality}.

To obtain counterexamples to the endoscopic formulation, one needs
agreement of endoscopic and microlocal packet membership for these
exceptional groups, with compatible parametrizations and relevance
conditions. The comparison for odd orthogonal groups does not establish
this agreement.

\subsection{Closure ordering and weak packets}

The closure-ordering conjecture \cite[Conjecture~1.2]{HLLZ} asserts that
$\pi_a\in\Pi_i$ implies $C_i\leq C_a$. For microlocal packets this
follows from the support condition for characteristic cycles. In
particular, both counterexamples satisfy this condition: the failure
concerns the existence of a least or greatest Arthur orbit among those
whose packets contain $\pi_a$.

The weak-packet decomposition conjecture, in the formulation of
\cite[Conjecture~2.10]{BL}, asserts that a weak packet is a union of
Arthur packets. Equation~\eqref{eq:e6-weak-example} establishes the
corresponding microlocal equality for the $D_4(a_1)$ grading in $E_6$,
with the Fourier--Aubert interpretation specified above. An endoscopic
interpretation requires the corresponding packet comparison.

The supplementary tables are available in the online Arthur Atlas
\cite{ArthurAtlas}.
The supplementary data contain the complete orbit and enhancement
dictionaries, closure relations, Arthur rows of $K$, and microlocal
multiplicities for all four examples. They also include the $E_6$
special-piece data and exact checks of the packet tables, closure
support, top-conormal ranks, Fourier symmetry, and the two extremality
examples. Source identifiers and checksums are supplied with the data.
These checks retain the geometric input hypotheses of the computations;
they do not determine microlocal monodromy.

\appendix

\section{Classification tables and orbit-label dictionaries}
\label{app:tables}
Table~\ref{tab:exceptional-by-grading} gives the numbers of graded orbits,
Arthur orbits, and $p$-fixed Arthur orbits for each diagonal nilpotent
orbit $\mathcal O^\vee$. Within each type, read down
the left block and then down the right block; $0$ denotes the zero orbit.

\begingroup
\small
\setlength{\tabcolsep}{5pt}
\renewcommand{\arraystretch}{1.05}
\begin{longtable}{@{}lrrr@{\hspace{1.8em}}lrrr@{}}
\caption{Exceptional classification by diagonal nilpotent orbit.}
\label{tab:exceptional-by-grading}\\
\hline
$\mathcal O^\vee$ & \shortstack{Graded\\orbits} & \shortstack{Arthur\\orbits} & \shortstack{$p$-fixed\\Arthur orbits} &
$\mathcal O^\vee$ & \shortstack{Graded\\orbits} & \shortstack{Arthur\\orbits} & \shortstack{$p$-fixed\\Arthur orbits}\\\hline
\endfirsthead
\multicolumn{8}{c}{\tablename~\thetable\ (continued)}\\
\hline
$\mathcal O^\vee$ & \shortstack{Graded\\orbits} & \shortstack{Arthur\\orbits} & \shortstack{$p$-fixed\\Arthur orbits} &
$\mathcal O^\vee$ & \shortstack{Graded\\orbits} & \shortstack{Arthur\\orbits} & \shortstack{$p$-fixed\\Arthur orbits}\\\hline
\endhead
\hline
\multicolumn{8}{r}{Continued on next page}\\
\endfoot
\hline
\endlastfoot
\multicolumn{8}{c}{$G_2$}\\*\hline
$0$ & 1 & 1 & 1 & $G_2(a_1)$ & 4 & 4 & 0\\
$A_1$ & 2 & 2 & 0 & $G_2$ & 4 & 2 & 0\\
$\tilde A_1$ & 2 & 2 & 0 &  & & & \\
\textit{Total} & 13 & 11 & 1 & & & & \\\hline
\multicolumn{8}{c}{$F_4$}\\*\hline
$0$ & 1 & 1 & 1 & $B_2$ & 6 & 2 & 0\\
$A_1$ & 2 & 2 & 0 & $C_3(a_1)$ & 8 & 6 & 0\\
$\tilde A_1$ & 3 & 3 & 1 & $F_4(a_3)$ & 12 & 11 & 1\\
$A_1{+}\tilde A_1$ & 4 & 4 & 0 & $C_3$ & 8 & 2 & 0\\
$\tilde A_2$ & 3 & 2 & 0 & $B_3$ & 8 & 2 & 0\\
$A_2$ & 5 & 5 & 1 & $F_4(a_2)$ & 14 & 6 & 0\\
$A_2{+}\tilde A_1$ & 6 & 4 & 0 & $F_4(a_1)$ & 16 & 2 & 0\\
$\tilde A_2{+}A_1$ & 6 & 4 & 0 & $F_4$ & 16 & 2 & 0\\
\textit{Total} & 118 & 58 & 4 & & & & \\\hline
\multicolumn{8}{c}{$E_6$}\\*\hline
$0$ & 1 & 1 & 1 & $D_4(a_1)$ & 18 & 11 & 1\\
$A_1$ & 2 & 2 & 0 & $D_4$ & 8 & 2 & 0\\
$2A_1$ & 3 & 3 & 1 & $A_4$ & 23 & 3 & 1\\
$3A_1$ & 4 & 4 & 0 & $A_4{+}A_1$ & 40 & 4 & 0\\
$A_2$ & 5 & 5 & 1 & $A_5$ & 32 & 2 & 0\\
$A_2{+}A_1$ & 10 & 6 & 0 & $D_5(a_1)$ & 30 & 4 & 0\\
$2A_2$ & 10 & 3 & 1 & $E_6(a_3)$ & 49 & 6 & 0\\
$A_2{+}2A_1$ & 9 & 7 & 1 & $D_5$ & 36 & 2 & 0\\
$A_3$ & 6 & 2 & 0 & $E_6(a_1)$ & 60 & 2 & 0\\
$2A_2{+}A_1$ & 20 & 6 & 0 & $E_6$ & 64 & 2 & 0\\
$A_3{+}A_1$ & 18 & 6 & 0 &  & & & \\
\textit{Total} & 448 & 83 & 7 & & & & \\\hline
\multicolumn{8}{c}{$E_7$}\\*\hline
$0$ & 1 & 1 & 1 & $D_4{+}A_1$ & 24 & 4 & 0\\
$A_1$ & 2 & 2 & 0 & $A_4{+}A_1$ & 50 & 6 & 0\\
$2A_1$ & 3 & 3 & 1 & $A_4{+}A_2$ & 24 & 4 & 0\\
$3A_1''$ & 4 & 4 & 0 & $D_5(a_1)$ & 36 & 6 & 0\\
$3A_1'$ & 4 & 4 & 0 & $A_5'$ & 20 & 2 & 0\\
$A_2$ & 5 & 5 & 1 & $D_5(a_1){+}A_1$ & 35 & 8 & 0\\
$4A_1$ & 8 & 8 & 0 & $A_5{+}A_1$ & 40 & 4 & 0\\
$A_2{+}A_1$ & 12 & 12 & 0 & $D_6(a_2)$ & 55 & 8 & 0\\
$A_2{+}2A_1$ & 11 & 11 & 1 & $E_6(a_3)$ & 35 & 6 & 0\\
$2A_2$ & 9 & 5 & 1 & $E_7(a_5)$ & 64 & 20 & 0\\
$A_3$ & 6 & 2 & 0 & $D_5$ & 26 & 2 & 0\\
$A_2{+}3A_1$ & 10 & 10 & 0 & $A_6$ & 32 & 2 & 0\\
$(A_3{+}A_1)''$ & 13 & 6 & 0 & $D_5{+}A_1$ & 48 & 4 & 0\\
$2A_2{+}A_1$ & 20 & 10 & 0 & $D_6(a_1)$ & 60 & 4 & 0\\
$(A_3{+}A_1)'$ & 13 & 6 & 0 & $E_7(a_4)$ & 84 & 12 & 0\\
$D_4(a_1)$ & 15 & 11 & 1 & $E_6(a_1)$ & 67 & 2 & 0\\
$A_3{+}2A_1$ & 26 & 12 & 0 & $D_6$ & 64 & 2 & 0\\
$D_4$ & 8 & 2 & 0 & $E_6$ & 40 & 2 & 0\\
$D_4(a_1){+}A_1$ & 32 & 18 & 0 & $E_7(a_3)$ & 98 & 4 & 0\\
$A_3{+}A_2$ & 30 & 20 & 0 & $E_7(a_2)$ & 110 & 4 & 0\\
$A_3{+}A_2{+}A_1$ & 25 & 20 & 0 & $E_7(a_1)$ & 120 & 2 & 0\\
$A_4$ & 24 & 3 & 1 & $E_7$ & 128 & 2 & 0\\
$A_5''$ & 20 & 2 & 0 &  & & & \\
\textit{Total} & 1561 & 287 & 7 & & & & \\\hline
\multicolumn{8}{c}{$E_8$}\\*\hline
$0$ & 1 & 1 & 1 & $A_5{+}A_1$ & 44 & 4 & 0\\
$A_1$ & 2 & 2 & 0 & $D_5(a_1){+}A_2$ & 70 & 8 & 0\\
$2A_1$ & 3 & 3 & 1 & $E_6(a_3){+}A_1$ & 70 & 12 & 0\\
$3A_1$ & 4 & 4 & 0 & $D_6(a_2)$ & 58 & 6 & 0\\
$A_2$ & 5 & 5 & 1 & $E_7(a_5)$ & 64 & 20 & 0\\
$4A_1$ & 5 & 5 & 1 & $E_8(a_7)$ & 63 & 34 & 2\\
$A_2{+}A_1$ & 10 & 10 & 0 & $D_5{+}A_1$ & 52 & 4 & 0\\
$A_2{+}2A_1$ & 10 & 10 & 2 & $D_6(a_1)$ & 64 & 6 & 0\\
$A_3$ & 6 & 2 & 0 & $A_6$ & 48 & 3 & 1\\
$A_2{+}3A_1$ & 10 & 10 & 0 & $E_7(a_4)$ & 84 & 12 & 0\\
$2A_2$ & 10 & 10 & 2 & $A_6{+}A_1$ & 64 & 4 & 0\\
$2A_2{+}A_1$ & 18 & 10 & 0 & $E_6(a_1)$ & 65 & 2 & 0\\
$A_3{+}A_1$ & 13 & 6 & 0 & $D_5{+}A_2$ & 86 & 18 & 0\\
$D_4(a_1)$ & 15 & 11 & 1 & $E_6$ & 40 & 2 & 0\\
$D_4$ & 8 & 2 & 0 & $D_6$ & 48 & 2 & 0\\
$2A_2{+}2A_1$ & 16 & 11 & 1 & $D_7(a_2)$ & 130 & 4 & 0\\
$A_3{+}2A_1$ & 20 & 10 & 0 & $E_6(a_1){+}A_1$ & 134 & 4 & 0\\
$D_4(a_1){+}A_1$ & 30 & 22 & 0 & $A_7$ & 70 & 2 & 0\\
$A_3{+}A_2$ & 31 & 21 & 1 & $E_7(a_3)$ & 98 & 4 & 0\\
$A_4$ & 21 & 3 & 1 & $E_8(b_6)$ & 105 & 8 & 0\\
$A_3{+}A_2{+}A_1$ & 25 & 20 & 0 & $E_6{+}A_1$ & 80 & 4 & 0\\
$D_4{+}A_1$ & 16 & 4 & 0 & $D_7(a_1)$ & 118 & 6 & 0\\
$D_4(a_1){+}A_2$ & 23 & 21 & 1 & $E_7(a_2)$ & 110 & 4 & 0\\
$A_4{+}A_1$ & 48 & 6 & 0 & $E_8(a_6)$ & 110 & 4 & 0\\
$2A_3$ & 20 & 3 & 1 & $E_8(b_5)$ & 128 & 8 & 0\\
$D_5(a_1)$ & 32 & 6 & 0 & $D_7$ & 96 & 2 & 0\\
$A_4{+}2A_1$ & 55 & 15 & 1 & $E_7(a_1)$ & 120 & 2 & 0\\
$A_4{+}A_2$ & 32 & 12 & 2 & $E_8(a_5)$ & 145 & 6 & 0\\
$A_5$ & 20 & 2 & 0 & $E_8(b_4)$ & 168 & 4 & 0\\
$A_4{+}A_2{+}A_1$ & 48 & 8 & 0 & $E_7$ & 128 & 2 & 0\\
$D_5(a_1){+}A_1$ & 35 & 8 & 0 & $E_8(a_4)$ & 178 & 2 & 0\\
$E_6(a_3)$ & 35 & 6 & 0 & $E_8(a_3)$ & 196 & 4 & 0\\
$D_4{+}A_2$ & 38 & 10 & 0 & $E_8(a_2)$ & 220 & 2 & 0\\
$D_5$ & 26 & 2 & 0 & $E_8(a_1)$ & 240 & 2 & 0\\
$A_4{+}A_3$ & 60 & 4 & 0 & $E_8$ & 256 & 2 & 0\\
\textit{Total} & 4501 & 498 & 20 & & & & \\\hline
\end{longtable}
\endgroup

Table~\ref{tab:exceptional-relative-lifts} lists the orbits in
$\mathscr S_{\mathrm{rel}}(\mathcal O^\vee)$ that are not in the image
of the saturation map for the grading defined by $\mathcal O^\vee$.
Only diagonal orbits with a nonempty list are shown. For every other
ambient orbit in the relative special piece, the intersection with
$V$ is a single Arthur orbit, as in
Theorem~\ref{thm:weak-injectivity}(\ref{item:weak-exceptional-lifts}).

\begingroup
\small
\begin{longtable}{@{}llp{0.57\textwidth}@{}}
\caption{Orbits in exceptional relative special pieces with no graded lift.}
\label{tab:exceptional-relative-lifts}\\
\hline
Type & Diagonal orbit $\mathcal O^\vee$ & Orbits outside the saturation image\\\hline
\endfirsthead
\hline
Type & Diagonal orbit $\mathcal O^\vee$ & Orbits outside the saturation image\\\hline
\endhead
$G_2$ & $\tilde A_1$ & $A_1$\\
\hline
$F_4$ & $\tilde A_2{+}A_1$ & $A_2{+}\tilde A_1$\\
 & $B_2$ & $A_2{+}\tilde A_1$\\
 & $C_3(a_1)$ & $A_2{+}\tilde A_1$, $\tilde A_2{+}A_1$\\
\hline
$E_6$ & $A_3{+}A_1$ & $2A_2{+}A_1$\\
\hline
$E_7$ & $(A_3{+}A_1)'$ & $2A_2{+}A_1$\\
 & $D_6(a_2)$ & $A_5{+}A_1$\\
\hline
$E_8$ & $A_3{+}A_1$ & $2A_2{+}A_1$\\
 & $A_3{+}2A_1$ & $2A_2{+}2A_1$\\
 & $A_5{+}A_1$ & $A_4{+}A_3$\\
 & $D_5(a_1){+}A_2$ & $A_4{+}A_3$\\
 & $E_6(a_3){+}A_1$ & $D_5(a_1){+}A_2$, $A_4{+}A_3$\\
 & $D_6(a_2)$ & $A_5{+}A_1$, $D_5(a_1){+}A_2$, $A_4{+}A_3$\\
 & $E_7(a_5)$ & $E_6(a_3){+}A_1$, $D_5(a_1){+}A_2$, $A_4{+}A_3$\\
 & $E_7(a_2)$ & $E_6{+}A_1$\\
\hline
\end{longtable}
\endgroup

\section{Component groups and relevance}
\label{app:components}

For the exceptional lifts listed below, put
$Q=Z_{\Gd}(e_1^\vee,h_1^\vee,f_1^\vee)$,
$\mathfrak q=\Lie Q$, and
$\mathfrak q(2;h_2^\vee)=\{v\in\mathfrak q:[h_2^\vee,v]=2v\}$.

\begin{table}[htbp]
\centering
\small
$$
\begin{array}{c|l|ccc}
 \mathcal O^\vee&\Gd e_1^\vee&\dim\mathfrak q&
              \dim\mathfrak z(\mathfrak q)&\dim\mathfrak q(2;h_2^\vee)\\\hline
 F_4(a_3)&C_3(a_1),\ \widetilde A_2+A_1,\ A_2+\widetilde A_1&3&0&1\\
 F_4(a_3)&B_2&6&0&2\\
 E_8(a_7)&E_7(a_5),\ E_6(a_3)+A_1&3&0&1\\
 E_8(a_7)&D_6(a_2)&6&0&2
\end{array}
$$
\caption{Centralizer dimensions verifying the principal second-triple
criterion in exceptional types.}
\label{tab:even-component-check}
\end{table}

The dimensions in Table~\ref{tab:even-component-check} imply that
$\mathfrak q$ is $\mathfrak{sl}_2$ or
$\mathfrak{sl}_2\oplus\mathfrak{sl}_2$. Since $h_2^\vee$ is a
Dynkin element, the last column shows that it is nonzero on every
simple factor, and hence principal. The weighted diagrams and
centralizer dimensions are recorded in the supplementary data.

\section{Computational verification and supplementary data}
\label{app:verification}

The tables accompanying this paper are available in the online Arthur
Atlas \cite{ArthurAtlas}.

\medskip

\paragraph{Scope and conventions.}
The computations cover all $157$ neutral gradings in types
$G_2$, $F_4$, $E_6$, $E_7$, and $E_8$.
Table~\ref{tab:exceptional-by-grading} records $6641$ graded orbits
and $937$ Arthur orbits, of which $39$ are fixed by $p$.
The dual groups are simply connected, and all equivariant local
systems are included. 

\medskip

\paragraph{Geometric and IC inputs.}
Orbit enumeration and the Arthur test are described in
Section~\ref{subsec:exceptional-classification}: the regrading reduces
enumeration to the Vinberg--de Graaf setting, and
Lemma~\ref{lem:exceptional-polynomial-criterion} reduces the residual
triple problem to polynomial equations over $\mathbb Q$.
The calculations use \textsc{GAP}/\textsc{SLA} and \textsc{Singular}
\cite{GAP,SLA,Singular}.
For reconstruction, the inputs are the closure order, local-system
ranks, the permutations $p,f$, and the IC polynomials of
\cite{CiubotaruOkadaKL}, including the non-principal series.
The passage to ordinary stalk Euler characteristics uses
\eqref{eq:kl-to-stalk}, with its rank factors and perverse signs.

\medskip

\paragraph{Reconstruction checks.}
We solve \eqref{eq:reconstruction-system} by exact rational row
reduction. The homogeneous solution space vanishes on the Arthur
rows, as required by Theorem~\ref{thm:arthur-row-reconstruction}.
For these rows we verify
$$
 \mathbf M_{i\bullet}=K_{i\bullet}X,\qquad
 \mathbf M_{ia}=\mathbf M_{p(i),f(a)}\in\mathbb Z_{\geq0},\qquad
 \mathbf M_{ia}=\dim\rho_a\quad(C_a=C_i),
$$
together with the support condition $\mathbf M_{ia}=0$ unless
$C_i\leq C_a$. We also check $p^2=f^2=\mathrm{id}$ and preservation
of the Arthur locus by $p$. Undetermined non-Arthur rows are retained
as such. The output consists of ordinary microlocal multiplicities
and packet membership; equivariant microlocal monodromy is additional
data.

\medskip

\paragraph{Special pieces and character checks.}
CHEVIE \cite{CHEVIE} supplies ambient duality and special-piece data.
Comparing these with the graded saturations gives the exceptional
checks in Section~\ref{sec:weak}. Packet containment is tested using
\eqref{eq:weak-fourier-set}; the extremality checks use the complete
induced posets $\mathscr A(a)$ of Section~\ref{subsec:extremality}.
For weak sphericity we use the Springer action in which the regular
orbit corresponds to the trivial Weyl-group character. Character restriction computes
the multiplicity of the trivial character of $\widehat W_K$, including
its length-zero elements. These checks concern principal-series
enhancements relevant to the split adjoint form.

\

\noindent{\bf Acknowledgements.} 
After the results of this paper were finalized, Cunningham, Hazeltine, Liu, and Lo informed the authors that they had also obtained a related characterization of Arthur packets for classical groups in their forthcoming work \cite{CHLL}.
The authors thank Lucas Mason-Brown for valuable conversations about Arthur parameters and pairs. The first author thanks Brin Center, University of Maryland, and the organizers of the workshop ``Recent progress on the unitary dual problem'', September 2026, for the invitation and the opportunity to present the results of this paper.

The authors used ChatGPT Codex to produce the software that implements the graded orbit enumeration, the computation of the IC polynomials \cite{CiubotaruOkadaKL}, the linear reconstruction of the Arthur packets, and the Weyl group branching of standard characters, and to assist with drafting and polishing parts of the exposition. The authors reviewed and revised the resulting text and take responsibility for the final version.

\bibliographystyle{amsplain}
\bibliography{references}

\end{document}